\documentclass{amsbook}

\usepackage[dvipsnames]{xcolor}
\usepackage{etoolbox}

\makeatletter
\numberwithin{section}{chapter}
\def\@secnumfont{\mdseries}
\def\section{\@startsection{section}{1}%
  \z@{.7\linespacing\@plus\linespacing}{.5\linespacing}%
  {\normalfont\scshape\centering}}
\def\subsection{\@startsection{subsection}{2}%
  \z@{.5\linespacing\@plus.7\linespacing}{-.5em}%
  {\normalfont\bfseries}}

\patchcmd{\@thm}{\let\thm@indent\indent}{\let\thm@indent\noindent}{}{}
\patchcmd{\@thm}{\thm@headfont{\scshape}}{\thm@headfont{\bfseries}}{}{}

\makeatother

\usepackage{tikz}
\usetikzlibrary{calc,shapes}

\usepackage{amsfonts}
\usepackage{amssymb}
\usepackage{enumitem}
\usepackage{moreenum}
\usepackage{chngcntr}

\newtheorem{theorem}{Theorem} % 1st argument is your name for it
\newtheorem{lemma}{Lemma}     % 2nd argument is what is printed

\theoremstyle{definition}
\newtheorem{definition}[]{Definition}
\newtheorem{remark}[]{Remark}
\newtheorem{property}{Property}

\counterwithin*{equation}{chapter}
\counterwithin*{lemma}{chapter}
\counterwithin*{theorem}{chapter}
\counterwithin*{definition}{chapter}

\newcommand{\ol}[1]{\mkern 1.5mu\overline{\mkern-1.5mu#1\mkern-1.5mu}\mkern 1.5mu}

\newcommand{\len}[1]{{[#1}^\partial}
\newcommand{\gre}{\eqcirc}

\newcommand{\fA}{\mathfrak{A}}
\newcommand{\fB}{\mathfrak{B}}
\newcommand{\fC}{\mathfrak{C}}
\newcommand{\fD}{\mathfrak{D}}
\newcommand{\fE}{\mathfrak{E}}
\newcommand{\fF}{\mathfrak{F}}
\newcommand{\fG}{\mathfrak{G}}
\newcommand{\fH}{\mathfrak{H}}

\newcommand{\fT}{\mathfrak{T}}

\newcommand{\cJ}{\mathcal{J}}

\newcommand{\tm}[1]{\tikz[overlay,remember picture] \node (#1) {};}
\newcommand{\dl}[3]{%
  \begin{tikzpicture}[overlay,remember picture]
    \draw[-,shorten >=7.7pt,shorten <=7.7pt,out=-60,in=-120,distance=#1cm,Black!50] (#2.north) to (#3.north);
  \end{tikzpicture}
}
\newcommand{\dla}[4]{%
  \begin{tikzpicture}[overlay,remember picture]
    \draw[-,shorten >=7.7pt,shorten <=7.7pt,out=-60,in=-120,distance=#1cm,Black!50] (#3.north) to node[Black]{\tiny $\mathbf{#2}$} (#4.north);
  \end{tikzpicture}
}

\title[The Identity Problem For Finitely Presented Groups and Semigroups]% end with percent
 {\Large} % This is the full 
 
\author[G. S. Makanin]{\vspace{-2cm} \normalsize \textsc{USSR Academy of Sciences} \\
Steklov Mathematical Institute \\
\vspace{2cm}
\large \textsc{G. S. Makanin}\\
    \vspace{0.5cm}
    \LARGE\textsc{On The Identity Problem For Finitely Presented Groups and Semigroups} \\
    \vspace{2cm}
 
\normalsize Dissertation for the degree of Candidate of Physical and Mathematical Sciences. \\
\vspace{7cm}
Supervisors \\
Corresponding Member of the USSR Academy of Sciences \textsc{A. A. Markov} \\
Doctor of Physical and Mathematical Sciences \textsc{S. I. Adian}

\vfill
\vspace{5cm}

Moscow (1966)
    }

\begin{document}
\maketitle

\tableofcontents

%\vspace{4cm}

\thispagestyle{empty}
\begin{center}
\textsc{Notes on the Translation}
\end{center}

\vspace{1cm}

The many fascinating ideas in this thesis, worked out in full detail and with close attention paid to ensuring a high standard of rigour, have made it a delight to translate. Nearly sixty years later, G. S. Makanin's insights into the structure of special monoids and small cancellation theory remain exceptional, and that without any of the modern machinery to approach these topics. 

In transcribing and translating this thesis from the original Russian, I have attempted to capture its original minimal style, in which the mathematics itself is in focus. Perhaps the largest change made to this original style is that proof environments, which were originally omitted, have been added, to ensure proper cut-off between statements, their proofs, and remarks following them. Any Cyrillic letters used for indexing, especially the convoluted nested cases in Chapter~2, have been replaced by Latin or Greek counterparts. While the page numbering is different from the original (with the original thesis being 98 pages), the numbering of all sections, lemmas, equations, etc. remains the same. Particular attention was given to accurately typesetting the numerous involved, and originally hand-drawn, small cancellation diagrams in Chapter~2.

As for linguistic choices, a direct translation of some of the definitions has not always been possible or feasible, as, for instance, some use terminology that would be best translated by an already overloaded term in English. In these cases, I have instead focused on capturing the mathematical sense of the word to be translated. Finally, some typos have been fixed -- but there were very few to be found indeed, as could be expected of a student working under the watchful gaze of S. I. Adian. Any typos still present in this translation are entirely my own. 

I wish to thank Prof L. D. Beklemishev of the Steklov Mathematical Institute for his support related to the project.

\vspace{2cm}

\noindent\textbf{C. F. Nyberg-Brodda}\\
29 January 2021

\clearpage

\thispagestyle{empty}

\begin{center}
\fbox{\includegraphics[scale=1.5]{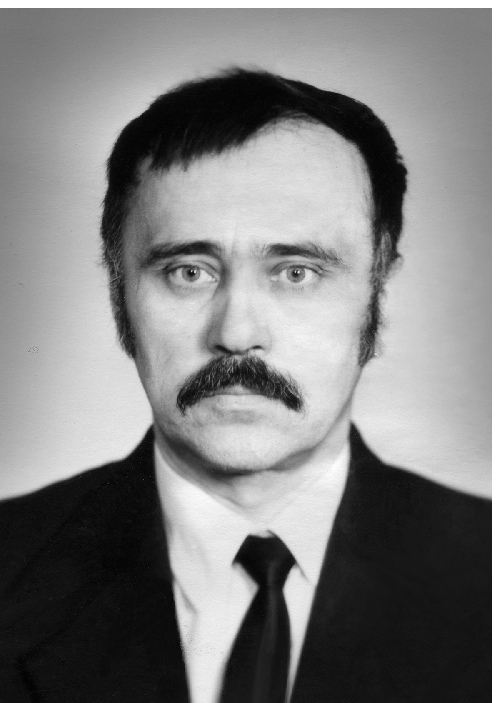}}
\end{center}
\vspace{0.5cm}
\begin{center}
Gennadij Semenovich Makanin (1938--2017).
\end{center}

\clearpage

\thispagestyle{empty}
\vspace*{5cm}
\begin{center}
\textit{This page intentionally left blank.}
\end{center}

\clearpage

%

%\section*{Notes on the Translation}
%\clearpage

%\thispagestyle{empty}
%{\large
%\tableofcontents}
%\clearpage

\thispagestyle{empty}
\chapter*{Introduction}

Let $\Pi$ be a semigroup given by a set of generators
\begin{equation}
a_1, a_2, \dots, a_n
\end{equation}
and defining relations
\begin{equation}
A_i = B_i \quad (i =1, 2, \dots, m)
\end{equation}
The system of generators $(1)$ is called the alphabet of the semigroup $\Pi$. The set of words over the alphabet $(1)$ are called the words of the semigroup $\Pi$. The empty word will be denoted $1$, and the length of the word $X$ will be denoted $\len{X}$. Graphical equality of two words $X$ and $Y$ will be indicated by $X \gre Y$.

The rules of inference by which, based on the relations $(2)$, we determine the relation of equality for words of $\Pi$, generate a form of calculus (see \cite{5}, pp. 205, 206). Most of the auxiliary concepts and statements of the present work is associated with this calculus. By considering this calculus as a means of defining a semigroup, we will not use the term calculus in our work, but rather phrase all concepts in terms of semigroups instead.

In this work we will reason about certain \textit{normal decision problems}, and prove the reducibility of some problems to others. In order to accurately formulate the notion of a normal decision problem, it is necessary to fix a way of writing the input data over the alphabet, and then pose the problem of finding a normal algorithm over this alphabet. For simplicity, we will not do this, and use the term \textit{normal decision problem} without further specification, with the implicit understanding that it is based on this properly refined formulation. The principle of normalisation (see \cite{5} II.\S5) makes it possible to assert that such a refinement of a decision problem is uniquely defined up to equivalence. In the same way, the proofs of the reducibility of one normal decision problem to another, which would require reasoning about normal algorithms, will be replaced for simplicity by a description of the processes by which the solution of a given problem, corresponding to a normal decision problem, is reduced to solving single cases of another normal decision problem. We are not presented with any particular fundamental difficulty in carrying out these proofs in full detail, but it is very cumbersome. 

We will call an \textit{elementary transformation} of the semigroup $\Pi$ the transition of a word of the form $XA_iY$ to the word $XB_iY$, or vice versa, where $X$ and $Y$ are arbitrary words of the semigroup $\Pi$, and $A_i = B_i$ is one of the defining relations in $(2)$. An elementary transformation is denoted by one of the two forms 
\[
XA_iY \to XB_iY, \qquad XB_iY \to XA_iY.
\]
We also say that the tautological transition $X \to X$ is an elementary transformation.

The relations $(2)$ define the equality of words in the semigroup $\Pi$, in a way connected to the elementary transformations of the semigroup $\Pi$ in the following way: two words $X$ and $Y$ of the semigroup $\Pi$ are equal in $\Pi$ if and only if there exists some sequence of elementary transformations of the semigroup $\Pi$ 
\begin{equation}
X \gre X_0 \to X_1 \to \cdots \to X_i \to X_{i+1} \to \cdots \to X_k \gre Y,
\end{equation}
transforming the word $X$ into the word $Y$. 

We will call $G$ the \textit{group} defined by the generators $(1)$ and the defining relations $(2)$ if it is the semigroup given by the alphabet 
\[
a_1, a_2, \dots, a_n, a_1^{-1}, a_2^{-1}, \cdots, a_n^{-1}
\]
and the system of defining relations
\setcounter{equation}{1}
\begin{equation}
A_i = B_i \quad (i = 1, 2, \dots, m)
\end{equation}
\setcounter{equation}{3}
\begin{equation}
\left.\begin{array}{lr}
        a_ja_j^{-1} = 1\\
        a_j^{-1}a_j = 1
        \end{array}\right\} \qquad (j = 1, 2, \dots, n)
\end{equation}
The alphabet $(1)$ will be called the \textit{positive alphabet} of the group $G$. The letter $a_j^{-1}$ will be called the \textit{inverse} of the letter $a_j$. The defining relations $(4)$ will be called \textit{trivial} relations. These relations are uniquely determined by the alphabet of the group, and are therefore generally note written out when defining a group. 

A group will be called a \textit{free group}, if all defining relations are trivial. The equality of two words $X$ and $Y$ in the free group will be denoted by $X \equiv Y$. 

The \textit{identity problem} for words in a semigroup $\Pi$ (group $G$) is the following problem: provide an algorithm, by which one can determine for any two words $X$ and $Y$ whether or not $X$ and $Y$ are equal in the semigroup $\Pi$ (in the group $G$).

A semigroup given over the alphabet 
\setcounter{equation}{0}
\begin{equation}
a_1, a_2, \dots, a_n
\end{equation}
\setcounter{equation}{4}
and the finitely many defining relations 
\begin{equation}
A_i = 1 \quad (i = 1, 2, \dots, k)
\end{equation}
such that 
\begin{equation}
\len{A_i} \leq \ell \quad (i = 1, 2, \dots, k),
\end{equation}
will be called a \textit{$(k, \ell)$-semigroup.} If the $(k, \ell)$-semigroup $\Pi_0$ is such that we have $\len{A_i} = \ell \:\: (i = 1, 2, \dots, k)$, then we say that $\Pi_0$ is a \textit{homogeneous} $(k, \ell)$-semigroup.

We will say, that a word $X$ in some semigroup $\Pi$ is \textit{left (right) invertible}, if there exists some word $Y$ in the semigroup such that $YX = 1$ (resp. $XY = 1$) in the semigroup $\Pi$. A word $X$ is \textit{two-sided invertible} in the semigroup $\Pi$ if $X$ is both left and right invertible.

If in the $(k, \ell)$-semigroup $\Gamma$ every word is two-side invertible, then $\Gamma$ will be called a \textit{$(k, \ell)$-group}. 

We will say that the word $X$ in some semigroup $\Pi$ is \textit{left (right) divisible} by the word $Y$, if there exists some word $Z$ such that in $\Pi$ we have the equality $X = YZ$ (resp. $X = ZY$). 

The \textit{left (right) divisibility problem} for a given semigroup $\Pi$ is defined as follows: provide an algorithm, which decides for any two words $X$ and $Y$ whether or not $X$ is left (right) divisible by $Y$ in the semigroup $\Pi$. 

S. I. Adian, in the third chapter of \cite{1}, has proved: if the natural numbers $k$ and $\ell$ are such that there exists an algorithm for solving the identity problem in every $(k, \ell)$-group, then there is an algorithm, which solves the identity problem in any homogeneous $(k, \ell)$-semigroup.

In Sections~1.1, 1.2, and 1.3 of the present work, it is proved that under the same assumptions as those made by S. I. Adian, there is an algorithm which solves the identity problem in any $(k, \ell)$-semigroup. 

We will call the \textit{maximal subgroup} of a semigroup the subgroup generated by the collection of two-sided invertible words of the semigroup. 

S. I. Adian, in the third chapter of \cite{1}, has proved: if the natural numbers $k$ and $\ell$ are such that there exists an algorithm for solving the identity problem in every $(k, \ell)$-group, then there is an algorithm which for any homogeneous $(k, \ell)$-semigroup $\Pi$ produces a $(k, \ell)$-group, isomorphic to the maximal subgroup of $\Pi$. 

In Section~1.4 of the present work, it is proved that  under the same assumptions as those made by S. I. Adian, there is an algorithm which for any $(k, \ell)$-semigroup $\Pi$ produces a $(k, \ell)$-group, isomorphic to the maximal subgroup of $\Pi$. 

Two systems of defining relations $\Omega_1$ and $\Omega_2$ over some alphabet $A$ will be called \textit{equivalent}, if for every choice of two words over $A$, the words are equal in the semigroup with the system of defining relations $\Omega_1$ if and only if the words are equal in the semigroup with system of defining relations $\Omega_2$. 

If two semigroups $\Pi_1$ and $\Pi_2$ are given by equivalent systems of defining relations, then the identity problem in $\Pi_1$, either of the divisibility problems in $\Pi_1$, and the problem of computing a group isomorphic to the maximal subgroup of $\Pi$, can each be reduced to the respective problem for the semigroup $\Pi_2$. 

Let the group $G$ be defined by the generators 
\setcounter{equation}{0}
\begin{equation}
a_1, a_2, \dots, a_n
\end{equation}
and a finite system of defining relations 
\setcounter{equation}{6}
\begin{equation}
R_i = 1 \quad (i = 1, 2, \dots, m)
\end{equation}
The left-hand sides of the defining relations $R_i = 1$  ($i = 1, 2, \dots, m$) are words written over the alphabet $(a_1^{\pm 1}, \dots, a_n^{\pm 1})$, which we will call the \textit{defining words} of the group $G$. We denote by $M$ the set of all defining words. Deleting a word of the form $a_j a_j^{-1}$ or $a_j^{-1} a_j$ from some word $W$ will be called \textit{reducing} $W$. A word, which cannot be reduced, will be called \textit{reduced}. If 
\[
Z \gre a_{i_1}^{\varepsilon_1} a_{i_2}^{\varepsilon_2} \cdots a_{i_{v-1}}^{\varepsilon_{v-1}} a_{i_v}^{\varepsilon_v}
\]
where $\varepsilon_i = \pm 1$ ($i = 1, 2, \dots, v$), then we will call the word
\[
(a_{i_{v}}^{\varepsilon_{v}})^{-1} (a_{i_{v-1}}^{\varepsilon_{v-1}})^{-1} \cdots (a_{i_{2}}^{\varepsilon_{2}})^{-1} (a_{i_{1}}^{\varepsilon_{1}})^{-1}
\]
the \textit{inverse} of $Z$, and denote it by $Z^{-1}$ or $\ol{Z}$. 

Without loss of generality, we may assume that every defining word $R_i$ of the group $G$ is reduced, and that the set $M$ is closed under the operations of taking inverses of and cyclic permutations of the words $R_i$ (this implies that each $R_i$ is cyclically reduced).  Such a set of defining words will be called \textit{symmetrized}.  

V. A. Tartakovskij \cite{6, 7}, and subsequently A. V. Gladkij \cite{2, 3}, M. D. Greendlinger \cite{4, 11} and others, have studied finitely presented groups given by a symmetrized set of defining relations. In some classes of these groups, it is possible to distinguish large parts of the defining words in every non-empty, reduced word equal to the identity of the group. In many cases, this permits a positive solution to the identity and conjugacy problems, and to describe some of the properties of these groups.

A group $G$, with generating set $(1)$ and defining relations $(7)$, is said to \textit{belong to the class }$K_\alpha$, where $0 < \alpha < 1$, if the set of $M$ of the group $G$ satisfies the following property: if $R_i$ and $R_j$ are arbitrary, not mutually inverse words in the set $M$, then reducing the word $R_iR_j$ removes less than $\alpha$ times the number of letters in $R_i$.

In the cited works, and also in the work by Britton \cite{9} and Sheik \cite{12}, the identity problem is solved for various generalisations of the class of groups $K_{1/6}$. 

In Chapter~2 of the present work, using a generalisation of Dehn's algorithm \cite{10}, the identity problem is solved for the class of groups $K_{2/11}$. 

The definitions and notation found in the introduction (which are borrowed from the works of S. I. Adian \cite{1}, M. D. Greendlinger \cite{4}, A. A. Markov \cite{5}, and G. S. Tsejtin \cite{8}) will be considered as known, and will not be re-introduced.

The results of Chapter~1 of the present work are published in the note \cite{13}. The results of the present work were presented at a seminar on algorithmic problems in algebra at Moscow State University. 

I wish to express my gratitude to A. A. Markov and S. I. Adian for their attention and help related to the work. 

\clearpage

\chapter{Algorithmic problems for $(k, \ell)$-semigroups}
\section{Construction of elementary words}

For any list of words $X_1, X_2, \dots, X_s$ on some alphabet $A$, we associate a \textit{reduced} list of words $\mathcal{T}_1(X_1, X_2, \dots, X_s)$, which is obtained from the list $X_1, X_2, \dots, X_s$, by repeatedly applying the following two operations until no longer possible:
\begin{enumerate}[label=(\greek*)]
\item if $X_j \gre 1$, then remove $X_j$ from the list;
\item if $X_k \gre X_\ell$ and $k < \ell$, then remove $X_\ell$ from the list.
\end{enumerate}
Starting with $X_1, X_2, \dots, X_s$, the above will result in a \textit{unique} list $X_{i_1}, X_{i_2}, \dots, X_{i_t}$, which we will denote by $\mathcal{T}_1(X_1, X_2, \dots, X_s)$. 

Order in some way the set of all pairs of natural numbers and, with $Y_1, Y_2, \dots, Y_t$ some list of words, say that an ordered pair of words $(Y_{j_1}, Y_{j_2})$ \textit{precedes} the ordered pair $(Y_{k_1}, Y_{k_2})$ if in our chosen ordering the pair of numbers $(j_1, j_2)$ precedes the pair $(k_1, k_2)$. 

\begin{definition}
We say that a pair of words $X$ and $Y$ on an alphabet $A$ \textit{overlap} if $X \gre MN$ and $Y \gre NL$, where $N$ and $ML$ are non-empty words. 
\end{definition}

Suppose we are given a list of words
\begin{equation}
Y_1, Y_2, \dots, Y_t, Z_1, Z_2, \dots, Z_s.
\end{equation}
If in the list of words $Y_1, Y_2, \dots, Y_t$ we delete the words $Y_{v_1}, Y_{v_2}, \cdots, Y_{v_m}$, we will denote the remaining words of $(1)$ by 
\[
Y_1, Y_2, \dots, Y_t, \overline{Y}_{v_1}, \overline{Y}_{v_2}, \dots, \overline{Y}_{v_m}, Z_1, Z_2, \dots, Z_s.
\]

\

\noindent\underline{The division algorithm $\Delta$.}

\

Consider an arbitrary reduced list of words $Y_1, Y_2, \dots, Y_t$ on some alphabet $A$. The algorithm $\Delta$, when applied to such a list, will act in one of two ways:

\begin{enumerate}[label =(\roman*)]
\item If no pair of words from the list $Y_1, Y_2, \dots, Y_t$ overlap, then 
\[
\Delta(Y_1, Y_2, \dots, Y_t) \rightleftharpoons Y_1, Y_2, \dots, Y_t.
\]
\item Suppose $Y_{v_1} \gre MN, Y_{v_2} \gre NL$, where $N$ and $ML$ are non-empty, where $(Y_{v_1}, Y_{v_2})$ is a pair of words from $Y_1, Y_2, \dots, Y_t$, such that this pair precedes all other overlapping pairs of words from this list; then 
\[
\Delta(Y_1, Y_2, \dots, Y_t) \rightleftharpoons \Delta(\mathcal{T}_1(Y_1, Y_2, \dots, Y_t, \overline{Y}_{v_1}, \overline{Y}_{v_2}, M, N, L)). 
\]
\end{enumerate}

\begin{definition}
Consider a list of words $Y_1, Y_2, \dots, Y_t$ such that $\len{Y_1} = \ell_i$, where $\ell_i$ is some positive integer for $i = 1, 2, \dots, t$. The \textit{parameter} of the list $Y_1, Y_2, \dots, Y_t$ is defined as the number $\sum_{i=1}^t (\ell_i -1)$, and is denoted by $\omega(Y_1, Y_2, \dots, Y_t)$. 
\end{definition}

\begin{lemma}
If the pair of words $(Y_{v_1}, Y_{v_2})$ overlap, where the words are taken from some list $Y_1, Y_2, \dots, Y_t$, and $Y_{v_1} \gre MN$, $Y_{v_2} \gre NL$, where $N$ and $ML$ are non-empty, then 
\[
\omega(Y_1, Y_2, \dots, Y_t) > \omega(\mathcal{T}_1(Y_1, Y_2, \dots, Y_t, \overline{Y}_{v_1}, \overline{Y}_{v_2}, M, N, L)).
\]
\end{lemma}

The lemma follows immediately from the facts that $\len{Y_{v_1}} = \len{M} + \len{N}$ and $\len{Y_{v_2}} = \len{N} + \len{L}$, and that any arbitrary list of words $X_1, X_2, \dots, X_s$ has parameter no less than that of $\mathcal{T}_1(X_1, X_2, \dots, X_s)$. 

Suppose the algorithm $\Delta$ is applied to a reduced list $Y_1, Y_2, \dots, Y_t$. We introduce the notion of a \textit{step} of the algorithm. If the algorithm $\Delta$ immediately applies rule (i), then we say that the algorithm $\Delta$ produces a non-overlapping list of words in $0$ steps, or that it finishes in $0$ steps. If the algorithm $\Delta$ first applies rule (ii) $k$ times, and then applies rule (i), then we say that $\Delta$ produces a non-overlapping list of words in $k$ steps, or that it finishes in $k$ steps. 

\begin{lemma}
For every reduced list of words $Y_1, Y_2, \dots, Y_t$, there exists some non-negative integer $k$, such that the algorithm $\Delta$ when applied to  $Y_1, Y_2, \dots, Y_t$ finishes in $k$ steps. 
\end{lemma}
\begin{proof}
By Lemma~1, at every step of the algorithm $\Delta$, we process a list with a given parameter $\omega_i$ into a list with parameter $\omega_{i+1}$, where $\omega_i > \omega_{i+1}$. As the parameter $\omega$ is defined as a positive integer, it follows that the algorithm $\Delta$ will eventually produce a list with no overlaps, at which point it will apply rule (i). 
\end{proof}

\begin{definition}
If the algorithm $\Delta$, when applied to a list of words $Y_1, Y_2, \dots, Y_t$, finishes with the list of non-overlapping words $V_1, V_2, \dots, V_k$, then we write
\begin{equation}
\Delta[Y_1, Y_2, \dots, Y_t] = V_1, V_2, \dots, V_k.
\end{equation}
\end{definition}

\begin{lemma}
If $\Delta[Y_1, Y_2, \dots, Y_t] = V_1, V_2, \dots, V_k$, then 
\begin{enumerate}
\item $V_1, V_2, \dots, V_k$ is a reduced list of non-overlapping words.
\item Every $Y_i \: (i = 1, 2, \dots, t)$ is graphically a product of some $V_{i_1}, V_{i_2}, \dots, V_{i_{s_i}}$, where $1 \leq i_1, i_2, \dots, i_{s_i} \leq t$. 
\item For every $V_j \: (j = 1, 2, \dots, k)$ there exists some $Y_i \gre V_{i_1} V_{i_2} \cdots V_{i_{s_i}}$ such that $V_j$ is graphically equal to at least one of $V_{i_1}, V_{i_2}, \dots, V_{i_{s_i}}$. 
\end{enumerate}
\end{lemma}

All parts of the lemma are easily proved by induction on the number of steps taken by the algorithm $\Delta$ when applied to the list $Y_1, Y_2, \dots, Y_t$. 

\begin{lemma}
If the words $XY$ and $YZ$ are invertible in some semigroup $\Pi$, then $X, Y$, and $Z$ are invertible in $\Pi$. 
\end{lemma}
\begin{proof}
As $XY$ and $YZ$ are invertible, we can find some words $U_1, U_2, V_1$, and $V_2$,  such that in $\Pi$ we have the equalities
\begin{equation}
U_1 XY = 1, \: XY U_2 = 1, \: YZ V_1 = 1, \: V_2 YZ = 1.
\end{equation}
From the first and third equalities in the above $(3)$ it follows that $Y$ is invertible in $\Pi$. Observe that $ZV_1Y = U_1 XYZV_1Y = U_1XY = 1$, and $YU_1X = YU_1XYZV_1 = YZV_1 = 1$ in the semigroup $\Pi$, and hence it follows that $X$ and $Z$ are also invertible in $\Pi$. 
\end{proof}

\begin{lemma}
If $\Delta[Y_1, Y_2, \dots, Y_t] = V_1, V_2, \dots, V_k$ and every $Y_i \: (i = 1, 2, \dots, t)$ is invertible in the semigroup $\Pi$, then all the $V_j \: (j = 1, 2, \dots, k)$ are invertible in $\Pi$. 
\end{lemma}

This lemma is easily proved by induction on the number of steps taken by the algorithm $\Delta$ when applied to the list $Y_1, Y_2, \dots, Y_t$, and using Lemma~4. 

Let $\Pi$ be a $(k, \ell)$-semigroup with generating set 
\begin{equation}
a_1, a_2, \dots, a_n
\end{equation}
and defining relations
\begin{equation}
A_i = 1 \quad (i = 1, 2, \dots, k),
\end{equation}
satisfying 
\begin{equation}
\len{A_i} \leq \ell \quad (i = 1, 2, \dots, k).
\end{equation}

\begin{definition}
A finite list of words $B_1, B_2, \dots, B_p$ over the alphabet $(4)$ is called a list of \textit{$B$-words of the $(k, \ell)$-semigroup $\Pi$} if
\begin{enumerate}
\item $B_1, B_2, \dots, B_p$ is a reduced list of non-overlapping words.
\item Every left-hand side $A_i \: (i = 1, 2, \dots, k)$ of the $(k, \ell)$-semigroup $\Pi$ is graphically equal to a product of some $B_{i_1}, B_{i_2}, \dots, B_{i_{s_i}} \: (1 \leq i_1, i_2, \dots, i_{s_i} \leq p)$.
\item For every $B_j \: (j = 1, 2, \dots, p)$ there exists some $A_i \gre B_{i_1} B_{i_2} \cdots B_{i_{s_i}}$ such that $B_j$ is graphically equal to one of the $B_{i_1}, B_{i_2}, \dots, B_{i_{s_i}}$. 
\end{enumerate}
\end{definition}

\begin{lemma}
There exists an algorithm which for every $(k, \ell)$-semigroup $\Pi$ constructs a list of $B$-words for this semigroup.
\end{lemma}
\begin{proof}
Let $A_1, A_2, \dots, A_k$ be the left-hand sides of the defining relations of some $(k, \ell)$-semigroup $\Pi$. Apply the algorithm $\Delta$ to the list $\mathcal{T}_1(A_1, A_2, \dots, A_k)$. Suppose that 
\[
\Delta[\mathcal{T}_1(A_1, A_2, \dots, A_k)] = B_1, B_2, \dots, B_p.
\]
Then by Lemma~3, the reduced list $B_1, B_2, \dots, B_p$ is a list of $B$-words of the $(k, \ell)$-semigroup $\Pi$.
\end{proof}

\begin{lemma}
If $A \gre B_{j_1} B_{j_2} \cdots B_{j_p} \gre B'_{m_1} B'_{m_2} \cdots B'_{m_s}$, where $B_{j_i}$ and $B'_{m_i}$ are $B$-words of the $(k, \ell)$-semigroup $\Pi$, then $p = s$, and $B_{j_i} \gre B'_{m_i}$ for all $i$.
\end{lemma}
\begin{proof}
Let $A$ be the shortest word such that some of the conclusions of Lemma~7 do not hold.  Neither of the words $B_{j_1}$ and $B'_{m_1}$ can be a proper prefix of the other, as $B$-words do not overlap. Hence $B_{j_1} \gre B'_{m_1}$, and $B_{j_2} \cdots B_{j_p} \gre B'_{m_2} \cdots B'_{m_s}$. Call this word $A'$. As $\len{A'} < \len{A}$, and as $A$ is the shortest word such that the conclusions of Lemma~7 do not hold, we have that $p-1 = s-1$, and $B_{j_i} \gre B'_{m_i}$ for $i = 1, 2, \dots, p$. But then the conclusions of Lemma~7 hold for $A$. This is a contradiction. Hence, Lemma~7 is true for all words $A$. 
\end{proof}

\

\noindent\underline{Specifying the $(k, \ell)$-semigroup $\Gamma$ from a $(k, \ell)$-semigroup $\Pi$ and its $B$-words.}

\

Establish a one-to-one connection between the $B$-words $B_1, B_2, \dots, B_p$ of the $(k, \ell)$-semigroup $\Pi$ and a set of symbols $\beta_1, \beta_2, \dots, \beta_p$. As $B_i \not\gre B_j$ for $i \neq j$, this can be done.

By Lemma~7 every $A_i$ can be uniquely written as $B_{i_1} B_{i_2} \cdots B_{i_{s_i}}$, where $(i = 1, 2, \dots, k)$ and $(1 \leq i_1, i_2, \dots, i_{s_i} \leq p)$. 

Let $\Gamma$ be the $(k, \ell)$-semigroup with generating set
\begin{equation}
\beta_1, \beta_2, \dots, \beta_p
\end{equation}
and defining relations 
\begin{equation}
M_i = 1 \quad (i = 1, 2, \dots, k)
\end{equation}
where $M_i \gre \beta_{i_1} \beta_{i_2} \cdots \beta_{i_{s_i}} \: (i = 1, 2, \dots, k), \: (1 \leq i_1, i_2, \dots, i_{s_i} \leq p)$. That is, $M_i$ is obtained from $A_i \gre B_{i_1} B_{i_2} \cdots B_{i_{s_i}}$ by replacing the $B$-words by their corresponding letters $\beta_{i_1}, \beta_{i_2}, \dots, \beta_{i_{s_i}}$. 

As $\len{M_i} \leq \len{A_i} \: (i = 1, 2, \dots, k)$, it follows that $\Gamma$ really is a $(k, \ell)$-semigroup. From Lemma~7, it follows that the $(k, \ell)$-semigroup $\Gamma$ is uniquely defined from the $(k, \ell)$-semigroup $\Pi$ and its $B$-words. 

\begin{lemma}
$\Delta[\mathcal{T}_1(M_1, M_2, \dots, M_k)] = \beta_1, \beta_2, \dots, \beta_p$.
\end{lemma} 
The lemma follows easily by induction on the number of steps taken by the algorithm $\Delta$. For this, one must note that any word appearing in the process of applying the algorithm $\Delta$ to $\mathcal{T}_1(M_1, M_2, \dots, M_k)$ is either empty, or else comes from some $B$-word; and then note that $\Delta[\mathcal{T}_1(A_1, A_2, \dots, A_k)] = B_1, B_2, \dots, B_p$. 

\begin{definition}
A finite list of words $C_1, C_2, \dots, C_s$ on the alphabet $(4)$ is called a list of \textit{$C$-words of the $(k, \ell)$-semigroup $\Pi$}, if $C_1, C_2, \dots, C_s$ is a list of $B$-words of $\Pi$ and 
\begin{enumerate}
\setcounter{enumi}{3}
\item The $(k, \ell)$-semigroup $\Gamma$, defined from the $(k, \ell)$-semigroup $\Pi$ and the $B$-words $C_1, C_2, \dots, C_s$, is a $(k, \ell)$-group.
\end{enumerate}
\end{definition}

\begin{lemma}
There exists an algorithm which for any $(k, \ell)$-semigroup $\Pi$ produces a list of $C$-words for this semigroup. 
\end{lemma}
\begin{proof}
Suppose that $\Delta[\mathcal{T}_1(A_1, A_2, \dots, A_k)] = B_1, B_2, \dots, B_p$. By Lemma~6, we have that $B_1, B_2, \dots, B_p$ are $B$-words of $\Pi$. By Lemmas~5 and 8 all generators $\beta_1, \beta_2, \dots, \beta_p$ of the $(k, \ell)$-semigroup $\Gamma$, corresponding to the $(k, \ell)$-semigroup $\Pi$ and the $B$-words $B_1, B_2, \dots, B_p$, are invertible. Thus the $(k, \ell)$-semigroup $\Gamma$ is a $(k, \ell)$-group. 
\end{proof}

\begin{remark}
If the $B$-words are $C$-words, then we may speak of specifying from the $(k, \ell)$-semigroup $\Pi$ and its $C$-words the $(k, \ell)$-group $\Gamma$. 
\end{remark}

\begin{remark}
When there is no risk of misunderstanding, we will sometimes simply speak of a semigroup rather than a $(k, \ell)$-semigroup.
\end{remark}

\

\noindent\underline{Generating $c_i$-words of the $(k, \ell)$-semigroup $\Pi$}

\

The generation of $c_i$-words from a given $(k, \ell)$-semigroup $\Pi$, some $C$-words $C_1, C_2, \dots, C_s$ of this semigroup, and a specified $C_i$ from the list $C_1, C_2, \dots, C_s$, will be achieved by constructing a set of words, which we will call \textit{$c_i$-words} of the $(k, \ell)$-semigroup $\Pi$. For this, we will assume that there exists an algorithm for solving the identity problem in the $(k, \ell)$-group $\Gamma$ associated to the semigroup $\Pi$ and its $C$-words $C_1, C_2, \dots, C_s$. 

We will define $c_i$-words inductively for a given $(k, \ell)$-semigroup $\Pi$, given a specified $C$-word $C_i$. 

\begin{enumerate}
\item Suppose that $C_i \gre h_1 C_{i_1} C_{i_2} \cdots C_{i_p} h_2$, where $p \geq 0$ and $h_1, h_2$ are non-empty; and suppose that $U \gre h_1 C_{j_1} C_{j_2} \cdots C_{j_t} h_2$, with $t \geq 0$; and that $\len{U} \leq \len{C_i}$ and $\beta_{i_1} \beta_{i_2} \cdots \beta_{i_p} = \beta_{j_1} \beta_{j_2} \cdots \beta_{j_t}$ in the group $\Gamma$. Then $U$ is a $c_i$-word of the $(k, \ell)$-semigroup $\Pi$. 
\item Suppose that $c_i^{(k)}$ is some $c_i$-word of the semigroup $\Pi$; suppose further that $c_{i}^{(k)} \gre g_1 C_{k_1} C_{k_2} \cdots C_{k_v} g_2$, where $v \geq 0$ and $g_1, g_2$ are non-empty; and suppose that $V \gre g_1 C_{m_1} C_{m_2} \cdots C_{m_w} g_2$ with $w \geq 0$; and that $\len{V} \leq \len{{c_i^{(k)}}}$ and $\beta_{k_1} \beta_{k_2} \cdots \beta_{k_v} = \beta_{m_1} \beta_{m_2} \cdots \beta_{m_w}$ in the group $\Gamma$. Then $V$ is a $c_i$-word of the $(k, \ell)$-semigroup $\Pi$. 
\end{enumerate}

Note that in the above construction that we may replace $C_{e_1} C_{e_2} \cdots C_{e_m}$ by the empty word, if $\beta_{e_1} \beta_{e_2} \cdots \beta_{e_m} = 1$ in the group $\Gamma$; and that $C_i$ is a $c_i$-word, as we may replace the empty word inside $C_i$ by the empty word. 

\begin{lemma}
There are only finitely many $c_i$-words.
\end{lemma}
\begin{proof}
For any $C$-word $C_i$ the number of graphically distinct $c_i$-words cannot exceed $n^{\len{C_i}}$, where $n$ is the number of generators of $\Pi$, which proves the lemma.
\end{proof}

Thus, if there is an algorithm for solving the identity problem in the $(k, \ell)$-group $\Gamma$, we may $c_1$-words, which we will denote by $c_1^{(1)}, c_1^{(2)}, \dots, c_1^{(v_1)}$; and all $c_2$-words, which we denote by $c_2^{(1)}, c_2^{(2)}, \dots, c_2^{(v_2)}$; $\dots$ ; and all $c_s$-words, which we will denote by $c_s^{(1)}, c_s^{(2)}, \dots, c_s^{(v_s)}$. Sometimes we will enumerate all these words in a single list $c_1, c_2, \dots, c_t$ and refer to them as the $c$-words of the $(k, \ell)$-semigroup $\Pi$. Furthermore, the set of all $c_i$-words $c_i^{(1)}, c_i^{(2)}, \dots, c_i^{(v_i)}$ will sometimes be denoted by $\{ c_{i}^{(j)} \}$ without specifying the set $1, 2, \dots, v_i$ through which $j$ runs. 

Hence we have proved the following theorem.

\begin{theorem}
If the natural numbers $k$ and $\ell$ are such that there exists an algorithm for solving the identity problem in every $(k, \ell)$-group, then there exists an algorithm which for any $(k, \ell)$-semigroup produces all $C$-words $C_1, C_2, \dots, C_s$ of this semigroup, and all sets $\{ c_1^{(j)} \}, \{ c_2^{(j)} \}, \dots, \{ c_s^{(j)} \}$ of this semigroup.
\end{theorem}

For the following arguments we will focus on those sets $\{ c_1^{(j)} \}, \{ c_2^{(j)} \}, \dots, \{ c_s^{(j)} \}$, of some $(k, \ell)$-semigroup, which satisfy the following  three properties.

\begin{enumerate}[label=\Roman*.]
\item From the construction of $c_i$-words, we will always have $\len{C_i} \geq \len{{c_i^{(j)}}}$, for all $i = 1, 2, \dots, s$ and $j = 1, 2, \dots, v_i$; for our purposes, we will need the stronger $\len{C_i} = \len{{c_i^{(j)}}}$ for all $i = 1, 2, \dots, s$ and $j = 1, 2, \dots, v_i$. 
\item Generally speaking, we may have that some word $f$ is an element of the set $\{ c_{i}^{(j)} \}$ as well as of the set $\{ c_{k}^{(j)} \}$, where $i \neq k$, but for our purposes this is not permitted.
\item Some pairs of words $c_{p}^{(t)}, c_q^{(\rho)}$ may overlap. For our purposes, we need that no pair of $c$-words overlap. 
\end{enumerate}

If a set of words $\{ c_1^{(j)} \}, \{ c_2^{(j)} \}, \dots, \{ c_s^{(j)} \}$ of some $(k, \ell)$-semigroup $\Pi$ does not satisfy property I, II, or III, then we will say that the set lacks I, II, or III, respectively. 

\begin{remark}
We will make use of the definitions of \textit{affected} and \textit{unaffected} letters introduced by P. S. Novikov. These definitions can be found in the work by S. I. Adian \cite[pp. 7-8]{1}. Furthermore, given a list of words $X_1, X_2, \dots, X_m$ such that each $X_{j+1}$ is obtained from $X_j \: (j = 1, 2, \dots, m-1)$ by a single step of the above generating operation. Then this sequence can be represented as a sequence of elementary transformations $X_1 \to X_2 \to \dots \to X_m$ in the semigroup, in which all substitutions allowed by the generating operation are given by the defining relations. In this way, we naturally extend the notion of affected and unaffected letters to a list of words $X_1, X_2, \dots, X_m$.
\end{remark}

\begin{definition}
A $(k, \ell)$-\textit{tuple} is any pair, the first element of which is some $(k, \ell)$-semigroup $T$, and the second of which is a reduced list of $C$-words of this semigroup $T$. Such a tuple will often be denoted by $\{ T ; C_1, C_2, \dots, C_s \}$.
\end{definition}

If the conditions of Theorem~1 are satisfied, then for any given $(k, \ell)$-tuple $\{ T ; C_1, C_2, \dots, C_s \}$ we can uniquely obtain the sets $\{ c_1^{(j)} \}, \{ c_2^{(j)} \}, \dots, \{ c_s^{(j)} \}$, where $s$ depends only on the tuple $\{ T ; C_1, C_2, \dots, C_s \}$, and unambiguously (from the unambiguity of the generating operation) obtain, from the $(k, \ell)$-semigroup $T$ and its $C$-words $C_1, C_2, \dots, C_s$, a $(k, \ell)$-group $\Gamma$. Thus in all the following reasoning we will assume that the conditions of Theorem~1 are satisfied, and so we can speak of the \textit{sets $\{ c_1^{(j)} \}, \{ c_2^{(j)} \}, \dots, \{ c_s^{(j)} \}$ of the tuple $\{ T ; C_1, C_2, \dots, C_s \}$}, and of the $(k, \ell)$-group $\Gamma$ of the tuple $\{ T ; C_1, C_2, \dots, C_s \}$. 

\begin{definition}
A $(k, \ell)$-tuple $V$ will be said to be \textit{distinguished}, if the sets $\{ c_1^{(j)} \}$, $\{ c_2^{(j)} \}, \dots, \{ c_s^{(j)} \}$ of this tuple does not lack a single one of the three properties I, II, or III. 
\end{definition}

\begin{lemma}
Let $\Gamma$ be a $(k, \ell)$-group with generating alphabet 
\begin{equation}
\gamma_1, \gamma_2, \dots, \gamma_s
\end{equation}
and defining relations
\begin{equation}
N_i = 1 \quad (i = 1, 2, \dots, k)
\end{equation}
and suppose that in this $(k, \ell)$-group there is some equality 
\begin{equation}
\gamma_{p_1} \gamma_{p_2} \cdots \gamma_{p_v} = \gamma_{q_1} \gamma_{q_2} \cdots \gamma_{q_w}.
\end{equation}
Let $F_1, F_2, \dots F_s$ be a list of words over some alphabet $A$. Let $F$ denote the semigroup with generating alphabet $A$ and defining relations obtained from $(10)$ by replacing each letter $\gamma_i$ with $F_i$ for $i = 1, 2, \dots, s$. Then in the semigroup $F$ we have the equality $F_{p_1} F_{p_2} \cdots F_{p_v} = F_{q_1} F_{q_2} \cdots F_{q_w}$.
\end{lemma}

The lemma follows easily from the definition of equality in the semigroup. 

\begin{definition}
Let $\{ \Pi ; C_1, C_2, \dots, C_s \}$ be a $(k, \ell)$-tuple, and let $\Gamma$ be the $(k, \ell)$-group associated to this tuple. Let $f$ be some word in the set $\{ c_m^{(j)} \}$ associated to this tuple. If we (like in Lemma~11) replace the generators $\gamma_1, \gamma_2, \dots, \gamma_s$ of $\Gamma$ by $C_1, C_2, \dots C_s$, then we obviously obtain the semigroup $\Pi$. If we replace the generators 
\[
\gamma_1, \gamma_2, \dots, \gamma_{m-1}, \gamma_m, \gamma_{m+1}, \dots, \gamma_s
\]
by the list 
\[
C_1, C_2, \dots, C_{m-1}, f, C_{m+1}, \dots, C_s
\]
then we obtain a $(k, \ell)$-semigroup which, as $\len{C_m} \geq \len{f}$, we will call  \textit{derived} from the $(k, \ell)$-semigroup of the tuple $\{ \Pi ; C_1, C_2, \dots, C_s \}$.
\end{definition}

\begin{lemma}
Let  $\{ \Pi ; C_1, C_2, \dots, C_s \}$ be a $(k, \ell)$-tuple and let $T$ be some semigroup derived from the $(k, \ell)$-semigroup of this tuple. Then the system of relations of the semigroup $\Pi$ is equivalent to that of $T$. 
\end{lemma}
\begin{proof}
Let $\Pi$ be a $(k, \ell)$-semigroup with generating set $(4)$ and defining relations $A_i = 1 \: (i = 1, 2, \dots, k)$, and let $T$ be a $(k, \ell)$-semigroup with generating set $(4)$ and defining relations $G_i = 1 \: (i = 1, 2, \dots, k)$. Because $f$ is a word from $\{ c_m^{(j)} \}$, we have that $f = C_m$ in the semigroup $\Pi$. 

\begin{enumerate}[label=(\alph*), mode = unboxed]
\item We will show that $G_i = 1 \: (i = 1, 2, \dots, k)$ in the semigroup $\Pi$. In fact, $G_i$ is obtained from $A_i$ by replacing $C_m$ by $f$, and as $f = C_m$ and $A_i = 1$ in $\Pi$, we have shown (a).
\item We will show that $A_i = 1 \: (i = 1, 2, \dots, k)$ in the semigroup $T$. We divide this into two separate cases.
\begin{enumerate}[label = (\roman*), mode = unboxed]
\item Consider the $p$-th defining relation of the group $\Gamma$. Suppose it does not include $\gamma_m$. Then $A_p \gre G_p$, and hence $A_p = 1$ in $T$. 
\item Consider the $j$-th defining relation of the group $\Gamma$. Suppose that $\gamma_m$ appears $t$ times on the left side of this relation. Then 
\begin{align*}
A_j &\gre h_1 C_m h_2 C_m \cdots C_m h_{t+1} \\
G_j &\gre h_1 \:f \:\:\: h_2 \: f \:\:\:\cdots \: f \:\:\:h_{t+1},
\end{align*}
where each $h_v \: (v = 1, 2, \dots, t+1)$ is a product of some $C$-words. Let us repeat the process by which we obtained the word $f$ from $C_m$. The word $C_m$ could be written as 
\[
C_m \gre m_1 C_{i(1,1)} C_{i(1,2)} \cdots C_{i(1,k_1)} n_1 \gre p_0,
\]
where $m_1, n_1$ are non-empty words and $1 \leq i(1, 1), \dots, i(1, k_1) \leq s$. The word
\[
p_1 \gre m_1 C_{j(1,1)} C_{j(1,2)} \cdots C_{j(1,t_1)} n_1
\]
is obtained from $C_m$ by a single application of the generating operation. 

The graphical equalities in $(12)$ record the $\mu$ steps of the generating operation by which we obtain $f$ from $C_m$. 

\begin{align}
\qquad\qquad p_0 &\gre m_1 C_{i(1,1)} C_{i(1,2)} \cdots C_{i(1,k_1)} n_1 \gre C_m \nonumber \\
\qquad\qquad p_1 &\gre m_1 C_{j(1,1)} C_{j(1,2)} \cdots C_{j(1,t_1)} n_1 \gre m_2 C_{i(2,1)} C_{i(2,2)} \cdots C_{i(2,k_2)} n_2 \nonumber \\ 
\qquad\qquad p_2 &\gre m_2 C_{j(2,1)} C_{j(2,2)} \cdots C_{j(2,t_2)} n_2 \gre m_3 C_{i(3,1)} C_{i(3,2)} \cdots C_{i(3,k_3)} n_3 \\
\qquad\qquad &\qquad\qquad\qquad\qquad\qquad\vdots \nonumber \\
\qquad\qquad p_\mu &\gre m_\mu C_{j(\mu,1)} C_{j(\mu,2)} \cdots C_{j(\mu,t_\mu)} n_\mu \gre f \nonumber
\end{align}
\\
In $(12)$ all $p_k \: (k = 1, 2, \dots, \mu)$ are $c_m$-words, so $\len{p_k} \leq \len{C_m}$ for all $k = 1, 2, \dots, \mu$, and because all $m_k$ and $n_k$ $(k = 1, 2, \dots, \mu)$ are non-empty words, $C_m$ does not appear among the $C$-words written out in (12). By the definition of the generating operation, in obtaining the word $f$ from $C_m$ we used the fact that in the group $\Gamma$ the following equalities hold.

\begin{align}
\qquad\qquad \gamma_{i(1,1)} \gamma_{i(1,2)} \cdots \gamma_{i(1,k_1)} &= \gamma_{j(1,1)} \gamma_{j(1,2)} \cdots \gamma_{j(1,t_1)} \nonumber \\
\qquad\qquad \gamma_{i(2,1)} \gamma_{i(2,2)} \cdots \gamma_{i(2,k_2)} &= \gamma_{j(2,1)} \gamma_{j(2,2)} \cdots \gamma_{j(2,t_2)} \nonumber \\ 
\qquad\qquad &\quad\vdots \\
\qquad\qquad \gamma_{i(\mu,1)} \gamma_{i(\mu,2)} \cdots \gamma_{i(\mu,k_\mu)} &= \gamma_{j(\mu,1)} \gamma_{j(\mu,2)} \cdots \gamma_{j(\mu,t_\mu)} \nonumber
\end{align}
\\
As we do not have $\gamma_m$ among the $\gamma$ appearing in the equalities in (13), it follows that in the semigroup $T$ we have the equalities

\begin{align}
\qquad\qquad C_{i(1,1)} C_{i(1,2)} \cdots C_{i(1,k_1)} &= C_{j(1,1)} C_{j(1,2)} \cdots C_{j(1,t_1)} \nonumber \\
\qquad\qquad C_{i(2,1)} C_{i(2,2)} \cdots C_{i(2,k_2)} &= C_{j(2,1)} C_{j(2,2)} \cdots C_{j(2,t_2)} \nonumber \\ 
\qquad\qquad &\quad\vdots \\
\qquad\qquad C_{i(\mu,1)} C_{i(\mu,2)} \cdots C_{i(\mu,k_\mu)} &= C_{j(\mu,1)} C_{j(\mu,2)} \cdots C_{j(\mu,t_\mu)}. \nonumber
\end{align}
\\

In the semigroup $T$ the equality $G_j = 1$ is true, as $G_j = 1$ is a defining relation of $T$. Those words $f$ which were substituted for $\gamma_m$ in the construction of $G_j$ (but not for all $f$ belonging to $G_j$) we will, using the equalities in (14), which are performed inside the semigroup $T$, perform the inverses of the substitutions which transform $C_m$ into $f$. As a result, the word $G_j$ will be transformed into $A_j$. Hence, $G_j = A_j$ in $T$, and hence $A_j = 1$ in the semigroup $T$, completing the proof of Lemma~12. 
\end{enumerate}
\end{enumerate}
\end{proof}

\begin{definition}
Two tuples will be said to be \textit{equivalent} if the system of relations of the semigroup associated to the first tuple is equivalent to the system of relations of the semigroup associated to the second tuple.
\end{definition}

\begin{definition}
The \textit{index} of the tuple $R$ is an ordered pair of natural numbers $(\alpha, \beta)$, where $\alpha$ is the sum of the lengths of the left-hand sides of the defining relations of the tuple $R$; and $\beta$ is the parameter $\omega$ of the reduced list of $C$-words of the tuple $R$.
\end{definition}

The above indices will be ordered lexicographically, i.e. $(\alpha_1, \beta_1) < (\alpha_2, \beta_2)$ if and only if one of the following hold: either $\alpha_1 < \alpha_2$, or $\alpha_1 = \alpha_2$ and $\beta_1 < \beta_2$. The index of a tuple $R$ will be denoted as $\mathcal{J}(R)$. 

\begin{lemma}
If the tuple $R$ lacks property \textnormal{I}, then we can construct a tuple $R'$, which is equivalent to $R$, and such that $\mathcal{J}(R') < \mathcal{J}(R)$. 
\end{lemma}
\begin{proof}
Let $R = \{ \Pi ; C_1, C_2, \dots, C_s \}$. As $R$ lacks property I, there exists some positive natural number $m \leq s$ such that the length of the word $c_m^{(t)}$ from the set $\{ c_m^{(j)} \}$ is strictly less than that of the length of $C_m$. 

In the defining relations of the group of the tuple $R$, we will replace the generators $\gamma_1, \gamma_2, \dots, \gamma_{m-1}, \gamma_m, \gamma_{m+1}, \dots, \gamma_s$ by the words $C_1, C_2, \dots, C, c_m^{(t)}, C_{m+1}, \dots, C_s$, and call the resulting semigroup $\Pi'$. This semigroup is derived from the semigroup of $R$, and is given by defining relations over the same generating alphabet as $\Pi$. Construct some $C$-words $C_1', C_2', \dots, C_{s'}'$ of the semigroup $\Pi'$, and construct the associated tuple $R' = \{ \Pi' ; C_1', C_2', \dots, C_{s'}' \}$. By Lemma~12 the tuples $R$ and $R'$ are equivalent; and as $\len{{c_m^{(t)}}} < \len{C_m}$, also $\alpha(R') < \alpha(R)$, and hence $\mathcal{J}(R') < \mathcal{J}(R)$. 
\end{proof}

\begin{lemma}
For every tuple $R$ one may construct a tuple $R_1$, equivalent to the tuple $R$, and which does not lack property I. 
\end{lemma}
\begin{proof}
This lemma is proved by induction on the index of the tuple $R$. 

Consider an arbitrary tuple $R_0 = \{ \Pi^0 ; C_1^0, C_2^0, \dots, C_s^0\}$, the index of which equals $(\alpha, 0)$, where $\alpha$ is some arbitrary positive number. The parameter $\omega$ of the list $C_1^0, C_2^0, \dots, C_s^0$ equals $0$, from which it follows, that every word $C_1^0, C_2^0, \dots, C_s^0$ is a single letter. But this means that every set $\{ c_i^{(j)} \}$, for $i = 1, 2, \dots, s$, of the tuple $R_0$ consists of a single letter, from which it follows, that the tuple $R_0$ does not lack property I. 

Suppose that for every tuple $R'$ with index less than $(\alpha_1, \beta_1)$ we can construct some tuple $R''$, which is equivalent to $R'$ and which does not lack property I. 

Consider a tuple $R$ with index $(\alpha_1, \beta_1)$. If this does not lack property I, then we may take this tuple itself as our desired tuple. If instead $R$ lacks property I, then by Lemma~13 we may construct a tuple $R'$, equivalent to $R$, and which satisfies $\mathcal{J}(R') < (\alpha_1, \beta_1)$. By the inductive hypothesis, from $R'$ we may construct a tuple $R''$, equivalent to $R'$ (and hence to $R$) and which does not lack property I. 
\end{proof}

\begin{lemma}
If the tuple $R$ does not lack property \textnormal{I}, but lacks property \textnormal{II}, then there exists some tuple $R''$, equivalent to $R$, with $\mathcal{J}(R'') < \mathcal{J}(R)$. 
\end{lemma}
\begin{proof}
Let $R = \{ \Pi ; C_1, C_2, \dots, C_2 \}$. As $R$ lacks property II, it follows that there exist natural numbers $i$ and $\mu$ $(i \neq \mu)$ such that there is some word $f$ in both $\{ c_i^{(j)} \}$ and $\{ c_\mu^{(j)} \}$, i.e. $c_i^{(m)} = c_\mu^{(n)}$. As $R$ does not lack property I, we have that $\len{C_i} =  \len{ c_i^{(m)} }$ and $\len{C_\mu} = \len{c_\mu^{(n)}}$, and hence $\len{C_i} = \len{C_{\mu}} = g$ (where $g > 1$, as the list $C_1, C_2, \dots, C_s$ is reduced, and $C_i \not\gre C_\mu$ for $i \neq \mu$), so $\{ c_i^{(j)} \}$ and $\{ c_\mu^{(j)} \}$ have length $g$. From the definition of the generating operation it follows that the sets of words $\{ c_i^{(j)} \}$ and $\{ c_\mu^{(j)} \}$ coincide, and hence there is some $t$ such that $C_\mu \gre c_i^{(t)}$. 

Define the semigroup $\Pi_2$ to have the same generating alphabet as the semigroup $\Pi$ and subject to the defining relations of the group of the tuple $R$, in which the generators 
\[
\gamma_1, \gamma_2, \dots, \gamma_{i-1}, \gamma_i, \gamma_{i+1}, \dots, \gamma_s
\] have been replaced by the words 
\[
C_1, C_2, \dots, C_{i-1}, c_i^{(t)}, C_{i+1}, \dots, C_s.
\]
It is easy to verify (as $c_i^{(t)} \gre C_\mu$) that $C_1, C_2, \dots, C_{i-1}, C_{i+1}, \dots, C_s$ are $C$-words of the semigroup $\Pi_2$. Define $R'' = \{ \Pi_2 ; C_1, C_2, \dots, C_{i-1}, C_{i+1}, \dots, C_s \}$. From the equality $\len{ c_i^{(t)} } = \len{C_\mu}$ it follows that $\alpha(R) < \alpha(R'')$. As $\len{C_i} > 1$, we have $\beta(R'') < \beta(R)$. This proves the lemma. 
\end{proof}

\begin{lemma}
If the tuple $R = \{ \Pi ; C_1, C_2, \dots, C_s \}$ does not lack property \textnormal{I} but lacks property \textnormal{III}, then there exists some $C$-word $C_p$ and some $c$-word $c_\mu^{(v)}$ which overlap.
\end{lemma}
\begin{proof}
As $R$ lacks property III, it follows that there exists some $c_p^{(w)}$  and $c_\mu^{(t)}$, such that $c_p^{(w)} \gre MN$, and $c_\mu^{(t)} \gre NL$, where $ML$ and $N$ are non-empty words. Suppose without loss of generality that $M$ is non-empty (if $M \gre 1$, then $L$ is non-empty, and all the reasoning below is symmetrical). Suppose that the generating operation obtains $c_p^{(w)}$ from $C_p$ by the following sequence of $c_p$-words:
\[
C_p \to c_p^{(1)} \to c_p^{(2)} \to \dots \to c_p^{(w-1)} \to c_p^{(w)},
\]
where each transition from $c_p^{(i)}$ to $c_p^{(i+1)}$ is a single step of the generating operation. 

If $w=0$, then $c_p^{(w)} \gre C_p$ and the lemma follows. 

Suppose that $w>0$, and that the lemma is true for all $c_p^{(k)}$, which can be obtained from $C_p$ using fewer than $w$ steps. 

Consider the step $c_p^{(w-1)} \to c_p^{(w)}$, and write 
\begin{align*}
c_p^{(w-1)} &\gre h_1 C_{i_1} \cdots C_{i_t} h_2 \\
c_p^{(w)} &\gre h_1 C_{j_1} \cdots C_{j_u} h_2
\end{align*}
where $h_1, h_2$ are non-empty and $\gamma_{i_1} \cdots \gamma_{i_t} = \gamma_{j_1} \cdots \gamma_{j_u}$ in the group $\Gamma$ of $R$. 

If some $C_{j_t}$ overlaps with $c_\mu^{(t)}$, then the lemma follows.   In the other case, $C_{j_1} \cdots C_{j_u}$ lies entirely within $M$ or entirely within $N$. If $C_{j_1} \cdots C_{j_u}$ lies in $M$, then doing the reverse substitution, we see that $c_p^{(w-1)}$ overlaps with $c_t^{(\mu)}$, from which by the inductive hypothesis the lemma follows. If $C_{j_1} \cdots C_{j_u}$ lies in $N$, then doing the reverse substitution, we see that $c_p^{(w-1)}$ overlaps with a word obtained from $c_t^{(\mu)}$ by a single substitution, which is the reverse of a single substitution in the generating operation. As $R$ does not lack property I, all substitutions of the operation preserve the lengths of the words to which they are applied, from which it follows that by applying this substitution to $c_t^{(\mu)}$ we get some word $c_t^{(\mu_1)}$. That is, $c_p^{(w-1)}$ overlap with $c_t^{(\mu_1)}$, and by the inductive hypothesis we have shown that some $C$-word overlaps with some $c$-word. 
\end{proof}

\begin{lemma}
If the tuple $R = \{ \Pi ; C_1, C_2, \dots, C_s \}$ does not lack property \textnormal{I}; and there exists some $i$ such that $C_i \gre MN$, $c_i^{(p)} \gre NL$ with $N$ and $ML$ non-empty; and no $C_j$ $(j = 1, 2, \dots, s)$ does not overlap with any $c_\mu^{(v)}$ for $j \neq \mu$; then there is some $c_i$-word $c_i^{(e)}$ which overlaps with itself.
\end{lemma}
\begin{proof}
As the tuple $R$ does not lack property I, we have $\len{C_i} = \len{c_i^{(p)}}$, from which it follows that $M$ and $L$ are non-empty, and furthermore that all substitutions made by the generating operation does not change the length of the words involved. 

Consider a sequence of $c_i$-words starting with the word $C_i$, where $c_i^{(\alpha_m)}$ is obtained from $c_i^{(\alpha_{m-1})}$ by a single step of the generating operation, and such that every word from $ \{ c_i^{(j)} \}$ is somewhere in this sequence. 
\begin{equation}
C_i \to c_i^{(\alpha_0)} \to c_i^{(\alpha_1)} \to \dots \to c_i^{(\alpha_{\mu -1} )} \to c_i^{(\alpha_\mu)}
\end{equation}
All words in the sequence (15) have the same length. We do not exclude the possibility that some $c_i$-word appears more than once in the sequence (15). 

We must have that $\len{C_i} > 2$, for otherwise the set $\{ c_i^{(j)} \}$ consists only of the single word $C_i$, and the conditions of the lemma cannot be met. 

Suppose that $C_i \gre x \mathcal{E} y$, where $x$ and $y$ are single letters. Necessarily, $x \not\gre y$, for otherwise $C_i$ overlaps with itself, contradicting the definition of a $C$-word. These letters $x$ and $y$ are not affected during the generation of the $c_i$-words in the sequence (15), and hence $c_i^{(k)} \gre x e_k y$ for $k = \alpha_0, \alpha_1, \alpha_2, \dots, \alpha_\mu$. 

The word $c_i^{(p)}$ appears in the sequence (15) as some $c_i^{(\alpha_u)}$, and $c_i^{(\alpha_u)} \gre x e_{\alpha_u} y$. 

As $C_i$ and $c_i^{(\alpha_u)}$ overlap, we have that $M \gre x M_1, N \gre x N_1 y$, and $L \gre L_1 y$, from which it follows that 
\begin{equation}
C_i \gre x M_1 x N_1 y \qquad \textnormal{and} \qquad c_i^{(\alpha_u)} \gre x N_1 y L_1 y.
\end{equation}

We will prove that the letter $x$ emphasised in the middle of $C_i$ in (16) is unaffected by the steps of the generating operation in (15). 

Suppose that $x$ is not affected in the part of the sequence (15) given by
\[
C_i \to c_i^{(\alpha_0)} \to c_i^{(\alpha_1)} \to \dots \to c_i^{(\alpha_w)}.
\]
Then $c_i^{(\alpha_w)} \gre x M_w x N_w y$, and $\len{M_w} = \len{M_1}$. Suppose that the $x$ under consideration is affected in the step $c_i^{(\alpha_w)} \to c_i^{(\alpha_{w+1})}$. Then some $C$-word $C_h$, the length of which is less than the length of the $C$-word $C_i$, overlaps with the word $xN_w y$. But $xN_w y L_1 x$ is some $c_i$-word $c_i^{(\beta)}$, and hence this word is obtained from $c_i^{(\alpha_u)}$ by length-preserving substitutions, and hence $C_h$ overlaps with $c_i^{(\beta)}$ where $h \neq i$, which contradicts the assumptions of the lemma. 

Analogously, one may show that the $y$ which appears in the middle of $c_i^{(\alpha_u)}$ in $(16)$ is unaffected by the steps in the operation in (15). 

Consider the word $M_1$ in $(16)$ and the word $\varepsilon M_1 \varepsilon$, where $\varepsilon$ is any arbitrary letter; we will apply the generating operation as many times as possible to this word. From this, we obtain the following list of words, all of the same length:
\[
\varepsilon M_1 \varepsilon, \varepsilon M_2 \varepsilon, \dots, \varepsilon M_f \varepsilon.
\]

We write $N_1$ in the following form 
\begin{equation}
N_1 \gre M_{\delta_1} x M_{\delta_2} x M_{\delta_3} x \cdots M_{\delta_e} x V \qquad \textnormal{and} \qquad V \not\gre M_t x V_1,
\end{equation}
where $1 \leq \delta_1, \delta_2, \dots, \delta_e$, and $t \leq f, e \geq 0$. This factorisation is unique, as all $M_{\delta_i}$ have the same length. Thus
\begin{equation}
\begin{aligned}
C_i &\gre x M_{1} x M_{\delta_2} x M_{\delta_3} x \cdots M_{\delta_{e-1}} x M_{\delta_e} x V y \\
c_i^{(\alpha_u)} &\gre x M_{\delta_1} x M_{\delta_2} x M_{\delta_3} x \cdots M_{\delta_{e}} x V y L_1 y
\end{aligned}
\end{equation}

We now show that all the $x$-letters marked inside $C_i$ in $(18)$ are unaffected in the steps of the sequence (15). We have already shown that the first two $x$-letters are unaffected in (15). Suppose that the $\lambda$ leftmost $x$-letters are unaffected in (15), and that the $(\lambda+1)$th is affected. By analogous reasoning to the above, this leads to a contradiction.

In the same way, we can show that the penultimate $y$ distinguished in $(18)$ is not affected in (15). In the subword $M_{\delta_e} x V$ of the word $C_i$ lies $y$, which is not affected by (15), and in the subword $V y L_1$ of the word $c_i^{\alpha_u}$ lies $x$, which is not affected by (15), and $M_{\delta_e} x V$ is transformed into $V y L_1$ by the generating operation. If $x$ would precede $y$, then $V \gre M_{\delta_e} x V_1$, which is impossible by virtue of (17). Consequently, $y$ precedes $x$. Hence, $M_{\delta_e} \gre V y W$ and $L_1 \gre WxV$. 

Hence we have 
\begin{equation*}
\begin{aligned}
C_i &\gre x M_{1} x M_{\delta_2} x M_{\delta_3} x \cdots x M_{\delta_{e-1}} x M_{\delta_e} x V y, \\
c_i^{(\alpha_u)} &\gre x M_{\delta_1} x M_{\delta_2} x M_{\delta_3} x \cdots x M_{\delta_{e}} x V y W x V y.
\end{aligned}
\end{equation*}
But the generating operation allows the replacement of any $M_{\delta_k}$ by any $M_{\delta_s}$ in the words $C_i$ and $c_i^{(\alpha_u)}$. By replacing inside all $C_i$ and $c_{i}^{\alpha_u}$ all occurrences of $M_{\delta_i}$ by $M_{\delta_e}$ (and noting that $M_{\delta_e} \gre V y W$), we obtain a $c_i$-word $c_i^{(e)}$ which overlaps with itself, proving the lemma. 
\end{proof}

\begin{lemma}
If the tuple $R = \{ \Pi ; C_1, C_2, \dots, C_s \}$ does not lack property \textnormal{I}, but lacks property \textnormal{III}, then there exists a tuple $R'''$, equivalent to the tuple $R$, and such that $\mathcal{J}(R''') < \mathcal{J}(R)$. 
\end{lemma}
\begin{proof}
By Lemma~16, there exists some $C$-word $C_p$ and some $c$-word $c_\mu^{(v)}$ which overlap. We break the proof of the lemma into two parts, with the first case being when $\mu \neq p$, and the second case being when $\mu = p$, but no $C_j$ and $c_{\mu}^{(v)}$, $j \neq \mu$, overlap. 

\

\noindent\underline{Case $a$) $\mu \neq p$}. 

\

In the process described in Definition~8, we replace the generators of the group $\Gamma$ of the tuple $R$ by the words $C_1, C_2, \dots, C_{\mu-1}, c_\mu^{(v)}, C_{\mu+1}, \dots, C_s$ and obtain a $(k, \ell)$-semigroup $T_1$ with a system of relations equivalent to the system of relations of the $(k, \ell)$-semigroup $\Pi$ (by Lemma~12). Consider the words
\[
\Delta[\mathcal{T}_1(C_1, C_2, \dots, C_{\mu-1}, c_\mu^{(v)}, C_{\mu+1}, \dots, C_s)] = C_1^{T_1}, C_2^{T_1}, \dots, C_{s(T')}^{T_1}.
\]
It is easy to see that $C_1^{T_1}, C_2^{T_1}, \dots, C_{s(T')}^{T_1}$ are $C$-words of the semigroup $T_1$, and hence the $(k, \ell)$-semigroup associated to the tuple $R''' = \{ T_1 ; C_1^{T_1}, C_2^{T_1}, \dots, C_{s(T')}^{T_1}\}$ is a $(k, \ell)$-group. 

Furthermore, $\alpha(R) = \alpha(R''')$, as $\len{C_\mu} = \len{c_\mu^{(v)}}$. However, $\beta(R''') < \beta(R)$, as $c_\mu^{(v)}$ and $C_p$ overlap, and hence by Lemma~1 we have 
\[
\omega(C_1^{T_1}, C_2^{T_1}, \dots, C_{s(T')}^{T_1}) < \omega(C_1, C_2, \dots, C_{\mu-1}, c_\mu^{(v)}, C_{\mu+1}, \dots, C_s).
\]

\

\noindent\underline{Case $b)$ $C_p$ and $c_p^{(v)}$ overlap, but none of the $C_j$ and $c_\mu$, for $j \neq \mu$, overlap.} 

\

In this case, by Lemma~17 there exists some $c_p$-word $c_p^{(e)}$ which overlaps with itself. In the process described in Definition~8, we replace the generators of the group $\Gamma$ of the tuple $R$ by the words $C_1, C_2, \dots, C_{p-1}, c_p^{(e)}, C_{p+1}, \dots, C_s$ and obtain a $(k, \ell)$-semigroup $T_2$, with a system of relations equivalent to the system of relations of the $(k, \ell)$-semigroup $\Pi$. The argument then continues in the same way as in case $a)$ of the lemma. Here we use the fact that $c_p^{(e)}$ overlaps with itself, and therefore the parameter $\omega$ for the $C$-words of the semigroup $T_2$ is therefore strictly less than the parameter $\omega$ for the semigroup $\Pi$. 
\end{proof}

\begin{theorem}
If the natural numbers $k$ and $\ell$ are such that there exists an algorithm for solving the identity problem in every $(k, \ell)$-group, then there exists an algorithm which for every $(k, \ell)$-tuple $R$ constructs a distinguished $(k, \ell)$-tuple $V$, equivalent to $R$. 
\end{theorem}
\begin{proof}
The theorem is true, when the index of the $(k, \ell)$-tuple $R$ has $\beta = 0$, since in this case every $C$-word of the tuple $R$ is a single letter, and consequently the tuple $R$ is distinguished.

Assume that for all tuples with index $\mathcal{J} < \mathcal{J}'$ the theorem is true. Suppose that the tuple $R$ has index $\mathcal{J}'$, and is lacking one of the properties (otherwise $R$ is already distinguished). By one of Lemmas~13, 15, or 18 we may find a tuple $R_1$ with index strictly less than $\mathcal{J}'$ and equivalent to the tuple $R$. By the inductive hypothesis, we may find a distinguished tuple $V$, equivalent to the tuple $R'$, as $\mathcal{J}(R_1) < \mathcal{J}'$. To complete the proof, we note that the tuples $V$ and $R$ are equivalent.
\end{proof}

\section{The identity problem for $(k, \ell)$-semigroups}

We introduce some notation and assumptions which will be limited to this section. Let the $(k, \ell)$-semigroup $V$ be given by the generating alphabet 
\setcounter{equation}{0}
\begin{equation}
a_1, a_2, \dots, a_n
\end{equation}
and the defining relations 
\setcounter{equation}{18}
\begin{equation}
\{ V_i = 1 \quad (i = 1, 2, \dots, k).
\end{equation}
Let the list $C_1, C_2, \dots, C_u$ be a list of $C$-words of the semigroup $V$. Let the $(k, \ell)$-group $\Gamma$, associated to the semigroup $V$ and its $C$-words $C_1, C_2, \dots, C_u$ be given by the generating alphabet 
\begin{equation}
\delta_1, \delta_2, \dots, \delta_u
\end{equation}
and the defining relations 
\begin{equation}
\{ \Psi_i = 1 \quad (i = 1, 2, \dots, k).
\end{equation}
Assume there exists an algorithm which solves the identity problem in the $(k, \ell)$-group $\Gamma$. 

Let the tuple $R = \{ V ; C_1, C_2, \dots, C_u \}$ be distinguished, i.e. the set of $c$-words $\{ c_1^{(j)} \}, \{ c_2^{(j)} \}, \dots, \{ c_u^{(j)} \}$ does not lack any of the properties from the previous sections.

% This is clearly being in Delta^\ast
\begin{definition}
The word $A$ over the alphabet $(1)$ is called \textit{integral}, if $A \gre c_{i_1}^{(j_1)} c_{i_2}^{(j_2)} \cdots c_{i_m}^{(j_m)}$, where $c_{i_1}^{(j_1)},  c_{i_2}^{(j_2)}, \dots, c_{i_m}^{(j_m)}$ are $c$-words of the tuple $R$. The empty word $1$ is always an integral word. 
\end{definition}

\begin{definition}
On the set of integral words we define the \textit{function $f$}, which maps an integral word to a word from the group $\Gamma$. Let $A \gre c_{i_1}^{(j_1)} c_{i_2}^{(j_2)} \cdots c_{i_m}^{(j_m)}$. Then $f(A) \rightleftharpoons \delta_{i_1} \delta_{i_2} \cdots \delta_{i_m}$. This function is well-defined, as in the distinguished tuple of $c$-words no two words overlap, and because no word belongs to $\{ c_i^{(j)}\}$ and $\{ c_p^{(j)} \}$ for $i \neq p$.
\end{definition}

We notice that $f(AB) \gre f(A)f(B)$; and that if $A \gre c_{i_1}^{(j_1)} c_{i_2}^{(j_2)} \cdots c_{i_m}^{(j_m)}$ and $B = c_{i_1}^{(p_1)} c_{i_2}^{(p_2)} \cdots c_{i_m}^{(p_m)}$, then $f(A) \gre f(B)$. 

\begin{definition}
Every elementary transformation of the $(k, \ell)$-semigroup $V$ is of one of the two forms 
\begin{equation}
XY \to XAY
\end{equation}
\begin{equation}
XAY \to XY
\end{equation}
where $A$ is one of the $V_i$ of the defining relations $(19)$, and $X$ and $Y$ are arbitrary words over the alphabet $(1)$. We say that a transformation of type $(22)$ is an \textit{insertion}, and that a transformation of type $(23)$ is a \textit{deletion}. Two integral words $A$ and $B$ are called \textit{equivalent} with respect to the tuple $R$, if $f(A) = f(B)$ in the group $\Gamma$ of the tuple $R$. If $S \gre XAY$ and $T \gre XBY$ are words over the alphabet $(1)$, with $A$ and $B$ equivalent, then we say that $S$ and $T$ are equivalent with respect to a \textit{replacement} in the tuple $R$. 
\end{definition}

It is clear that every elementary transformation is a replacement, and that every replacement can be carried out using some finite number of elementary transformations. If $\len{A} \leq \len{B}$, then the replacement $XAY \to XBY$ is called \textit{non-decreasing}, and the replacement $XBY \to XAY$ is called \textit{non-increasing}. It is sometimes convenient to consider an insertion as a non-decreasing replacement, and a deletion as a non-increasing replacement. We extend the notion of \textit{affected} and \textit{unaffected} individual letters in a natural way (see Remark~3) to sequences of elementary transformations and replacements. 

\begin{definition}
If $A$ is a word over the alphabet $(1)$, then by $\fT(A)$ we denote the set of all such words $B$ over the alphabet $(1)$, such that there exists a finite sequence of words $X_1, X_2, \dots, X_\alpha$, where $X_1 \gre A, X_\alpha \gre B$, and  $X_{i+1}$ is obtained from $X_i$ by a non-increasing replacement for all $i = 1, 2, \dots, \alpha-1$.
\end{definition}

As $\len{B} \leq \len{A}$ for every word in the set $\fT(A)$, there exists an algorithm which takes as input a word $A$ over the alphabet $(1)$, and produces the set of words $\fT(A)$. The word $A$ is \textit{final} if for every $B \in \fT(A)$ we have $\len{B} = \len{A}$. It is clear that any subword of a final word is final, and that every $c$-word of a distinguished tuple $R$ is final (as $R$ does not lack property I and III). 

\begin{definition}
Let $A$ be a final word over the alphabet $(1)$, and 
\begin{align}
A &\gre y_1 y_2 \cdots y_{s_1} c_{i(1,1)} c_{i(1,2)} \cdots c_{i(1,m_1) y_{s_1+1}} y_{s_1+2} \cdots y_{s_2} c_{i(2,1)} c_{i(2,2)} \cdots \nonumber\\ &\cdots c_{i(2,m_2)} y_{s_2+1} y_{s_2+2} \cdots y_{s_3} \cdots c_{i(v-1,1)} c_{i(v-1,2)} \cdots \\ &\cdots c_{i(v-1,m_{v-1})} y_{s_{v-1}+1} y_{s_{v-1}+2} \cdots y_{s_v} \nonumber
\end{align}
where every $y_i \: (i = 1, 2, \dots, s_v)$ is a letter; and every $c$-word $c_{i(t,\mu)}$ appearing in $(24)$ has the following property: 

Denote the prefix of the word $A$ appearing before the word $c_{i(t,\mu)}$ by $X_1$, and the suffix of the word $A$ appearing after the word $c_{i(t,\mu)}$ by $Y_1$, i.e. $A \gre X_1 c_{i(t, \mu)} Y_1$. Then there does not exist a $c$-word $c_p$ such that $A \gre X_2 c_p Y_2$, where $X_2 Y_2$ is non-empty, and $c_p \gre X_3 c_{i(t,\mu} Y_3$, where $X_3 Y_3$ is non-empty, and $A \gre X_2 X_3 c_{i(t,\mu)} Y_3 Y_2$, where $X_1 \gre X_2 X_3, Y_1 \gre Y_3 Y_2$. Furthermore, no subword $y_{s_i+1} y_{s_i+2} \cdots y_{s_{i+1}}$ of $A$ is a $c$-word. 

We then say that the right-hand side of the equality in $(24)$ is the \textit{representation} of the final word $A$. 
\end{definition}

There is an obvious algorithm which for any final word $A$ computes its representation. It is also easy to see that the representation of a final word is unique. Consider the word $Y \gre y_1 y_2 \cdots y_{s_v}$, which we will call the \textit{stable} word of the final word $A$, and every letter $y_m$ of the word $Y$ will be called the $m$\textit{th} \textit{stable letter} of the final word $A$. The letters $y_m$ and $y_{m+1}$ are called \textit{adjacent stable letters}. It is clear that between two adjacent stable letters $y_m$ and $y_{m+1}$ of a final word, there is an integral word, which will be called the $m$\textit{th main integral word} of the word $A$. 

\begin{definition}
Suppose we are given a final word $T_0$ and a finite sequence of a non-decreasing replacements 
\begin{equation}
T_0 \to T_1 \to T_2 \to \cdots \to T_{\mu-1} \to T_\mu \quad (\mu \geq 0).
\end{equation}
The word $T_\mu$ will be called a $T$\textit{-word}, if a final word $T_0$ and a finite sequence of non-decreasing replacements indicating how it was obtained from $T_0$ is specified for it. The \textit{rank} of a $T$-word is the number of non-decreasing replacements by which it was derived from the final word. 
\end{definition}

Let
\begin{align}
T_0 &\gre y_1 y_2 \cdots y_{s_1} c_{i(1,1)} c_{i(1,2)} \cdots c_{i(1,m_1) y_{s_1+1}} y_{s_1+2} \cdots y_{s_2} c_{i(2,1)} c_{i(2,2)} \cdots \nonumber\\ &\cdots c_{i(2,m_2)} y_{s_2+1} y_{s_2+2} \cdots y_{s_3} \cdots c_{i(v-1,1)} c_{i(v-1,2)} \cdots \\ &\cdots c_{i(v-1,m_{v-1})} y_{s_{v-1}+1} y_{s_{v-1}+2} \cdots y_{s_v} \nonumber
\end{align}
be the representation of the word $T_0$, and suppose the $\mu$ non-decreasing replacements of $(25)$ are applied to $T_0$. 

Suppose that the stable letters of the final word $T_0$ are all unaffected by the application of the replacements in $(25)$. Then 
\begin{equation}
T_\mu \gre M_1 y_1 M_2 y_2 M_3 \cdots M_{s_v} y_{s_v} M_{s_v+1},
\end{equation}
and the letters $y_1, y_2, \dots, y_{s_v}$ are called the \textit{stable} letters of the $T$-word $T_\mu$; the letters $y_m$ and $y_{m+1}$ are called \textit{adjacent stable letters} of the $T$-word $T_\mu$. The right-hand side of the equality $(27)$ is called the \textit{representation} of the $T$-word $T_\mu$, and the word $M_i$ is called the $i$\textit{th changing word} of the $T$-word $T_\mu$. Notice that every $M_i$ is obtained from some (possibly empty) integral word by a finite number of non-decreasing replacements. 

\begin{definition}
To every $T$-word $T_i$ with rank $i$ we assign the index $\lambda = (\gamma, \delta)$, where $\gamma = i$, and $\delta = \len{T_i}$. We order the indices $\lambda$ lexicographically (see Definition~10). 
\end{definition}

\begin{definition}
In the distinguished tuple $R$ every $c$-word is a final word. If in Definition~16 the final word $T_0$ is a $c$-word; if the stable letters of this stable $c$-word are unaffected by the sequence $(25)$; and if in the representation $(27)$ of the $T$-word $T_\mu$ the words $M_1$ and $M_{s_v+1}$ are empty; then the word $T_\mu$ is called a \textit{$d$-word} of rank $\mu$. 
\end{definition}

The following lemma is proved by induction on the index $\lambda$. 

\begin{lemma}
For any $\lambda$ the following three statements are true: 
\textnormal{\begin{enumerate}[label=\Roman*.]
\item \textit{The stable letters of the final word $T_0$ are not affected by the sequence $(25)$ of replacements, if $(\mu, \len{T_\mu}) \leq \lambda$. }
\item \textit{Every changing word $M_j$ of the $T$-word $T_\mu$, with $(\mu, \len{T_\mu}) \leq \lambda$, is graphically equal to some product of $d$-words $M_j \gre d_{j_1}d_{j_2} \cdots d_{j_{t_j}}$, where the index of every $d_{j_i} \: (i = 1, 2, \dots, t_j)$ is strictly less than $\lambda$.}
\item \textit{If the final word $T_0$ does not overlap with any $c$-word, then $T_\mu$ does not overlap with any $c$-word, when $(\mu, \len{T_\mu}) \leq \lambda$.}
\end{enumerate}}
\end{lemma}
\begin{proof}
If the index $\lambda$ equals $(0,1)$, then all three statements of the lemma are obviously true. 

Let $\lambda' = (\mu, \len{T_\mu})$, and suppose that for all $\lambda < \lambda'$ the lemma is true. We will prove the lemma for $\lambda = \lambda'$. 

\

\noindent\underline{Part I.} The sequence $(25)$ begins with 
\begin{equation}
T_0 \to T_1 \to \cdots \to T_{\mu-1}
\end{equation}
By the inductive hypothesis of I, the stable letters of the final word $T_0$ are not affected by the sequence $(25)$ of replacements, as $(\mu-1, \len{T_{\mu-1}}) < \lambda'$. By the inductive hypothesis of II, we can write 
\begin{equation}
\begin{split}
T_{\mu-1} &\gre d_{\mu-1(1,1)} d_{\mu-1(1,2)} \cdots d_{\mu-1(1,m_1)} y_1 d_{\mu-1(2,1)} d_{\mu-1(2,2)} \cdots\\
& \cdots d_{\mu-1(2,m_2)} y_2 \cdots \cdots y_{s_v} d_{\mu-1(s_v+1,1)} d_{\mu-1(s_v+1,2)} \cdots d_{\mu-1(s_v+1,m_{s_v+1})}
\end{split}
\end{equation}
where $y_1 y_2 \cdots y_{s_v}$ is the stable word of the final word $T_0$; and every $d_{\mu-1(i,j)}$ appearing is a $d$-word, and furthermore the index of every $d_{\mu-1(i,j)}$ is less than $\lambda'$. 

By the inductive hypothesis of III, none of the $d_{\mu-1(i,j)}$ appearing in $(29)$ overlaps with any $c$-word. 

Consider the non-decreasing replacement
\begin{equation}
T_{\mu-1} \to T_\mu.
\end{equation}
\begin{enumerate}[label=\arabic*)]
\item Suppose that $(30)$ is an insertion. In this case no letter of the word $T_{\mu-1}$ is affected, and thus property I holds. 
\item Suppose that $(30)$ is a deletion. That is, $T_{\mu-1} \gre XAY$ and $T_\mu \gre XBY$, where $A$ and $B$ are non-empty, integral, and equivalent words. The word $A$ cannot be contained in any subword of $T_{\mu-1}$ consisting only of stable letters. This would contradict the definition of the representation of the final word $T_0$. The word $A$ does not overlap with the $d$-words appearing in $(29)$. Thus none of the stable letters of $T_{\mu-1}$ are affected by the replacement $(30)$. 
\end{enumerate}
\

\noindent\underline{Part II.} Consider the non-decreasing replacement
\setcounter{equation}{29}
\begin{equation}
T_{\mu-1} \to T_\mu.
\end{equation}
\setcounter{equation}{30}
\begin{enumerate}[label = \arabic*)]
\item Suppose that $(30)$ is an insertion $XY \to XV_iY$, and suppose that none of the $d$-words appearing in $(29)$ is divided when dividing $T_{\mu-1}$ into the words $X$ and $Y$. Then the representation of the word $T_{\mu-1}$ will differ from the representation of $T_\mu$ only in that some $M_j \gre d_{j_1} d_{j_2} \cdots d_{j_{t_j}}$ will be changed into 
\[
M_j' \gre d_{j_1} d_{j_2} \cdots d_{j_s} c_{i_1} c_{i_2} \cdots c_{i_m} d_{j_{s+1}} \cdots d_{j_{t_j}}.
\]
That is, the resulting word will be graphically equal to a product of $d$-words, and moreover (as $c$-words are $d$-words of rank zero) the index of every $d$-word of $M_j'$ is strictly less than $(\mu-1, \len{T_{\mu-1}})$, which in turn is less than $\lambda$. 
\item Suppose that $(30)$ is an insertion $XY \to XV_iY$, and that at least one of the $d$-words $d_{j_m}$ appearing in $(29)$ is divided when dividing $T_{\mu-1}$ into the words $X$ and $Y$. Then the representation of the word $T_{\mu-1}$ differs from the representation of $T_\mu$ only in that $d_{j_m}$, with index $\lambda_1 < (\mu-1, \len{T_{\mu-1}})$, will be replaced by some $d$-word with index $\lambda_1+1 < (\mu, \len{T_\mu})$. 
\item Suppose that $(30)$ is not an insertion, i.e. $T_{\mu-1} \gre XAY$ and $T_\mu \gre XBY$, where $A$ and $B$ are non-empty, integral, and equivalent words, and suppose that $A$ is included in some $d_{\mu-1(i, j}$ appearing in $(29)$, and that it is not equal to this $d_{\mu-1(i, j)}$. Then this $d_{\mu-1(i, j)}$, with index $\lambda_1 <  (\mu-1, \len{T_{\mu-1}})$ is replaced in the representation of $T_\mu$ by some $d$-word with index $\lambda_1+1 < (\mu, \len{T_\mu})$. 
\item Suppose that $(30)$ is a non-decreasing replacement which is not an insertion, and let $A \gre d_{\mu-1(i, j)} \cdots d_{\mu-1(i, j+t)}$. But $A$ is an integral word, and $d_{\mu-1(i, j)} \cdots d_{\mu-1(i, j+t)}$ does not, by the inductive hypothesis of III, overlap with any $c$-word. Therefore, each $d_{\mu-1(i, j)}, \dots, d_{\mu-1(i, j+t)}$ is a $c$-word, and hence the replacement $(30)$ replaces one changing word, graphically equal to a product of $d$-words of zero rank, by another changing word with the same property.  
\end{enumerate}
\

\noindent\underline{Part III.} We prove this part by contradiction. Suppose that the word $T_\mu$ from the sequence $(25)$ overlaps with some $c$-word $c_s$, i.e. $T_\mu \gre MN$ and $c_s \gre NL$, where $N$ and $ML$ are non-empty words. 

By parts I and II we can write the $T$-word $T_\mu$ as its representation $(27)$. As $c_s$ is non-empty, we have that $c_s \gre eK$, where $e$ is some letter. 
\begin{enumerate}[label = \arabic*)]
\item Suppose that $e \gre y_i$, where $y_i$ is some stable letter. Then every $d$-word appearing in $(27)$ to the right of $e$ is a $c$-word, as the distinguished tuple $R$ does not lack property I. From here, the end of the word $T_\mu$ which starts with the $e$ in question, coincides with the end of the word $T_0$ which starts with the same $e$, which contradicts III. 
\item Suppose that $e$ belongs to some $M_j$ in the representation $(27)$. Now $M_j \gre d_{j_1} d_{j_2} \cdots d_{j_{t_j}}$ and so $e$ appears in some $d_{j_m}$. But as the index of $d_{j_m}$ is strictly less than the index of $T_\mu$ by II, it follows by the inductive hypothesis of III that $d_{j_m}$ does not overlap with any $c$-word, a contradiction.
\end{enumerate}
This completes the proof of Lemma~19.
\end{proof}

\begin{definition}
We place every $T$-word $T_\mu$, obtained from the final word $T_0$ by $\mu$ specified non-decreasing replacements, in correspondence with a word $O_\mu$ with the following properties:

Let the representation of $T_0$ be as in $(26)$. Suppose we are given the finite sequence of non-decreasing replacements $(25)$. Let the representation of the $T$-word $T_\mu$ be as in $(27)$ (by Lemma~19.I, we can speak of the representation of the $T$-word $T_\mu$ without the assumptions in Definition~16). By Lemma~19.II, the $s$th changing word of the $T$-word $T_\mu$ is 
\begin{equation}
M_s \gre d_{j(s,1)} d_{j(s,2)} \cdots d_{j(s,k_s)},
\end{equation}
where every $d_{j(s,t)}$ is a $d$-word; and for every $d_{j(s,t)}$ there is a $c$-word $c_{j(s,t)}$, from which it is obtained by the indicated non-decreasing replacements, and the index of every $d_{j(s,t)}$ is strictly less than the index of $T_\mu$. 

Then
\begin{equation}
\begin{split}
U_\mu &\gre c_{j(1,1)} c_{j(1,2)} \cdots c_{j(1,k_1)} y_1 c_{j(2,1)} c_{j(2,2)} \cdots c_{j(2,k_2)} y_2 \cdots \cdots \\ &\cdots y_{s_v} c_{j(s_v+1,1)} c_{j(s_v+1,2)} \cdots c_{j(s_v+1,k_{s_v+1})}
\end{split}
\end{equation}
and moreover 
\begin{enumerate}[label = \arabic*)]
\item Every $c_{j(s,t)}$ (for $s = 1,2, \dots, s_v+1$ and $t = 1, 2\dots, k_s$) which appears in $(32)$ is precisely the $c$-word from which one obtains the $d$-word $d_{j(s,t)}$ in the representation $(27), (31)$ of the $T$-word $T_\mu$. 
\item Every integral word 
\begin{equation}
c_{j(m,1)} c_{j(m,2)} \cdots c_{j(m,k_m)}
\end{equation}
in the word $O_\mu$ can be obtained by a single non-decreasing replacement from the whole word appearing between the stable letters $y_{m-1}$ and $y_m$ in the final word $T_0$. 
\end{enumerate}
\end{definition}

\begin{lemma}
Suppose we are given a finite sequence of non-decreasing replacements 
\begin{equation}
T_0 \to T_1 \to T_2 \to \cdots \to T_\mu \to T_{\mu+1}
\end{equation}
where $T_0$ is a final word. Then for the $T$-word $T_{\mu+1}$ we can construct the corresponding $O_{\mu+1}$. 
\end{lemma}
\begin{proof}
For $T_0$, the corresponding word $O_0$ is just $T_0$ itself. 

Suppose that for the $T$-word $T_\mu$ we already know the corresponding word $O_\mu$. We represent $T_\mu$ of the form given in $(27), (31)$, and the the corresponding $O_\mu$ of the form $(32)$. 

Consider the the non-decreasing replacement 
\begin{equation}
T_\mu \to T_{\mu+1}.
\end{equation} 

\begin{enumerate}[label =\arabic*)]
\item Suppose that $(35)$ is an insertion. If the insertion is inside some $d$-word $d_{j(\alpha, \beta)}$ appearing in $(27), (31)$, then we set $O_\mu \gre O_{\mu+1}$. Indeed, from the properties of $O_\mu$, it is easy to see that $O_\mu$ will be the word $O_{\mu+1}$ corresponding to $T_{\mu+1}$. If instead the insertion is between two $d$-words, or between two stable letters, or between a stable letter and a $d$-word, then $O_{\mu+1}$ is obtained from $O_\mu$ by an insertion between the same words (where we consider $c_{i(\alpha, \beta)}$ and $d_{i(\alpha, \beta)}$ as the same for these purposes). It is easy to verify from the definition of $O_\mu$ that this choice of $O_{\mu+1}$ is correct.
\item Suppose that $(35)$ is a non-decreasing replacement, but not an insertion. By Lemma~19, it follows that the replaced word either completely lies inside some $d$-word, or else is equal to a product of $d$-words of rank zero (i.e. $d$-words which are $c$-words). In the first case, we set $O_\mu \gre O_{\mu+1}$. In the second case, the replaced $d$-words are $c$-words, and $O_{\mu+1}$ is obtained from $O_\mu$ by the same replacement as the one by which $T_{\mu+1}$ is obtained from $T_\mu$. Furthermore, if in $T_{\mu+1}$ the replaced word is $d_{j(\alpha, \beta)} \cdots d_{j(\alpha, \delta)}$, then in $O_{\mu+1}$ the corresponding replaced word is $c_{j(\alpha, \beta)} \cdots c_{j(\alpha, \delta)}$. 
\end{enumerate}
This completes the proof of Lemma~20.
\end{proof}
\begin{remark}
By saying that we are given a $T$-word $T_\mu$, we are really saying that we are given a final word $T_0$ and a sequence of $\mu$ non-decreasing replacements. If we wish to prove that some word $X$ is a $T$-word $T_\mu$, then we must prove that there exists some final word $T_0$, to which the application of $\mu$ non-decreasing replacements produces the word $X$. 
\end{remark}

\begin{lemma}
If one performs on a $T$-word $T_\mu$, obtained from a final word $T_0$, a deletion
\begin{equation}
T_\mu \gre XV_SY \to XY,
\end{equation}
where $V_s$ is one of the left-hand sides of the defining relations in $(19)$, then $XY$ is a $T$-word, obtainable from the same $T_0$. 
\end{lemma}
\begin{proof}
Suppose the lemma is true for all $T$-words whose length is less than $\lambda$, and let $\len{T_\mu} = \lambda$. From the definition of $(19)$ and Lemma~20 it follows that $T_\mu$ can be obtained from $T_0$ in the following way. From the final word $T_0$ with the presentation $(26)$ we construct via non-decreasing replacements the word $O_\mu$ with the presentation $(32)$, and then via non-decreasing replacements the word $T_\mu$ with the presentation $(27), (31)$; furthermore, every $d_{j(\alpha, \beta)}$ in $(27), (31)$ is replaced by the corresponding $c_{j(\alpha, \beta)}$ in the non-decreasing replacements of $(32)$.

When reducing $(36)$, there are two cases.
\begin{enumerate}[label = \arabic*)]
\item $d_{j(\alpha, \beta)} \gre X_1 V_s Y_1$. But $d_{j(\alpha, \beta)}$ is a $T$-word, obtained from $c_{j(\alpha, \beta)}$ by non-decreasing replacements, and $\len{d_{j(\alpha, \beta)}} < \len{T_\mu}$ (if $\len{d_{j(\alpha, \beta)}} = \len{T_\mu}$, then $XY \gre 1$, and $1$ is a $T$-word). By the inductive hypothesis, $X_1Y_1$ can be obtained from $c_{j(\alpha, \beta)}$ by non-decreasing replacements, from which it is easy to see that $X_1Y_1$ is a $d$-word. 
\item $d_{j(\alpha_1, \beta_1)} d_{j(\alpha_2, \beta_2)} \cdots d_{j(\alpha_s, \beta_s)} \gre V_S$. As no $d$-word overlaps with any $c$-word, and $V_s$ is an integral word, it follows that any $d_{j(\alpha_t, \beta_t)} \: (t = 1,2,\dots, s)$ is a $c$-word. From here, it follows that when $T_\mu$ is obtained from $O_\mu$, that then in the $c$-word $c_{j(\alpha_1, \beta_1)} c_{j(\alpha_2, \beta_2)} \cdots c_{j(\alpha_s, \beta_s)}$ no replacements were made. Suppose that the word
\[
X \gre c_{j(\alpha_1, \beta_1)} c_{j(\alpha_2, \beta_2)} \cdots c_{j(\alpha_s, \beta_s)}
\]
is enclosed in $O_\mu$ in the word 
\[
y_w c_{j(\gamma_1, \delta_1)} \cdots c_{j(\gamma_n, \delta_n)} X c_{j(\gamma_{n+1}, \delta_{n+1})} \cdots c_{j(\gamma_m, \delta_m)} y_{w+1}.
\]
The word between $y_w$ and $y_{w+1}$ is derived from some integral final word $Y$. It is clear that 
\[
c_{j(\gamma_1, \delta_1)} \cdots c_{j(\gamma_n, \delta_n)} c_{j(\gamma_{n+1}, \delta_{n+1})} \cdots c_{j(\gamma_m, \delta_m)}
\]
can be obtained from $Y$ by a non-decreasing replacement.
\end{enumerate}
This completes the proof of the lemma.
\end{proof}

\begin{lemma}
If some finite number $s$ of deletions are carried out on the $T$-word $T_\mu$, then the resulting word is some $T$-word obtained from the same final word $T_0$ as $T_\mu$ is obtained from.
\end{lemma}

The proof is an immediate consequence of Lemma~21. 

\begin{lemma}
If $X$ and $Y$ are final words, and $X$ is obtained from $Y$ by a sequence of elementary transformations, then $X \in \fT(Y)$ and $Y \in \fT(X)$. 
\end{lemma}
\begin{proof}
Let 
\begin{equation}
X \gre Z_1 \to Z_2 \to \cdots \to Z_n \gre Y
\end{equation}
be the sequence of elementary transformations in the statement of the lemma. As $X$ is final, we must have that $Z_1 \to Z_2$ is an insertion. Consider the sequence of elementary transformations 
\[
Z_1 \to Z_2 \to \cdots \to Z_{k'}
\]
consisting only of insertions, such that 
\[
Z_{k'} \to Z_{k'+1} \to \cdots \to Z_{k''}
\]
consists only of deletions, such that 
\[
Z_{k''} \to Z_{k''+1} \to \cdots \to Z_{k'''}
\]
consists only of insertions, etc. By Lemma~22, the sequence
\[
Z_1 \to Z_2 \to \cdots \to Z_{k''}
\]
can be reproduced by a sequence of non-decreasing substitutions. After a finite number of steps, repeating this reasoning, we find that $X$ can be obtained from $Y$ by a sequence of non-decreasing replacements; from this, we have that $X \in \fT(Y)$. Reasoning symmetrically, we also have that $Y \in \fT(X)$. 
\end{proof}

\begin{lemma}
Two words $X$ and $Y$ are equal in the $(k, \ell)$-semigroup $V$ if and only if the sets $\fT(X)$ and $\fT(Y)$ have a common element.
\end{lemma}
\begin{proof}
Suppose $Z \in \fT(X)$ and $Z \in \fT(Y)$. Then $Z = X$ and $Z = Y$ in the semigroup $V$, from which we have $X = Y$ in the semigroup $V$. 

Now suppose that $X = Y$ in the semigroup $V$. We find some final words $X_1$ and $Y_1$ such that $X_1 \in \fT(X)$ and $Y_1 \in \fT(Y)$. This is always possible, as every word of minimal length in $\fT(X)$ is a final word. This means that $X_1 = Y_1$ in $V$, from which it follows that $X_1$ is obtained from $Y_1$ by a sequence of elementary transformations. By Lemma~23, $X_1 \in \fT(Y_1)$ and as $Y_1 \in \fT(Y)$, we have $X_1 \in \fT(Y)$. Thus, we have found a word $X_1$ such that $X_1 \in \fT(X)$ and $X_1 \in \fT(Y)$. 
\end{proof}

\begin{theorem}
If the natural numbers $k$ and $\ell$ are such that there exists an algorithm for solving the identity problem in every $(k, \ell)$-group, then there exists an algorithm for solving the identity problem in the semigroup of any distinguished $(k, \ell)$-tuple. 
\end{theorem}
\begin{proof}
If we are given two words $X$ and $Y$ over the alphabet $(1)$, then we can find the sets of words $\fT(X)$ and $\fT(Y)$, which are finite, as for any word $A$ over the alphabet $(1)$ the set of words $\fT(A)$ consists of at most $n^{\len{A}}$ elements, where $n$ is the number of letters in the alphabet $(1)$. We then check whether $\fT(X)$ and $\fT(Y)$ have any element in common or not. By Lemma~24, the answer to this question is the same as whether $X$ and $Y$ are equal in the semigroup $V$ of the tuple $R$. 
\end{proof}

\begin{theorem}
If the natural numbers $k$ and $\ell$ are such that there exists an algorithm for solving the identity problem in every $(k, \ell)$-group, then there exists an algorithm for solving the identity problem in any $(k, \ell)$-semigroup.
\end{theorem}
\begin{proof}
By Theorem~1, for every $(k, \ell)$-semigroup $\Pi$ we can find a $(k, \ell)$-tuple $K$ such that $K = \{ \Pi ; C_1^k, C_2^k, \dots, C_s^k \}$. By Theorem~2 for every $(k, \ell)$-tuple we can find a reduced tuple $R = \{ V ; C_1, C_2, \dots, C_u \}$ such that the $(k, \ell)$-semigroups $\Pi$ and $V$ are equivalent. Thus to prove Theorem~4 we must only show that there exists an algorithm for solving the identity problem in the semigroup of any arbitrary distinguished $(k, \ell)$-tuple. But this is precisely what is proved in Theorem~3.
\end{proof}

\section{The divisibility problems for $(k, \ell)$-semigroups}

We retain the same notation and assumptions as in Section~1.2. 

\begin{lemma}
Suppose we are given a final word $T_0$ which does not have any end (beginning) which overlaps with any $c$-word of the tuple $R$. Suppose we are given the finite sequence of non-decreasing replacements 
\setcounter{equation}{24}
\begin{equation}
T_0 \to T_1 \to \cdots \to T_{\mu-1} \to T_\mu \quad (\mu \geq 0)
\end{equation}
Then the $T$-word $T_\mu$ does not have any end (beginning) overlapping with any $c$-word of the tuple $R$.
\setcounter{equation}{37}
\end{lemma}
\begin{proof}
Suppose $T_0$, given by the representation $(26)$, does not have any end overlapping with any $c$-word, and $T_\mu$, given by the representation $(27)$, overlaps with some $c$-word $c_i^{(j)}$. The word $c_i^{(j)}$ is a non-empty word, so $c_i^{(j)} \gre aX$. The letter $a$ is a stable letter of the $T$-word $T_\mu$, as no $c$-word overlaps with any $d$-word by Lemma~19.III. The changing words of $T_\mu$, located to the right of the specified $a$, are final words (as $c_i^{(j)}$ is final), and are therefore integral words. Accordingly, the end of the word $T_\mu$ located to the right of the specified $a$ is obtained from the end of $T_0$ to the right of the same $a$ (as $a$ is a stable letter of the sequence $(25)$) by the same number of replacements. Carrying out these replacements in the reverse order, we find that $T_0$ has an end which overlaps with some $c$-word, obtained from $c_i^{(j)}$ by these replacements. This is a contradiction. We can carry out the analogous proof for the beginnings of the $T$-word $T_\mu$.
\end{proof}

\begin{lemma}
If the non-empty final word $T_0$ has no end which overlaps with some $c$-word of the tuple $R$, and the word $X$ is left divisible by $T_0$ in the $(k, \ell)$-semigroup $V$, then $X \gre X_1 X_2$, where $X_1 = T_0$ in the $(k, \ell)$-semigroup $V$.
\end{lemma}
\begin{proof}
By definition of left divisibility, there exists some word $Z$ over the alphabet $(1)$, such that in the semigroup $V$ we have the equality $T_0 Z = X$. In other words, there exists some sequence of elementary transformations 
\begin{equation}
T_0 Z \to \cdots \to X.
\end{equation}
As $T_0$ is non-empty, we have that $T_0 \gre Ya$, where $a$ is some letter. We will consider
\begin{equation}
YaZ \to \cdots \to X \tag{$38'$}
\end{equation}
and prove, that the specified letter $a$ in the word $YaZ$ is unaffected by the sequence of elementary transformations $(38')$. 

Suppose that the letter $a$ becomes affected by the $w$th transformation 
\begin{equation}
YaZ \gre Y_0aZ_0 \to Y_1aZ_1 \to \cdots \to Y_{w-1} a Z_{w-1} \to X_w \to \cdots X. \tag{$38''$}
\end{equation} 
As in the sequence $(38'')$ the transformations are all elementary, the $w$th transformation, which affects $a$, is a deletion. The word $Y_0$ is a final word, as it is a subword of the final word $T_0$.  The word $Y_{w-1}$ can be obtained from $Y_0$ by a non-decreasing replacement (by Lemma~23). By Lemma~25, and as $Y_0 a$ has no end which overlaps with some $c$-word, it follows that $Y_{w-1}a$ has no end which overlaps with some $c$-word. This is a contradiction, and hence $a$ is not affected by the sequence $(38'')$, i.e. 
\[
Y_0 a Z_0 \to \cdots Y_{w-1} a Z_{w-1} \to Y_w a Z_w \to \cdots \to Y_s a Z_s \gre X
\]
where $T_0 \gre Y_0 a$. Thus we can take $X_1$ to be $Y_s a$, in which case $T_0 = X_1$ in the semigroup $V$. 
\end{proof}

\begin{lemma}
If $Y_1$ is the beginning of some $c$-word of the semigroup $V$ ($Y_1$ can itself be a $c$-word), and if the word $X$ is left divisible by $Y$, then $X$ is left divisible by $YY_1$.
\end{lemma}
\begin{proof}
Let $Y_1Z_1 \gre c_j^{(i)}$, and let $f(c_j^{(i)}) = \delta_j$. For the letter $\delta_j$ in the group $\Gamma$ of the tuple $R$, there exists an inverse $\delta_{j_1} \delta_{j_2} \cdots \delta_{j_m}$, i.e. such that $\delta_j \delta_{j_1} \delta_{j_2} \cdots \delta_{j_m} = 1$ in the group $\Gamma$. From this it follows that in the semigroup $V$ we have the equality $c_j^{(i)} C_{j_1} C_{j_2} \cdots C_{j_m} = 1$.

We have $Y_1 Z_1 C_{j_1} C_{j_2} \cdots C_{j_m} = 1$ in the semigroup $V$, and also $X = YZ$ in the semigroup $V$, as $X$ is left divisible by $Y$. Thus 
\[
X = YZ = YY_1Z_1 C_{j_1} C_{j_2} \cdots C_{j_m} Z
\]
that is, $X$ is left divisible by $YY_1$. 
\end{proof}

\begin{theorem}
If the natural numbers $k$ and $\ell$ are such that there exists an algorithm for solving the identity problem in every $(k, \ell)$-group, then there exists an algorithm which solves the left (right) divisibility problem in the semigroup of any distinguished $(k, \ell)$-tuple.
\end{theorem}
\begin{proof}
Suppose $X$ and $Y$ are words over the alphabet $(1)$, and we wish to decide whether or not $X$ is left divisible by $Y$ in the semigroup $V$ of the distinguished tuple $R$. Construct the set of words $\fT(Y)$. Choose from the set $\fT(Y)$ some final word $Y_1$. We have $Y = Y_1$ in the semigroup $V$, and hence our problem is equivalent to deciding whether $X$ is left divisible by $Y_1$. 

By direct verification, we check whether some end of $Y_1$ overlaps with some $c$-word of the tuple $R$. If there is some such overlap, then $Y_1 \gre Y_2 H_1$, where $H_1$ is a beginning of some $c$-word. We next check whether some end of $Y_2$ overlaps with some $c$-word of the tuple $R$. If there is some such overlap, then $Y_2 \gre Y_3 H_2$, where $H_2$ is the beginning of some $c$-word. This process continues until it ends with the construction of a (possibly empty) word $Y_k$, such that it has no end which overlaps with any $c$-word, and the words $H_1, H_2, \dots, H_{k-1}$, where $k \geq 0$, and every $H_i$ is the beginning of some $c$-word. We have $Y_k H_{k-1} H_{k-2} \cdots H_2 H_1 \gre Y_1$. Here $Y_k$ is a final word, as it is a subword of the final word $Y_1$. 

If $X$ is left divisible by $Y_1$, then $X$ is left divisible by $Y_k$. This is obvious from the equality
\[
X \gre Y_1 Z_1 \gre Y_k H_{k-1} H_{k-2} \cdots H_2 H_1 Z_1.
\]
On the other hand, if $X$ is left divisible by $Y_k$, then by Lemma~27 $X$ is left divisible $Y_k H_{k-1} H_{k-2} \cdots H_2 H_1 \gre Y_1$. 

Thus it remains to check whether $X$ is left divisible by $Y_k$. To do this, we can directly check (as the word problem is decidable in the semigroup $V$ by Theorem~3) whether or not $Y_k$ is equal to some beginning of $X$. If $Y_k = X_1$ in the semigroup $V$, then $X \gre X_1 X_2$, then by Lemma~25 we have that $X$ is left divisible by $Y_k$. If $X$ is left divisible by $Y_k$, then $X \gre X_1 X_2$, where $X_1 = Y_k$ in the semigroup $V$. 
\end{proof}

\begin{theorem}
If the natural numbers $k$ and $\ell$ are such that there exists an algorithm for solving the identity problem in every $(k, \ell)$-group, then there is also an algorithm for solving the left (right) divisibility problem in any $(k, \ell)$-semigroup. 
\end{theorem}
\begin{proof}
By Theorem~1, for any $(k, \ell)$-semigroup $\Pi$ we can construct a $(k, \ell)$-tuple $K$ such that $K = \{ \Pi ; C_1^k, C_2^k, \dots, C_s^k \}$. By Theorem~2, for every $(k, \ell)$-tuple $K$ we can construct some distinguished tuple $R = \{ V ; C_1, C_2, \dots, C_u \}$ such that the $(k, \ell)$-semigroups $\Pi$ and $V$ are equivalent. Thus, to prove Theorem~6 we only need to prove, that there exists an algorithm which solves the left divisibility problem in the semigroup of any distinguished $(k, \ell)$-tuple. But this is precisely what is proved in Theorem~5. 
\end{proof}

\section{The maximal subgroup of a $(k, \ell)$-semigroup}

We retain the same notation and assumptions as in Section~1.2. 

\begin{lemma}
No two $d$-words of a tuple $R$ overlap.
\end{lemma}
\begin{proof}
The proof is by induction on the sum of the lengths of the two considered $d$-words. If the sum of the two $d$-words is $2$, then these words cannot overlap by Definition~1.

Suppose that for any pair of $d$-words for which the sum of the lengths is less than $\lambda$ we have that they do not overlap. Suppose $d_\alpha \gre MN$, $d_\beta \gre NL$, where $N$ and $ML$ are non-empty words, and with $\len{d_\alpha} + \len{d_\beta} = \lambda$. 

Consider the $c$-word $c_\beta$, such that
\begin{equation}
c_\beta \gre X_0 \to X_1 \to \cdots \to X_m \gre d_\beta
\end{equation}
(see Remark~3) and the representation $(26)$ of the final word $c_\beta$. Let $N$ be a non-empty word such that $N \gre Ya$ and $d_\beta \gre YaL$. The letter $a$ highlighted in the $T$-word $d_\beta$ is stable, for otherwise $d_\alpha$ intersects with the $d$-word $d_\gamma$, which is a changing word of the $T$-word $d_\beta$, but $\len{d_\alpha} + \len{d_\gamma} < \lambda$. 

Every changing word of the $T$-word $d_\beta$ is changed into an integral, final word by non-increasing substitutions, by performing the sequence of elementary transformations $(39)$ in reverse. We obtain that $c_\beta$ has a beginning which intersects with some word, obtained from $d_\alpha$ by non-increasing replacements, i.e. some word $d_\delta$.

The proof that $d_\delta$ is a $d$-word follows from the statement that any replacement can be reproduced by a finite number of elementary transformations, that $c_\alpha$ is a final word, and Lemma~23. 
\end{proof}

\begin{lemma}
If the word $X$ is two-sided invertible in the semigroup $V$, then there exists some integral word $Z$ such that $X = Z$ in the semigroup $V$. 
\end{lemma}
\begin{proof}
As $X$ is two-sided invertible in the semigroup $V$, there exists some word $Y$ such that $XY = 1$ and $YX = 1$ in $V$.

Consider the sequence of elementary transformations 
\begin{equation}
1 \gre U_0 \to U_1 \to \cdots \to U_{e-1} \to U_e \gre XY
\end{equation}
The word $1$ is an integral word. By Lemma~23, we know that $1$ can be converted to $XY$ by a sequence of non-decreasing replacements, i.e. $U_e$ is a $T$-word. Construct the word $O_e$ of the $T$-word $T_e$, and consider the sequence of non-decreasing replacements (Lemma~20)
\begin{equation}
1 \to O_e \gre c_{i_1} c_{i_2} \cdots c_{i_s} \to U_2' \to U_3' \to \cdots \to U_e' \gre XY,
\end{equation}
where, starting from the second replacement, every replacement happens inside some $c$-word $c_{i_j}$. In this way, $XY$ is obtained from a $d$-word. In exactly the same way, we can prove that $YX$ can be obtained from a $d$-word. But as (by Lemma~28) $d$-words do not intersect, it follows that $X$ can be obtained from a $d$-word, and is obtained by non-decreasing replacements from the integral word $c_{i_1} c_{i_2} \cdots c_{i_t}$, where $t \leq s$. It is now clear that we can take $Z$ to be $c_{i_1} c_{i_2} \cdots c_{i_t}$. 
\end{proof}

\begin{lemma}
The group $\Gamma$ of the $(k, \ell)$-tuple $R$ is isomorphic to the maximal subgroup $F$ of the $(k, \ell)$-semigroup $V$ of the $(k, \ell)$-tuple $R$.  
\end{lemma}
\begin{proof}
For every two-sided invertible word $X$ of the semigroup $V$, we will define $\mu(X)$ as follows. By Lemma~29, for every two-sided invertible word $X$ there exists some integral word $Z$ such that $Z = X$ in $V$. If we choose for every $X_\alpha$ a $Z_\alpha$, then we set $\mu(X_\alpha) \rightleftharpoons f(Z_\alpha)$, see Definition~12, and, if $X_\alpha$ is itself integral, then $\mu(X_\alpha) \rightleftharpoons f(X_\alpha)$. 

Thus, we have an algorithm, which to every two-sided invertible word $X$ of the semigroup $V$ associates a word $f(Z_\alpha)$ of the group $\Gamma$ of the tuple $R$. Obviously, every word of the group $\Gamma$ can be obtained in this way from some word $X$. We will prove, that this mapping is an isomorphism between the group $\Gamma$ and the group $F$. For this, it is sufficient to prove the following lemma.

\begin{lemma} If $X$ and $Y$ are two-sided invertible words of the semigroup $V$, then $X = Y$ in $V$ if and only if $f(\mu(X)) = f(\mu(Y))$ in the group $\Gamma$. 
\end{lemma}
\begin{proof}
By Lemma~29 $\mu(X) = \mu(Y)$ in $V$ is equivalent to $X = Y$ in $V$. 

If $f(\mu(X)) = f(\mu(Y))$ in $\Gamma$, then $\mu(X) = \mu(Y)$ in $V$ by the definition of the group $\Gamma$ of the tuple $R$.

Suppose $\mu(X) = \mu(Y)$ in $V$. Without loss of generality, we can assume $\mu(X)$ is final. 

By Lemmas~20 and 23 there exists some sequence of non-decreasing replacements 
\begin{equation}
\mu(X) \gre Z_0 \to Z_1 \to \cdots \to Z_{m-1} \to Z_m \gre \mu(Y),
\end{equation}
and moreover $Z_1$ is an integral word, where $Z_0$ and $Z_1$ are equivalent words; $Z_1$ is an $O$-word of the $T$-word $\mu(Y)$; and every replacement in the sequence 
\begin{equation}
Z_1 \to Z_2 \to \cdots \to Z_{m-1} \to Z_m
\end{equation}
where $Z_1 \gre c_{i_1} c_{i_2} \cdots c_{i_n}$, replaces some $c$-word $c_{i_t}$, or some part of $c_{i_t}$, by the word $d_{i_t}$. In this way, $\mu(Y)$ is obtained from a $d$-word. But by definition, $\mu(Y)$ is an integral word, and hence $f(Z_1) = f(\mu(Y))$. As $Z_0 $ and $Z_1$ are equivalent words, we also have $f(\mu(X)) = f(\mu(Y))$ in $\Gamma$. This completes the proof of Lemma~31. 
\end{proof}
Using Lemma~31, this completes the proof of Lemma~30. 
\end{proof}

\begin{theorem}
If the natural numbers $k$ and $\ell$ are such that there exists an algorithm for solving the identity problem in every $(k, \ell)$-group, then there is an algorithm which for any $(k, \ell)$-semigroup $\Pi$ constructs a $(k, \ell)$-group, isomorphic to the maximal subgroup of $\Pi$. 
\end{theorem}
\begin{proof}
By Theorem~1, for every $(k, \ell)$-semigroup $\Pi$ we can construct a $(k, \ell)$-tuple $K$, such that $K = \{ \Pi ; C_1^k, C_2^k, \dots, C_s^k \}$. By Theorem~2, for every $(k, \ell)$-tuple $K$ we can construct a distinguished tuple $R = \{ V ; C_1, C_2, \dots, C_s \}$ such that the $(k, \ell)$-semigroup $\Pi$ and $V$ are equivalent. Thus, to prove Theorem~7, it suffices to prove that there exists an algorithm which for any semigroup of a distinguished $(k, \ell)$-tuple constructs some $(k, \ell)$-group, which is isomorphic to the maximal subgroup of the semigroup of this tuple. But this is precisely what is proved in Lemma~30. 
\end{proof}

\chapter{The Identity Problem In Finitely Presented Groups}
\section{Representing words equal to the identity in $K_{2/11}$-groups}\hfill\\

Let $G$ be a group with generating set
\begin{equation}
a_1, a_2, \dots, a_n
\end{equation}
and defining relations
\begin{equation}
R_i = 1 \quad (i = 1, 2, \dots, m).
\end{equation}

Let $M$ be the set of defining words of $G$ and, for the remainder of this chapter, suppose that $M$ is symmetrised. 

\textit{The class $K_{2/11}$}. The group $G$, given by the generators $(1)$ and the defining relations $(2)$, belongs to the class $K_{2/11}$ if the set $M$ of the group $G$ satisfies the following condition: if $R_i$ and $R_j$ are arbitrary, not mutually inverse words in $M$, then when reducing the product $R_iR_j$ we delete fewer than $2/11$ of the letters of the word $R_i$. 

If a non-empty reduced word $Q$ is equal to the identity of the group $G$, then by a theorem of Dyck there exists a natural number $w$, words $T_1, T_2, \dots, T_w$, and defining words $R_{i_1}, R_{i_2}, \dots, R_{i_w}$ such that $Q \equiv \prod_{j=1}^{w} \ol{T}_{j} R_{i_j}^{\varepsilon_j} T_j$, where $\varepsilon_j = \pm 1$ $(j = 1, 2, \dots, w)$. As the set $M$ is symmetrised, we may assume that this representation of $Q$ is of the form $Q \equiv \prod_{j=1}^{w} \ol{T}_{j} R_{i_j} T_j$.

We are not interested in any particular numbering of the defining words of $G$, and hence we will instead in the sequel write the product $\prod_{j=1}^{w} \ol{T}_{j} R_{i_j} T_j$ as $\prod_{i=1}^{w} \ol{T}_{i} R_i T_i$, where the two words $R_t$ and $R_z$ can be the same word, even if $t < z$. To simplify notation, the subword $\ol{T}_{p} R_p T_p \dots \ol{T}_{k} R_k T_k$ of the word $\prod_{i=1}^w \ol{T}_{i} R_i T_i$ will be denoted $[p, k]$. 

Consider the set of finite vectors $\{ k_1, k_2, \dots, k_w \}$, where $w >0$, and $k_1, k_2, \dots, k_w$ are non-negative integers. 

The \textit{parameter $\alpha$} of the vector $\omega = \{ k_1, k_2, \dots, k_w \}$ is defined as the natural number $\sum_{i=1}^w 6^{k_i}$ and will be denoted by $\alpha\{ k_1, k_2, \dots, k_w \}$, or simply $\alpha(\omega)$. 

The \textit{defect} of the number $k_j$ in the vector $\{ k_1, k_2, \dots, k_w \}$ is defined as the number $(k_{i_1} - k_j) + (k_{i_2} - k_j) + \dots + (k_{i_v} - k_j)$, where $k_{i_1}, k_{i_2}, \dots, k_{i_v}$ are all the numbers from the vector $\{ k_1, k_2, \dots, k_w \}$ such that $k_{i_z} > j$ and $i_z < j$, with $z = 1, 2, \dots, v$. 

The \textit{parameter $\beta$} of the vector $\omega = \{ k_1, k_2, \dots, k_w \}$, denoted $\beta(\omega)$, is defined as the sum of the defects of all numbers $k_1, k_2, \dots, k_w$ of the vector $\omega$. 

The \textit{index} of the vector $\omega = \{ k_1, k_2, \dots, k_w \}$ is defined as the ordered pair of natural numbers $(\alpha', \beta')$, where $\alpha' = \alpha(\omega)$ and $\beta' = \beta(\omega)$. The set of indices will be ordered lexicographically, i.e. $(\alpha_1, \beta_1) < (\alpha_2, \beta_2)$, where $\alpha_1, \alpha_2, \beta_1, \beta_2$ are non-negative natural numbers, if and only if one of the following two properties is true: \textbf{I.} $\alpha_1 < \alpha_2$, \textbf{II.} $\alpha_1 = \alpha_2$ and $\beta_1 < \beta_2$. We will denote the index of the vector $\omega$ as $\cJ(\omega)$. 

We will say that the vector $\omega_1 = \{ k_1, k_2, \dots, k_w \}$ is \textit{less} than the vector $\omega_2 = \{ k'_1, k'_2, \dots, k'_{w'} \}$ if $\cJ(\omega_1) < \cJ(\omega_2)$. 

The following three lemmas easily follow from the definition of our ordering on the vectors. 

\begin{lemma}
The vector $\omega_1 = \{ k_1, k_2, \dots, k_{i-1}, k_i, k_{i+1}, \dots, k_w \}$ is greater than the vector $\omega_2 = \{ k_1, k_2, \dots, k_{i-1}, s_1, s_2, \dots, s_v, k_{i+1}, \dots, k_w \}$, if we have $0 \leq v < 6$ and $s_z < k_i \: (z = 1, 2, \dots, v)$.
\end{lemma}

\begin{lemma}
The vector $\omega_1 = \{ k_1, k_2, \dots, k_{i-1}, k_i, \dots, k_w \}$ is greater than the vector $\omega_2 = \{ k_1, k_2, \dots, k_i, k_{i-1}, \dots, k_w \}$, if $k_{i-1} > k_i$. 
\end{lemma}

\begin{lemma}
For every vector $\omega = \{ k_1, k_2, \dots, k_w \}$ there exist only finitely many vectors $\omega_1, \omega_2, \dots, \omega_t$ such that $\cJ(\omega_t) < \dots < \cJ(\omega_2) < \cJ(\omega_1) < \cJ(\omega)$. 
\end{lemma}

We will now prove a theorem regarding the form of reduced words which are equal to the identity $K_{2/11}$. The theorem will describe some properties of the word $\prod_{i=1}^{w} \ol{T}_{i} R_{i} T_i$. For brevity, we will not describe the \textit{dual} properties (the dual properties of the word $\prod_{i=1}^{w} \ol{T}_{i} R_{i} T_i$ being the properties of $\left( \prod_{i=1}^{w} \ol{T}_{i} R_{i} T_i \right)^{-1}$ ), but all the properties are implied, and the proof of these properties coincide with the ordinary properties). 

\begin{theorem}
Let $G$ be a group belonging to the class $K_{2/11}$. If a non-empty word $Q$ is equal to the identity in $G$, then there exists a natural number $w$, words $T_1, T_2, \dots, T_2$ and reduced words $R_1, R_2, \dots, R_w$ such that $Q \equiv \prod_{i=1}^{w} \ol{T}_{i} R_{i} T_i$, where the following properties hold: 

\begin{property}
For every $i = 1, 2, \dots, w$ the word $\ol{T}_{i} R_{i} T_i$ is reduced.
\end{property}

\begin{property}
Let $R_k \gre U\ol{V}$ and $\len{U} > \frac{1}{2} \len{R_k}$. Then neither $\ol{T}_i$ nor $T_i \: (i = 1, 2, \dots, w)$ contains the word $U$. 
\end{property}

\begin{property}
If $\ol{R}_k$ and $R_i$ are not mutually inverse, and $R_k \gre ZXY$, $R_i \gre ZK$, $\ol{T}_i \gre NY$, then $\len{X} \geq \len{Y}$. 
\end{property}

\begin{property}
If $R_k \gre BAX$ and $R_s \gre \ol{X}CD$ are not mutually inverse, then in $\ol{T}_i$ there is a word $AC$, such that $\len{B} + \len{D} \geq \len{A}+\len{C}$. 
\end{property}
\begin{property}
If $R_k \gre BAX$ and $R_s \gre \ol{X}CDE$ are not mutually inverse, if $\ol{C}X\ol{E}\ol{D}$ and $R_i \gre DK$ are not mutually inverse, and $\ol{T}_i \gre NAC$, then $\len{E} + \len{B} \geq \len{C} + \len{A}$. 
\end{property}

\begin{property}
Suppose $R_k \gre U\ol{V}$ and $\len{U} \geq \frac{13}{22} \len{R_k}$. Then the word $U$ can be found inside the word $\ol{T}_iR_iT_i \: (i = 1, 2, \dots, w)$ only as a subword of the word $R_i$. 
\end{property}

\begin{property}
Suppose that $R_s \gre A\ol{X}Y$ and $R_k \gre \ol{Y}\ol{Z}\ol{U}D$ are not mutually inverse. Suppose that $\len{A} < \frac{4}{11}\len{R_s}$, $\len{D} < \frac{4}{11} \len{R_k}$, and $\len{U} < \frac{2}{11}\len{R_k}$. Then $UZK$ is not a subword of any $\ol{T}_iR_iT_i \: (i = 1, 2, \dots, w)$ such that $T_i \gre ZXT_i'$, $R_i \gre R_i' U$, and $\len{U} < \frac{2}{11}\len{R_i}$. 
\end{property}

\begin{property}
If $R_p \gre XYZ$, $R_k \gre U\ol{Y}V$, $p < k$ and for some way of reducing $\prod_{i=1}^{w} \ol{T}_{i} R_{i} T_i$ the subword $Y$ of the word $R_p$ is reduced by the subword $\ol{Y}$ of the word $R_k$, then the word $ZXY$ and $\ol{Y}VU$ are not mutally inverse. 
\end{property}

\begin{property}
If $R_p \gre XYZ$, $\ol{T}_{k} \gre U\ol{Y}V$, $p< k$ and for some way of reducing $\prod_{i=1}^{w} \ol{T}_{i} R_{i} T_i$ the subword $Y$ of the word $R_p$ reduces with the subword $\ol{Y}$ of the word $\ol{T}_k$, then $\len{T_p} \leq \len{U}$, and $\len{T_p} < \len{T_k}$. 
\end{property}

\begin{property}
If $T_p \gre AB$, $\ol{T}_j \gre CD$, where $p<j$, and for some way of reducing $\prod_{i=1}^{w} \ol{T}_{i} R_{i} T_i$ the word $B[p+1, j-1]C$ is absorbed, without reducing anything, then either $[p+1, j-1] \equiv 1$, or $p+1 = j$, or there is some $k$, with $p < k < j$, such that $\len{T_k} \geq \len{T_p}$. 
\end{property}
\end{theorem}
\begin{proof}
Any non-empty word $Q$ can be written of the form $Q \equiv \prod_{i=1}^{w} \ol{T}_{i} R_{i} T_i$ (by the theorem of Dyck). Let the vector $(\len{T_1}, \len{T_2}, \dots, \len{T_w})$ be the vector of the lengths of the words $T_1, T_2, \dots, T_w$, and suppose its index is $\cJ$. 

Property $k$ ($k = 1, 2, \dots, 10$) of Theorem~1 will follow from Lemma~3 and the proof of the following disjunction: either the representation $Q \equiv \prod_{i=1}^{w} \ol{T}_{i} R_{i} T_i$ satisfies Property~$k$; or there is some representation $Q \equiv \prod_{i=1}^{w} \ol{T}'_{i} R'_{i} T'_i$, such that the index of the lengths of the words $T_1', T_2', \dots, T_t'$ is less than the index $\cJ$. 

\setcounter{property}{0}
\begin{property}
For any $i = 1, 2, \dots, w$ the word $\ol{T}_{i} R_{i} T_i$ is reduced. 

The words $\ol{T}_{i}$, $R_{i}$, and $T_i$ are reduced. If $\ol{T}_{i} R_{i} T_i$ is a reducible word, then $\ol{T}_i \gre \ol{T}_i' a^\varepsilon$, $R_i \gre a^{-\varepsilon}X$, and $T_i \gre a^{-\varepsilon} T_i'$, from which we have $\ol{T}_{i} R_{i} T_i \equiv \ol{T}_i' X a^{-\varepsilon} T_i'$, where $Xa^{-\varepsilon}$ is a reduced word. In this way, we get in the free group the equality $Q \equiv [1, i-1]\ol{T}_i' X a^{-\varepsilon} T_i'[i+1, w]$, and
\[
\alpha\{ \len{T_1}, \len{T_2}, \dots, \len{T_w} \} > \alpha \{ \len{T_1}, \len{T_2}, \dots, \len{T_{i-1}}, \len{T_i'}, \len{T_{i+1}}, \dots, \len{T_w}\}. 
\]
In this way, we can write a new representation for $Q$ with a smaller parameter $\alpha$ for the vector of the lengths of the words $T_1, T_2, \dots, T_w$, and consequently, with a smaller index. 
\end{property}

\begin{property}
Let $R_k \gre U\ol{V}$ and $\len{U} > \frac{1}{2} \len{R_k}$. Then neither $\ol{T}_i$, nor $T_i$ $(i = 1, 2, \dots, w)$ contains the word $U$. 

Suppose $\ol{T_i} \gre XUY$ and $\len{U} > \len{V}$. Then 
\[
\ol{T}_i R_i T_i \gre XUYR_i\ol{Y}\ol{U}\ol{X} \equiv XU\ol{V}\ol{X} \cdot XVYR_i\ol{Y}\ol{V}\ol{X} \cdot XV\ol{U}\ol{X}.
\]
Notice that $V\ol{U}$ is a reduced word, and $\len{X} < \len{T_i}$, $\len{XVY} < \len{T_i}$. Hence, we can write a new presentation for $Q$ with a smaller parameter $\alpha$ for the vector of the lengths of the words $T_1, T_2, \dots, T_w$. 
\end{property}

\begin{property}
If $\ol{R}_k$ and $R_i$ are not mutually inverse, and $R_k \gre ZXY$, $R_i \gre ZK$, $\ol{T}_i \gre NY$, then $\len{X} \geq \len{Y}$. 

Suppose $\len{X} < \len{Y}$. Then 
\[
\ol{T}_i R_i T_i \gre NYZK\ol{Y}\ol{N} \equiv NYZX\ol{N} \cdot N\ol{X}KZX\ol{N} \cdot N\ol{X}\ol{Z}\ol{Y}N.
\]
Notice that $YZK$, $\ol{X}\ol{Z}\ol{Y}$ and $KZ$ are defining words, and $\len{N} < \len{T_i}$, $\len{NX} < \len{T_i}$. 
\end{property}

\begin{property}
If $R_k \gre BAX$ and $R_s \gre \ol{X}CD$ are not mutually inverse, then in $\ol{T}_i$ there is a word $AC$, such that $\len{B} + \len{D} \geq \len{A}+\len{C}$. 

Suppose $\len{B} + \len{D} < \len{A}+\len{C}$. Then 
\begin{align*}
\ol{T}_i R_i T_i &\gre UACVR_i \ol{V}\ol{C}\ol{A}\ol{U} \\
&\equiv UAXV\ol{U} \cdot U\ol{B}\ol{X}CDB\ol{U} \cdot U\ol{B}\ol{D}VR_i\ol{V}DB\ol{U} \cdot U\ol{B}\ol{D}\ol{C}XB\ol{U} \cdot U\ol{B}\ol{X}\ol{A}\ol{U}.
\end{align*}
Notice that $AXB$, $\ol{X}CD$, $\ol{D}\ol{C}X$, $\ol{B}\ol{X}\ol{A}$, are defining words, and $\len{U} < \len{T_i}$, $\len{UB} < \len{T_i}$, $\len{UBDV}<\len{T_i}$. 
 
\end{property}

\begin{property}
If $R_k \gre BAX$ and $R_s \gre \ol{X}CDE$ are not mutually inverse, if $\ol{C}X\ol{E}\ol{D}$ and $R_i \gre DK$ are not mutually inverse, and $\ol{T}_i \gre NAC$, then $\len{E} + \len{B} \geq \len{C} + \len{A}$. 

Suppose $\len{E} + \len{B} < \len{C} + \len{A}$. Then 
\begin{align*}
\ol{T}_i R_i T_i &\gre NACDK\ol{C}\ol{A}\ol{N} \\ 
&\equiv NAXB\ol{N}\cdot N\ol{B}\ol{X}CDEB\ol{N}\cdot  N\ol{B}\ol{E}KDEB\ol{N} \cdot N\ol{B}\ol{E}\ol{D}\ol{C}XB\ol{N} \cdot N\ol{B}\ol{X}\ol{A}\ol{N}.
\end{align*}
Notice that $AXB$, $\ol{X}CDE$, $KD$, $\ol{E}\ol{D}\ol{C}X$, and $\ol{B}\ol{X}\ol{A}$ are all defining words, and $\len{N} + \len{T_i}$, $\len{NB} <\len{T_i}$, and $\len{NBE} <\len{T_i}$. 
\end{property}

\begin{property}
Suppose $R_k \gre U\ol{V}$ and $\len{U} \geq \frac{13}{22} \len{R_k}$. Then the word $U$ can be found inside the word $\ol{T}_iR_iT_i \: (i = 1, 2, \dots, w)$ only as a subword of the word $R_i$. 

Property~$6$ follows from Property~$3$.
\end{property}

\begin{property}
Suppose that $R_s \gre A\ol{X}Y$ and $R_k \gre \ol{Y}\ol{Z}\ol{U}D$ are not mutually inverse. Suppose that $\len{A} < \frac{4}{11}\len{R_s}$, $\len{D} < \frac{4}{11} \len{R_k}$, and $\len{U} < \frac{2}{11}\len{R_k}$. Then $UZK$ is not a subword of any $\ol{T}_iR_iT_i \: (i = 1, 2, \dots, w)$ such that $T_i \gre ZXT_i'$, $R_i \gre R_i' U$, and $\len{U} < \frac{2}{11}\len{R_i}$. 

Suppose the opposite. Then $\len{XY} > \frac{7}{11} \len{R_s}$, $\len{ZUY} > \frac{7}{11} \len{ R_k}$, $\len{X} > \frac{5}{11} \len{R_s}$, $\len{ZU} > \frac{5}{11} \len{R_k}$, and $\len{Z} > \frac{3}{11}\len{R_k}$. There are two cases to consider. 
\begin{enumerate}
\item Suppose $\len{R_k} \leq \len{R_s}$. Then 
\[
\len{Z} + \len{X} > \frac{3}{11}\len{R_k} + \frac{5}{11}\len{R_s} \geq \frac{4}{11}\len{R_k} + \frac{4}{11}\len{R_s} > \len{A} + \len{D}.
\]
But this contradicts Property~5. 
\item Suppose $\len{R_k} - \len{R_s} = \delta > 0$. Then 
\begin{align*}
\len{Y} &< \frac{2}{11}\len{R_s} = \frac{2}{11}\left( \len{R_k} - \delta \right). \\
\len{Z} &> \frac{3}{11} \len{R_k} + \frac{2}{11} \delta. \\
\len{Z} + \len{X} &> \frac{3}{11} \len{R_k} + \frac{2}{11}\delta + \frac{5}{11}\len{R_s} = \frac{4}{11}\len{R_k} + \frac{1}{11}\delta + \frac{4}{11}\len{R_s} > \len{A} + \len{D}.
\end{align*}
But this contradicts Property~5. 
\end{enumerate}
\end{property}

\begin{property}
If $R_p \gre XYZ$, $R_k \gre U\ol{Y}V$, $p < k$ and for some way of reducing $\prod_{i=1}^{w} \ol{T}_{i} R_{i} T_i$ the subword $Y$ of the word $R_p$ is reduced by the subword $\ol{Y}$ of the word $R_k$, then the word $ZXY$ and $\ol{Y}VU$ are not mutally inverse. 

Suppose $ZXY$ and $\ol{Y}VU$ are mutually inverse. Then $UZXV \equiv 1$. We have 
\[
[p+1, k-1] \equiv \ol{T}_p \ol{Z}\ol{U}T_k,
\]
and thus we have 
\[
[p, k] \equiv \ol{T}_p XVT_k \equiv \ol{T}_p\ol{Z}\ol{U}T_k \cdot \ol{T}_k UZXVT_k \equiv [p+1, k-1].
\]
\end{property}

\begin{property}
If $R_p \gre XYZ$, $\ol{T}_{k} \gre U\ol{Y}V$, $p< k$ and for some way of reducing $\prod_{i=1}^{w} \ol{T}_{i} R_{i} T_i$ the subword $Y$ of the word $R_p$ reduces with the subword $\ol{Y}$ of the word $\ol{T}_k$, then $\len{T_p} \leq \len{U}$, and $\len{T_p} < \len{T_k}$. 
\[
\dots \ol{T}_p X\tm{MA}YZT_p \dots U\ol{Y\tm{MB}}VR_k \ol{V} Y \ol{U}
\dla{0.8}{}{MA}{MB}
\]\vspace{0.2cm}\\
We have $[p+1, k-1] \equiv \ol{T}_p \ol{Z} \ol{U}$, and 
\begin{align*}
[p, k] &\equiv \ol{T}_p XVR_kT_k \\
&\equiv \ol{T}_p \ol{Z}\ol{U} \cdot UZXY\ol{U} \cdot U\ol{Y}VR_kT_k \\
&\equiv [p+1, k-1] \: UZXY\ol{U} \cdot \ol{T}_k R_k T_k.
\end{align*}
Notice that $ZXY$ is a defining word, from which we have $\len{U} \geq \len{T_p}$, and since $Y$ is a non-empty word, thus $\len{T_k} > \len{T_p}$. 
\end{property}

\begin{property}
If $T_p \gre AB$, $\ol{T}_j \gre CD$, where $p<j$, and for some way of reducing $\prod_{i=1}^{w} \ol{T}_{i} R_{i} T_i$ the word $B[p+1, j-1]C$ is absorbed, without reducing anything, then either $[p+1, j-1] \equiv 1$, or $p+1 = j$, or there is some $k$, with $p < k < j$, such that $\len{T_k} \geq \len{T_p}$. 

We have $[p+1, j-1] \equiv \ol{B}\ol{C}$, and that 
\begin{align*}
[p, j] &\equiv \ol{B}\ol{A}R_pADR_j\ol{D}\ol{C} \\
&\equiv \ol{B}\ol{C} \cdot C\ol{A}R_pA\ol{C} \cdot CDR_j\ol{D}\ol{C} \\
&\equiv \ol{B}\ol{A}R_pAB \cdot \ol{B}DR_j\ol{D}B\cdot \ol{B}\ol{C}.
\end{align*}
If $\len{B}>\len{C}$, then we consider the penultimate presentation; if $\len{B}<\len{C}$, then we consider the final presentation. If $\len{B} = \len{C}$ and $\len{T_k} < \len{T_p}$ for all $k$, where $p < k < j$, then the presentation 
\[
[p+1, j-1] \: C\ol{A}R_pA\ol{C}\ol{T}_jR_jT_j
\]
has the same parameter $\alpha$ as the presentation $[p, j]$, but a smaller parameter $\beta$, i.e. it has smaller index. 
\end{property}
\end{proof}

\section{Solving the identity problem for $K_{2/11}$-groups}

\begin{definition}
Consider the word $W \gre \prod_{i=1}^w \ol{T}_i R_i T_i$. The number of factors of the form $\ol{T}_q R_q T_q$ in the word $W$, i.e. the number $w$, will be called the $L$-\textit{length} of the word $W$, and will be denoted by $L[W]$. 
\end{definition}

\begin{definition}
Suppose that when shortening the word $[i, j]$ there is some $R_k$, where $i \leq k \leq j$, such that when shortening with the subwords of $[i, j]$ this word loses fewer letters than $\frac{9}{22}$ times its length. Then the word $X$, which remains of $R_k$ as the result of reducing, will be called a \textit{major trace} of the word $[i, j]$.

Suppose that when reducing the word $[i, j]$ there are words $R_{k_1}$ and $R_{k_2}$, where $i \leq k_1 < k_2 \leq j$, such that the two words cancel each other. Furthermore, suppose that each of these words, when reduced with the subwords of $[i, j]$ (not counting the reductions between $R_{k_1}$ and $R_{k_2}$) has fewer letters than $\frac{4}{11}$ times its length removed. Then the result of reducing $R_{k_1}$ will produce a subword $X_1$, and of $R_{k_2}$ a subword $X_2$. Then the word $X_1X_2$ will be called a \textit{minor trace} of the word $[i, j]$. 

Suppose that when reducing the word $[i, j]$, where 
\[
i = s_0 < z_1 < s_1 < z_2 < s_2 < \dots < z_{t-1} < s_{t-1} < z_t < s_t = j, \quad t > 3,
\]
we have that the word $R_{s_0}$ reduces with the word $R_{s_1}$, the word $R_{s_1}$ reduces with the word $R_{s_2}, \dots$, and the word $R_{s_{t-1}}$ reduces with the word $R_{s_t}$; and moreover, any $R_{s_k} (k = 0, 1, 2, \dots, t)$ can be reduced (but not necessarily cancelled) to the right only with $R_{z_{k+1}}$ and $\ol{T}_{s_{k+1}}$, and to the left only with $R_{z_k}$ and $T_{s_{k-1}}$. Suppose further that when reducing the word $[i, j]$ there is no major or minor trace. Then the word which remains after reducing $[i,j]$ will be called a \textit{full trace} of the word $[i, j]$. 
\end{definition}

\begin{definition}
When reducing some word $S_1 V_1 S_2 V_2 S_3$, we will say that the highlighted subword $V_1$ reduces with the highlighted subword $V_2$ if we have that $V_1 \gre AXB$, $V_2 \gre C\ol{X}D$, where $X$ is a non-empty word, $BS_2C \equiv 1$, and that when carrying out this reduction, the word $BS_2C$ is first reduced, and then $X$ is reduced with $\ol{X}$. We will denote this reduction by the arc
\[
S_1 \tm{MA}V_1S_2V\tm{MB}_2 S_3
\dla{0.4}{}{MA}{MB}
\]
This definition naturally defines to a larger number of reductions, i.e. we have as an example the word with two arcs
\[
S_1 \tm{MA}V_1S_2V\tm{MB}_2 S_3 V\tm{MC}_3 S_4
\dla{0.4}{}{MA}{MB}
\dla{0.8}{}{MA}{MC}
\]\\
which denotes that when reducing the word $S_1 V_1 S_2 V_2 S_3 V_3 S_4$, the subword $V_1$ reduces with the highlighted word $V_2$, and the highlighted subword $V_1$ reduces with the highlighted subword $V_3$. 
\end{definition}

We will consider a number of statements regarding possible ways of reducing a word of the form $W \gre \prod_{i=1}^w \ol{T}_i R_i T_i$, where the product $\prod_{i=1}^w \ol{T}_i R_i T_i$ satisfies Properties~$1-10$, and such that $L[W] \leq m$, where $m$ is some positive integer. 

These statements will concern reductions of $R_i$ on the right. For brevity, we will not formulate the dual statements for reducing $R_i$ on the left, but all these statements are implied, and the proofs of these statements are entirely symmetrical. 

$\fA_m.$ If for any way of reducing $W$ 
\[
\dots \ol{T}_i\tm{MA}R_iT_i \dots \ol{T}_jR\tm{MB}_jT_j \dots \ol{T}_kR\tm{MC}_kT_k
\dla{0.8}{}{MA}{MB}
\dla{1.2}{}{MA}{MC}
\]\vspace{0.2cm}\\
then when reducing the word $[i+1, k-1]$ the word $R_j$ reduces with at most one $R$ to the left and one $R$ to the right. 

$\fB_m$. If for every way of reducing $W$ 
\[
\dots \ol{T}_i \tm{MA} R_i T_i \dots \ol{T}_j R\tm{MB}_j \dl{0.8}{MA}{MB} T_j \dots \ol{T\tm{MB}}_k R_k\dl{1.2}{MA}{MB} T_k \dots
\]
\vspace{0.3cm}\\
then when reducing $[i+1, k-1]$, the word $R_j$ reduces with at most one $R$. 

$\fC_m$. We cannot have the reduction
\[
\dots \ol{T}_i \tm{MA} R_i T_i \dots \ol{T}_j R\tm{MB}_j \dl{0.8}{MA}{MB} T_j \dots \ol{T\tm{MB}}_k R_k\dl{1.2}{MA}{MB} T_k \dots \ol{T}_p R\tm{MB}_p\dl{1.6}{MA}{MB} T_p \dots
\]
\vspace{0.4cm}\\
when reducing the word $W$ in any way; and we cannot have  the reduction
\[
\dots \ol{T}_i \tm{MA} R_i T_i \dots \ol{T}_j R\tm{MB}_j \dl{0.8}{MA}{MB} T_j \dots \ol{T\tm{MB}}_k R_k\dl{1.2}{MA}{MB} T_k \dots \ol{T\tm{MB}}_p \dl{1.6}{MA}{MB} R_p T_p \dots
\]
\vspace{0.4cm}\\
when reducing the word $W$ in any way. 

$\fD_m$. We cannot have the reduction
\[
\dots \ol{T}_i \tm{MA}R_i T_i \dots \ol{T}_j R_j  T\tm{MB}_j \dl{0.8}{MA}{MB}\dots
\]
\vspace{0cm}\\
when reducing the word $W$ in any way. 

$\fE_m$. We cannot have 
\[
\dots \ol{T}_i \tm{MA}R_i T_i \dots \ol{T\tm{MB}}_j\dl{0.8}{MA}{MB} R_j T_j \dots [k\tm{MB}, p]\dl{1.2}{MA}{MB} \dots
\]
\vspace{0.2cm}\\
when reducing the word $W$ in any way. 

$\fF_m$. We cannot have 
\[
\dots \ol{T}_i \tm{MA}R_i T_i \dots \ol{T}_j R\tm{MB}_j T_j \dots \ol{T}_k \tm{MC}R_k T_k \dots \ol{T}_p R\tm{MD}_p T_p \dots
\dla{0.8}{}{MA}{MB}
\dla{0.8}{}{MC}{MD}
\dla{1.6}{}{MA}{MD}
\]
\vspace{0.2cm}\\

$\fG_m$. If $R_i$ reduces with $\ol{T}_j R_j$, where $i < p < j$, then 
\[
\len{T_p} < \len{T_j} \qquad \text{and} \qquad \len{T_p} < \len{T_i}.
\]
\[
\dots \ol{T}_i \tm{MA}R_i T_i \dots \tm{MC}\ol{T\tm{MB}\dl{0.8}{MA}{MB}}_j R_j  T\tm{MD}_j \dl{0.6}{MC}{MD}\dots
\]
\vspace{0cm}\\

$\fH_m$. For all ways of reducing the word $W$, there is a trace in the resulting word.

\begin{lemma}
$(\fA_m, \fB_m, \fC_m, \dots, \fH_m) \implies \fA_{m+1}$.
\end{lemma}
\begin{proof}
Suppose that for any way of reducing the word $W$, where we have $L[W] \leq m+1$, there are the following reductions. 
\[
\dots \ol{T}_i R_i T_i \dots \ol{T}_j R_j T_j \dots \ol{T}_k R_k T_k \dots
\]
Then $[i+1, k-1] \not\equiv 1$, and by $\fH_m$ for any reduction of $[i+1, k-1]$ there is a trace of this word. 

\noindent\textbf{Case I.} When reducing $[i+1, k-1]$, there is a major trace. \\

\noindent\textbf{Case IA.} When reducing $[i+1, k-1]$ there is a $R_v$ which reduces with at most one $R$. If $v < j$, then the word $R_v$ cannot reduce with $\ol{T}_kR_k$ and by Property~6, cannot be entirely reduced when reducing the word $[i, k]$. If $v = j$, then statement $\fA_{m+1}$ holds true. If $v>j$, then the word $R_v$ cannot reduce with $T_i$, and therefore cannot be entirely reduced when reducing the word $[i, k]$.\\ 

\noindent\textbf{Case IB.} When reducing $[i+1, k-1]$ there will be some $R_v$, which reduces with $R_s$ and $R_t$, where $s<t$. \\

\noindent\textbf{Case IB1.} $v<j$. The word $R_v$ cannot reduce with $\ol{T}_kR_k$, and by Property~6 cannot be entirely reduced when reducing the word $[i, k]$. \\

\noindent\textbf{Case IB2.} $v=j$. If $s<t<j$, then the resulting reductions contradict the dual of $\fC_m$. If $j<s<t$, then to be completely reduced in $[i, k]$ (by Property~6) $R_j$ must reduce with $\ol{T}_kR_k$, which contradicts $\fC_m$. If $s<j<t$, then statement $\fA_{m+1}$ holds. \\

\noindent\textbf{Case IB3.} $v>j$. If $s<t<v$, then for the word $R_v$ to be entirely reduced when reducing $[i, k]$, the word $R_v$ needs to reduce with $\ol{T}_kR_k$, which contradicts $\fC_m$. 

Suppose $s<v<t$. The word $R_v$ must reduce with $R_i$, from which by $\fF_m$ we have $s=j$. Furthermore $R_v$ must reduce with $\ol{T}_k$ and $R_k$ (the reduction of $R_v$ and $R_k$ can be empty). 

In this way, we have the following reduction (for convenience, some indices are changed). 
\[
\dots \ol{T}_i  \tm{MA}R_i T_i \dots \ol{T}_s \tm{MC}R\tm{MB}_s T_s \dots \ol{T}_j \tm{ME}R\tm{MD}_j T_j \dots \ol{T}_t R\tm{MF}_t T_t \dots \ol{T}\tm{MG}_k  R\tm{MH}_k T_k \dots 
\dla{0.8}{A}{MA}{MB}
\dla{0.8}{B}{MC}{MD}
\dla{1.2}{C}{MA}{MD}
\dla{0.8}{D}{ME}{MF}
\dla{1.2}{E}{ME}{MG}
\dla{1.7}{F}{ME}{MH}
\dla{2.4}{G}{MA}{MH}
\]
\vspace{1.2cm}\\
Suppose that the words $R_i$ and $R_s$ reduce by the word $A$ (this means that $R_i \gre X_1 A X_2$, and $R_s \gre Y_1 \ol{A}Y_2$ and for the given reduction the subword $A$ of $R_i$ reduces with the subword $\ol{A}$ of $R_s$), and that $R_s$ and $R_j$ reduce. \\

\noindent\textbf{Case IB3$\alpha$.} The word $R_t$ does not reduce when reducing the word $[j+1, k-1]$. 

In this case, to be entirely reduced when reducing $[j,k]$, $R_t$ will be reduced with $T_j$ (say with the word $H$) and with $\ol{T}_k$ (say with the word $J$). Let $T_i$ and $R_s$ reduce (if they reduce) by the word $K$, and the words $R_s$ and $\ol{T}_j$ (if they reduce) by the word $L$. 
\[
\dots \ol{T}_i \tm{MA}R_i \tm{MB}T_i \dots \ol{T}_s \tm{MD}R\tm{MC}_s T_s \dots \ol{T\tm{ME1}}_j \tm{MF}R\tm{ME}_j \tm{MG}T_j \dots \ol{T}_t \tm{MI}R\tm{MH}_t T_t\dots \ol{T\tm{MJ}}_k R\tm{MK}_k T_k\dots
\dla{0.6}{K}{MB}{MC}
\dla{0.6}{L}{MD}{ME1}
\dla{1.0}{A}{MA}{MC}
\dla{1.0}{B}{MD}{ME}
\dla{2.0}{}{MA}{ME}
\dla{1.0}{D}{MF}{MH}
\dla{0.6}{H}{MG}{MH}
\dla{0.6}{J}{MI}{MJ}
\dla{2.0}{}{MF}{MK}
\dla{1.6}{}{MF}{MJ}
\dla{3.0}{}{MA}{MK}
\]
\vspace{1.4cm}\\
By Property~1, we have $\len{H} \geq \frac{4}{11}\len{R_t}$. 

We will prove that $\len{L} < \frac{2}{11}\len{R_s}$ and $\len{L} < \frac{2}{11}\len{R_t}$. Indeed, otherwise $R_s$ and $\ol{R}_t$ have a common piece, containing at least $\frac{2}{11}$ letters of one of them, and therefore $R_s$ comes from $\ol{R}_t$ by a cyclic permutation. Thus,
\[
R_s \gre X_1 \ol{a} \Phi Y_1, \qquad R_t \gre X_2 \ol{\Phi} \ol{b} Y_2,
\]
where the word $\Phi$ is the shared piece of the words $R_s$ and $\ol{R}_t$, $a$ is the first letter of $R_j$, and $b$ is the final letter of $R_j$. As $\Phi Y_1 X_1 \ol{a} \gre \Phi \ol{X}_2 \ol{Y}_2 b$, we have $\ol{a} \gre b$, and hence the word $R_j$ is cyclically reduced.

By the fact that $\len{L} < \frac{2}{11}\len{R_s}$, and Property~6, to be entirely reduced in $[i, j]$ the word $R_s$ must be reduced via reduction with some $R_p$, where $i < p < j$. 

Suppose $p<s$. Then by statement $\fG_m$ and Property~10 we have $s+1 = j$. \\
\[
\dots \ol{T}_i \tm{MA}R_i \tm{MB}T_i \dots \ol{T}_p \tm{MD}R\tm{MC}_p T_p \dots \tm{MF}\ol{T\tm{ME}}_s \tm{MH}R\tm{MG}_s T_s \ol{T\tm{MI}}_j \tm{MK}R\tm{MJ}_j \tm{ML}T_j \dots \ol{T}_t R\tm{MM}_t T_t\dots \ol{T\tm{MN}}_k R\tm{MO}_k T_k\dots
\dla{2.0}{A}{MA}{MG}
\dla{0.6}{N}{MB}{MC}
\dla{1.2}{K}{MB}{MG}
\dla{0.6}{}{MD}{ME}
\dla{0.9}{}{MD}{MG}
\dla{0.6}{}{MH}{MI}
\dla{0.9}{}{MH}{MJ}
\dla{2.4}{}{MA}{MJ}
\dla{0.9}{}{MK}{MM}
\dla{2.4}{}{MK}{MO}
\dla{0.6}{H}{ML}{MM}
\dla{1.5}{}{MK}{MN}
\dla{4.0}{}{MA}{MO}
\]
\vspace{2.3cm}\\
The word $R_p$ must be reduced via reduction with $T_i$ and $\ol{T}_s$, as by $\fH_m$ the word $R_p$, when reducing $[i+1, s-1]$, can be reduced with at most one $R$. 

Suppose $R_p$ and $\ol{T}_s$ reduce with the word $P$, and $T_i$ and $R_p$ with the word $N$. We will prove that $\len{P} < \frac{2}{11}\len{R_p}$ and $\len{P} < \frac{2}{11}\len{R_t}$. Indeed, otherwise $R_p$ and $\ol{R}_t$ have a common piece, containing not less than $\frac{2}{11}$ letters of one of them. $R_p \gre X_3 \ol{c} \Phi_1 Y_3$, $R_t \gre X_4 \ol{\Phi}_1 \ol{d} Y_4$, where $c$ is the first letter of $R_s$, and $d$ is the final letter of $R_s$. As $\Phi_1 Y_3 X_3 \ol{c} \gre \Phi_1 \ol{X}_4 \ol{Y}_4 d$, we have $\ol{c} \gre d$, and the word $R_s$ is cyclically reduced. 

In this way, we have that $\len{ NU} > \frac{7}{11}\len{R_p}$ and $\len{UKA} > \frac{7}{11}\len{R_s}$, which contradicts Property~7. 

Suppose $p > s$. Then by statement $\fG_m$ and Property~10 we have $i+1 = s$. 

\[
\dots \ol{T}_i \tm{MA}R\tm{MB}_i \tm{MC}T_i \ol{T}_s \tm{MD}R\tm{ME}_s \tm{MF}T_s \dots \ol{T}_p \tm{MG}R\tm{MH}_p T_p \dots \ol{T\tm{MJ}}_j \tm{MK}R\tm{ML}_j \tm{MM}T_j \dots \ol{T}_t \tm{MN}R\tm{MO}_t T_t\dots \ol{T\tm{MP}}_k R\tm{MQ}_k T_k\dots
\dla{0.9}{A}{MA}{ME}
\dla{2.4}{}{MA}{ML}
\dla{4.0}{}{MA}{MQ}
\dla{0.6}{K}{MC}{ME}
\dla{1.0}{U}{MD}{MH}
\dla{1.4}{L}{MD}{MJ}
\dla{1.8}{B}{MD}{ML}
\dla{0.6}{P_1}{MF}{MH}
\dla{0.6}{N_1}{MG}{MJ}
\dla{0.9}{}{MK}{MO}
\dla{1.4}{}{MK}{MP}
\dla{1.8}{}{MK}{MQ}
\dla{0.6}{H}{MM}{MO}
\dla{0.6}{}{MN}{MP}
\]
\vspace{2.0cm}\\

The word $R_p$ must reduce with $T_s$ and $\ol{T}_j$, as by $\fB_m$ the word $R_p$, when reducing with $[s+1, j-1]$, can reduce with at most one $R$. 

Suppose $T_s$ and $R_p$ reduce with the word $P_1$, and $R_p$ and $\ol{T}_j$ with the word $N_1$. We will prove that $\len{N_1} < \frac{2}{11} \len{R_p}$ and $\len{N_1} < \frac{2}{11}\len{R_t}$. Indeed, otherwise $R_p$ and $\ol{R}_t$ have a common piece, containing not less than $\frac{2}{11}$ letters of one of them. $R_p \gre X_s \ol{a}\Phi_2Y_s$, $R_t \gre X_6 \ol{\Phi}_2 \ol{b} Y_6$, where $R_s \gre X_7abY_7$. As $\Phi_2 Y_5 X_5 \ol{a} \gre \Phi \ol{X}_6 \ol{Y}_6 b$, we have $\ol{a} \gre b$, i.e. the word $R_s$ is reduced. 

In this way, $\len{UP_1} > \frac{7}{11}\len{R_p}$, which contradicts Property~6. 

If when reducing $[i+1, j-1]$, the word $R_s$ reduces with two $R$, then we have a contradiction much the same way as in the case just considered. \\

\noindent\textbf{Case IB3$\beta$.} When reducing $[j+1, k-1]$ the word $R_t$ reduces with $R_p$, where $p <t$. 

By statement $\fG_m$ and Property~10 we have $t+1 = k$. 
\[
\dots \ol{T}_i \tm{MA}R_i \tm{MB}T_i \dots \ol{T}_s \tm{MC}R\tm{MD}_s T_s \dots \ol{T\tm{ME}}_j \tm{MF}R\tm{MG}_j \tm{MH}T_j \dots \ol{T}_p \tm{MI}R\tm{MJ}_p T_p \dots \ol{T\tm{MK}}_t \tm{ML}R\tm{MM}_t T_t \ol{T\tm{MN}}_k R\tm{MO}_k T_k\dots
\dla{0.9}{A}{MA}{MD}
\dla{1.6}{}{MA}{MG}
\dla{4.0}{}{MA}{MO}
\dla{0.6}{K}{MB}{MD}
\dla{0.6}{L}{MC}{ME}
\dla{0.9}{B}{MC}{MG}
\dla{1.6}{}{MF}{MM}
\dla{2.0}{}{MF}{MN}
\dla{2.4}{}{MF}{MO}
\dla{0.6}{P}{MH}{MJ}
\dla{1.4}{H}{MH}{MM}
\dla{0.6}{Q}{MI}{ML}
\dla{0.9}{S}{MI}{MM}
\dla{0.6}{J}{ML}{MN}
\]
\vspace{2.6cm}\\

Suppose $T_i$ and $R_s$ reduce with the word $K$; $R_s$ and $\ol{T}_j$ with the word $L$; $T_j$ and $R_p$ with the word $P$; and $R_p$ and $\ol{T}_t$ with the word $Q$. \\

\noindent\textbf{Case IB3$\beta$1.} When reducing $[i+1, j-1]$ the word $R_s$ does not reduce. 

Then $\len{L} > \frac{2}{11}\len{R_s}$, $\len{HS} \geq \frac{4}{11}\len{R_t}$, as otherwise we do not have Property~7, from which $\len{H} > \frac{2}{11}\len{R_t}$, and we get that $R_j$ is cyclically reduced. \\

\noindent\textbf{Case IB3$\beta$2.} When reducing $[i+1, j-1]$ the word $R_s$ reduces with $R_z$, where $z < s$. 

By statement $\fG_m$ and Property~10 we have $s+1=j$.
\[
\dots \ol{T}_i \tm{MA}R\tm{MB}_i T_i \dots \ol{T}_z \tm{MC}R_z T_z \dots \ol{T}_s \tm{MD}R\tm{ME}_s T_s \ol{T\tm{MF}}_j \tm{MG}R\tm{MH}_j \tm{MI}T_j \dots \ol{T}_p \tm{MJ}R\tm{MK}_p T_p \dots \ol{T\tm{ML}}_t \tm{MM}R\tm{MN}_t T_t \ol{T\tm{MO}}_k R\tm{MP}_k T_k
\dla{1.6}{A}{MA}{ME}
\dla{2.4}{}{MA}{MH}
\dla{1.3}{K}{MB}{ME}
\dla{0.8}{X}{MC}{ME}
\dla{0.4}{L}{MD}{MF}
\dla{0.8}{B}{MD}{MH}
\dla{1.6}{}{MG}{MN}
\dla{2.0}{}{MG}{MO}
\dla{2.4}{}{MG}{MP}
\dla{0.7}{P}{MI}{MK}
\dla{1.4}{H}{MI}{MN}
\dla{0.7}{Q}{MJ}{ML}
\dla{1.0}{S}{MJ}{MN}
\dla{0.7}{J}{MM}{MO}
\]
\vspace{1.2cm}\\
$\len{HS} \geq \frac{4}{11}\len{R_t}$, as otherwise we do not have Property~7, from which we have $\len{H} > \frac{2}{11}\len{R_t}$. Then $\len{L} < \frac{2}{11}\len{R_s}$, $\len{L} < \frac{2}{11}\len{R_t}$, as otherwise $R_j$ is cyclically reduced. 

Consider the word $[i, s][j+1, k]$. 

\[
\dots \ol{T}_i \tm{MA}R_i \tm{MB}T_i \dots \ol{T}_z \tm{MC}R_z T_z \dots \ol{T}_s \tm{MD}R\tm{ME}_s \tm{MF}T_s \dots \ol{T}_p \tm{MG}R\tm{MH}_p T_p\dots \ol{T\tm{MI}}_t \tm{MJ}R\tm{MK}_t T_t \ol{T\tm{ML}}_k R_k T_k\dots
\dla{1.6}{A}{MA}{ME}
\dla{1.2}{K}{MB}{ME}
\dla{0.9}{X}{MC}{ME}
\dla{1.6}{L}{MD}{MK}
\dla{1.4}{H}{MF}{MK}
\dla{0.6}{P}{MF}{MH}
\dla{0.6}{Q}{MG}{MI}
\dla{0.9}{S}{MG}{MK}
\dla{0.6}{J}{MJ}{ML}
\]
\vspace{0.5cm}\\

As $\len{H} > \len{L}$, we have that $R_t$ reduces with $R_s$ and $T_s$. $LT_s \gre T_j$, and hence the word $T_s$ reduces with $R_p$ by the word $P$. After reducing $[i,s][j+1, k]$, there remains of $R_s$ a word $B$ such that $\len{B} < \frac{2}{11}\len{R_s}$. 

We will prove, that the word $[i, s][j+1, k]$ satisfies all ten properties of Theorem~1. 

Properties 1-5 of the word $[i, s][j+1, k]$ are satisfied, as these properties concern the subword $\ol{T}_i R_i T_i$, which in the word $[i, s][j+1, k]$ are the same as those of $[i, k]$. Properties~6-7 of the word $[i, s][j+1, k]$ are satisfied, as these properties are consequences of Properties~3 and 5. 

Property~8 holds for the subwords $[i, s]$ and $[j+1, k]$, and hence it holds for the same subwords in the word $[i, k]$. Furthermore, in the word $[i, s][j+1, k]$ there is only one reduction: the word $R_s$ reduces with the word $R_t$ by the word $L$, but, as we have proved, $\len{L} < \frac{2}{11}\len{R_s}$ and $\len{L} < \frac{2}{11}\len{R_t}$. 

Property~9 is also enough to check for a single reduction: the reduction of $R_t$ and $T_s$. 

Let $T_j \gre HU$ and that in the word $[i,k]$ $T_j$ reduces with $R_t$ by the word $H$. Then by Property~9 we have $\len{U} \geq \len{T_t}$. Suppose that $T_s$ and $R_t$ intersect in $[i, s][j+1, k]$ by the word $H_1$. Then $T_s \gre H_1U$, but $\len{U} \geq \len{T_t}$, from which $\len{T_s} > \len{T_t}$. 

Property~10 follows, as $s+1 = j$ and the word $[k+1, p-1]\ol{T}_p\ol{P}$, both in the word $[i, k]$ and in the word $[i, s][j+1, k]$, is absorbed from the left by the same word $U$. 

Therefore, in the word $[i, s][j+1, k]$, all properties $\fA_m, \fB_m, \dots, \fH_m$ hold. 

Suppose that the words $R_z$ and $\ol{T}_s$ reduce by the word $Y$, and that $R_z$ does not reduce when reducing the word $[i+1, s-1]$. Then if $\len{Y} \geq \frac{4}{11}\len{R_z}$, then we do not have Property~7; if $\len{Y} \geq \frac{4}{11}\len{R_z}$, then $\len{P} < \frac{2}{11}\len{R_p}$ (as otherwise $R_t$ is reduced), but then $\len{Q} > \frac{5}{11}\len{R_p}$, $\len{J} > \frac{5}{11}\len{R_t}$, $\len{H} < \frac{4}{11}\len{R_t}$, and $\len{PF}<\frac{4}{11}\len{R_p}$, where when reducing $[s+1, t-1]$ the word $R_p$ is reduced by the word $F$. 

Hence, when reducing $[i+1, s-1]$ $R_z$ reduces with some $R$, and we can continue our argument by splitting into a case (which essentially does not differ from \textbf{IB3}). As the length $L$ of the word $W$ is finite, we have our contradiction. \\

\noindent\textbf{Case IB3$\beta$3.} When reducing $[i+1, j-1]$, the word $R_s$ reduces with $R_z$, where $s < z$. 

By statement $\fG_m$ and Property~10 we have $i+1 = s$. 
\[
\dots \ol{T}_i \tm{MA}R_i T_i \ol{T}_s \tm{MB}R\tm{MC}_s \tm{MD}T_s \dots \ol{T}_z \tm{ME}R\tm{MF}_z T_z \dots \ol{T\tm{MG}}_j \tm{MH}R\tm{MI}_j \tm{MJ}T_j \dots \ol{T}_p \tm{MK}R_p T_p\dots \ol{T}_t \tm{ML}R\tm{MM}_t T_t \ol{T\tm{MN}}_k R\tm{MO}_k T_k\dots
\dla{0.6}{}{MA}{MC}
\dla{2.0}{}{MA}{MI}
\dla{0.9}{}{MB}{MF}
\dla{1.4}{L}{MB}{MG}
\dla{1.6}{}{MB}{MI}
\dla{0.6}{}{MD}{MF}
\dla{0.6}{}{ME}{MG}
\dla{1.2}{}{MH}{MM}
\dla{1.6}{}{MH}{MN}
\dla{2.0}{}{MH}{MO}
\dla{1.0}{H}{MJ}{ML}
\dla{0.6}{}{ML}{MN}
\]
\vspace{1.0cm}\\
As we already know, $\len{H} > \frac{2}{11}\len{R_t}$, from which we have $\len{L} <\frac{2}{11}\len{R_s}$ and $\len{L} < \frac{2}{11}\len{R_t}$, i.e. $\len{L} < \len{H}$, and the words $R_z$ and $\ol{T}_j$ reduce. 

Consider the word $[s, j-1][j+1,t]$. 
\[
\dots \ol{T}_s \tm{MA}R_s T_s \dots \ol{T}_z \tm{MB}R_z T_z \dots \ol{T}_p \tm{MC}R_p T_p \dots \ol{T}_t R\tm{MD}_t T_t \dots 
\dla{1.8}{L}{MA}{MD}
\dla{1.2}{H}{MB}{MD}
\dla{0.6}{}{MC}{MD}
\]
\vspace{1.0cm}\\
The words $R_t$ and $R_z$ reduce inside the word $[s, j-1][j+1, t]$. It is easy to show, that the word $[s, j-1][j+1, t]$ satisfies all ten properties of Theorem~1, but $R_t$ reduces on the left with three $R$, which contradicts the dual of $\fC_m$. \\

\noindent\textbf{Case IB3$\beta$4.} When reducing the word $[i+1, j-1]$ the word $R_s$ reduces with $R_z$ and $R_q$, where $z < s < q$. 

\[
\dots \ol{T}_i \tm{MA}R_i \tm{MAA}T_i \dots \ol{T}_z \tm{MB}R_z T_z \dots \ol{T}_s \tm{MC}R\tm{MD}_s T_s \dots \ol{T}_q \tm{ME}R\tm{MF}_q T_q \dots  \ol{T}\tm{MG}_j \tm{MH}R\tm{MI}_j \tm{MJ}T_j \dots \ol{T}_p \tm{MK}R_p T_p\dots \ol{T}_t \tm{ML}R\tm{MM}_t T_t \ol{T\tm{MN}}_k R\tm{MO}_k T_k
\dla{2.4}{}{MA}{MI}
\dla{1.2}{}{MA}{MD}
\dla{1.0}{}{MAA}{MD}
\dla{0.8}{}{MB}{MD}
\dla{0.8}{}{MC}{ME}
\dla{1.0}{}{MC}{MG}
\dla{1.2}{}{MC}{MI}
\dla{2.2}{}{MH}{MO}
\dla{2.0}{}{MH}{MN}
\dla{1.6}{}{MH}{MM}
\dla{1.2}{}{MJ}{MM}
\dla{0.8}{}{MK}{MM}
\dla{0.6}{}{ML}{MN}
\]
\vspace{1.4cm}\\
We have $\len{H} > \frac{2}{11}\len{R_t}$, whence $\len{L} < \frac{2}{11}\len{R_s}$ and $\len{L}< \frac{2}{11}\len{R_t}$, i.e. $\len{L}<\len{H}$, and the word $R_q$ and $\ol{T}_j$ reduce. 

Consider the word $[s, j-1][j+1, t]$.
\[
\dots \ol{T}_s \tm{MA}R_s T_s \dots \ol{T}_q \tm{MB}R_q T_q \dots \ol{T}_p \tm{MC}R_p T_p \dots \ol{T}_t R\tm{MD}_t T_t \dots 
\dla{1.8}{}{MA}{MD}
\dla{1.2}{}{MB}{MD}
\dla{0.6}{}{MC}{MD}
\]
\vspace{1.2cm}\\

As the word $[s, j-1][j+1, t]$ satisfies all ten properties of Theorem~1, we hence have a contradiction to the dual of $\fC_m$. \\

\noindent\textbf{Case IB3$\gamma$.} When reducing $[j+1, k-1]$ the word $R_t$ reduces with $R_p$, where $p>t$. 

By statement $\fG_m$ and Property~10, we have $j+1=t$. 

\[
\dots \ol{T}_i \tm{MA}R_i \tm{MB}T_i \dots \ol{T}_s \tm{MC}R\tm{MD}_s T_s \dots \ol{T\tm{ME}}_j \tm{MF}R\tm{MG}_j \tm{MH}T_j \ol{T}_t \tm{MI}R\tm{MJ}_t T_t \dots \ol{T}_p R\tm{MK}_p T_p\dots \ol{T}\tm{ML}_k R\tm{MM}_k T_k\dots
\dla{0.9}{}{MA}{MD}
\dla{1.8}{}{MA}{MG}
\dla{0.7}{}{MB}{MD}
\dla{0.7}{}{MC}{ME}
\dla{0.9}{}{MC}{MG}
\dla{2.4}{}{MF}{MM}
\dla{2.0}{}{MF}{ML}
\dla{0.9}{}{MF}{MJ}
\dla{0.7}{}{MH}{MJ}
\dla{0.9}{}{MI}{MK}
\dla{1.3}{}{MI}{ML}
\]
\vspace{0.7cm}\\

Here we reach the same contradiction as in Case~\textbf{IB3$\beta$}, but already when considering the word $[i, j-1][t+1, k]$. \\

\noindent\textbf{Case~IC.} When reducing $[i+1, k-1]$ there is some $R_v$ which reduces with more than two $R$. 

In this case, we immediately get a contradiction to $\fC_m$ or the dual of $\fC_m$, as $R_v$, which reduces completely in $[i, k]$ (by Property~6) must reduce with $R_iT_i$ and $\ol{T}_kR_k$. \\

\noindent\textbf{Case~ID.} When reducing $[i+1, k-1]$, there is an $R_v$ which reduces with $\ol{T}_s$ or $T_p$. 

In this case, we immediately get a contradiction to $\fD_m$ or $\fE_m$. \\

\noindent\textbf{Case II.} When reducing $[i+1, k-1]$, there remains a minor trace, i.e. there are words $R_{v_1}$ and $R_{v_2}$ which reduce with each other, and that in this reduction each word loses fewer than $\frac{4}{11}$ letters of its length. \\

\noindent\textbf{Case IIA.} $v_1 < v_2 < j$. 

In this case, neither $R_{v_1}$ nor $R_{v_2}$ can reduce with $\ol{T}_kR_k$, and the minor trace must be completely reduced when reducing with $R_iT_i$. But by Property~7, this is impossible. \\

\noindent\textbf{Case IIB.} $v_1 < v_2 = j$. 
\[
\dots \ol{T}_i \tm{MA}R_i T_i \dots \ol{T}_{v_1} \tm{MB}R_{v_1} T_{v_1} \dots \ol{T}_j R\tm{MC}_j T_j \dots \ol{T}_k R\tm{MD}_k T_k\dots
\dla{1.8}{}{MA}{MD}
\dla{1.2}{}{MA}{MC}
\dla{0.8}{}{MB}{MC}
\]
\vspace{0.4cm}\\

By Property~7, the word $R_j$ should be reduced with $\ol{T}_kR_k$. By $\fC_m$ and the dual of $\fE_m$  $R_j$ cannot reduce on the left. By $\fC_m$ and $\fE_m$ $R_j$ can be reduced on the right only by a single $R$ (other than the fact that $R_j$ reduces with $\ol{T}_kR_k$), but then when reducing $[i+1, k-1]$ by $R_j$ there remains a major trace, and we have already considered this case. \\

\noindent\textbf{Case IIC.} $j = v_1 < v_2$. 
\[
\dots \ol{T}_i \tm{MA}R_i T_i \dots \ol{T}_{j} \tm{MB}R\tm{MBB}_{j} T_{j} \dots \ol{T}_{v_2} R\tm{MC}_{v_2} T_{v_2} \dots \ol{T}_k R\tm{MD}_k T_k\dots
\dla{1.8}{}{MA}{MD}
\dla{0.8}{}{MA}{MBB}
\dla{0.8}{}{MB}{MC}
\]
\vspace{0cm}\\

$R_j$ cannot reduce with $T_{v_2}[v_2+1, k-1]$, as otherwise when reducing $[i+1, k-1]$ by the words $R_j$ and $R_{v_2}$ no trace would remain. By $\fE_m$ and the dual of $\fE_m$, $\fC_m$ and the dual of $\fC_m$, and as no major trace remains of $R_j$, then when reducing $[i+1, k-1]$ the word $R_j$ reduces with $R_s$ and $R_{s_2}$, where $i < s_1 < j < s_2 < v_2$. 
\[
\dots \ol{T}_i \tm{MA}R_i T_i \dots \ol{T}_{s_1} \tm{MB}R_{s_1} T_{s_1} \dots \ol{T}_j \tm{MC}R\tm{MD}_j T_j \dots \ol{T}_{s_2} R\tm{ME}_{s_2} T_{s_2} \dots \ol{T}_{v_2} R\tm{MF}_{v_2} T_{v_2} \dots \ol{T}_k R\tm{MG}_k T_k\dots
\dla{1.8}{}{MA}{MG}
\dla{1.2}{}{MA}{MD}
\dla{1.2}{}{MC}{MF}
\dla{0.8}{}{MB}{MD}
\dla{0.8}{}{MC}{ME}
\]
\vspace{0.5cm}\\

By Property~7 the word $R_{v_2}$ must reduce with $\ol{T}_kR_k$, and in $[i+1, k-1]$ the word $R_{v_2}$ (by $\fC_m, \fE_m$, and $\fF_m$) must reduce with $R_{s_2}$ and $R_{s_3}$, where $v_2 < s_3 < k$. 

\[
\dots \ol{T}_i \tm{MA}R_i T_i \dots \ol{T}_{s_1} \tm{MB}R_{s_1} T_{s_1} \dots \ol{T}_j \tm{MC}R\tm{MD}_j T_j \dots \ol{T}_{s_2} \tm{ME}R\tm{MF}_{s_2} T_{s_2} \dots \ol{T}_{v_2} \tm{MG}R\tm{MH}_{v_2} T_{v_2} \dots \ol{T}_{s_3} R\tm{MI}_{s_3} T_{s_3} \dots \ol{T}_k R\tm{MJ}_k T_k \dots 
\dla{2.4}{}{MA}{MJ}
\dla{1.2}{}{MA}{MD}
\dla{1.2}{}{MC}{MH}
\dla{1.2}{}{MG}{MJ}
\dla{0.6}{}{MB}{MD}
\dla{0.6}{}{MC}{MF}
\dla{0.6}{}{ME}{MH}
\dla{0.6}{}{MG}{MI}
\]
\vspace{0.6cm}\\

$R_j$ and $R_{v_2}$ must be completely absorbed when reducing $[i, k]$, and therefore either $R_j$ reduces with $T_i$ (and has no reductions other than those indicated) or else $R_{v_2}$ reduces with $\ol{T}_k R_k$ (and has no reductions other than those indicated). We then arrive at a contradiction in the same way as in Case~\textbf{IB3}. \\

\noindent\textbf{Case~IID.} $j < v_1 < v_2$. 

\[
\dots \ol{T}_i \tm{MA}R_i T_i \dots \ol{T}_j R\tm{MB}_j T_j \dots \ol{T}_{v_1} \tm{MC}R_{v_1} T_{v_1} \dots \ol{T}_{v_2} R\tm{MD}_{v_2} T_{v_2} \dots \ol{T}_k R\tm{ME}_k T_k\dots
\dla{1.6}{}{MA}{ME}
\dla{0.8}{}{MA}{MB}
\dla{1.0}{}{MC}{MD}
\]
\vspace{0.5cm}\\ 

$R_{v_1}$ reduces with $R_i$, $R_{v_2}$ with $\ol{T}_kR_k$. The word $R_{v_1}$ reduces with $R_j$ and $R_{s_1}$, where $v_1 < s_1 < v_2$, and therefore $R_{v_1}$ does not reduce with $\ol{T}_kR_k$. We then arrive at a contradiction in the same way as in Case~\textbf{IIC}. \\

\noindent\textbf{Case~III.} When reducing $[i+1, k-1]$ there remains a full trace, i.e. there are numbers 
\[
i+1 = s_0 < z_1 < s_1 < z_2 < s_2 < \dots < z_{t-1} < s_{t-1} < z_t < s_t = k-1
\]
where $t > 3$, and the word $R_{s_0}$ reduces with the word $R_{s_1}$, the word $R_{s_1}$ reduces with the word $R_{s_2}$, $\dots$, and the word $R_{s_{t-1}}$ reduces with the word $R_{s_t}$; wherein $R_{s_v} \: (v = 0, 1, 2, \dots, t)$ can only reduce with $R_{z_{v+1}}$, $\ol{T}_{s_{v+1}}$, $R_{z_v}$, and $T_{s_{v-1}}$; additionally, when reducing $[i+1, k-1]$ no major nor minor trace remains. 

Then there are $R_{s_{q-1}}$, $R_{s_q}$, $R_{s_{q+1}}$ such that the following reductions occur
\begin{equation}
\small
\dots \ol{T}_{s_q-1} \tm{MA}R_{s_q-1} T_{s_q-1} \dots \ol{T}_{z_q} \tm{MB}R_{\tm{MC}z_q} T_{z_q} \dots \ol{T}_{s_q} \tm{MD}R\tm{ME}_{s_q} T_{s_q} \dots \ol{T}_{z_q+1} \tm{MF}R\tm{MG}_{z_q+1} T_{z_q+1}\dots \ol{T}_{s_q+1} R\tm{MH}_{s_q+1} T_{s_q+1} \dots
\dla{1.8}{}{MA}{ME}
\dla{0.9}{}{MA}{MC}
\dla{0.9}{}{MC}{ME}
\dla{1.6}{}{MD}{MH}
\dla{0.8}{}{MD}{MG}
\dla{1.0}{}{MF}{MH}
\end{equation}\\

Indeed, $R_{s_0}$ must reduce with $\ol{T}_{s_1}$, as otherwise from $R_{s_0}$ (by statements $\fC_m, \fD_m, \fE_m$) there remains a major trace; and $R_{s_t}$ must reduce with $T_{s_{t-1}}$, as otherwise from $R_{s_t}$ there remains a major trace. 

Consider the leftmost $R_{s_\alpha}$ which reduces with $T_{s_{\alpha-1}}$, and the rightmost $R_{s_\beta}$ in $[s_0, s_{\alpha-1}]$ which reduces with $\ol{T}_{s_{\beta+1}}$. If $\beta-2 = \alpha$, then from $R_{s_{\beta+1}}$ there remains a major trace. If $\beta-\alpha = 3$, then from $R_{s_{\beta+1}}$ and $R_{s_{\beta+2}}$ there remain minor traces. If $\beta-\alpha > 3$, and from $R_{s_{\beta+1}}, R_{s_{\beta+2}}, \dots, R_{s_{\alpha-1}}$ neither major nor minor traces remain, then we have the reductions in $(3)$. 

As when reducing $[i, k]$ the full trace in question must be completely absorbed, we have that $R_{s_q}$ reduces either with $R_iT_i$, or with $\ol{T}_kR_k$, which contradicts either the dual of $\fC_m$, or $\fC_m$. This completes the proof of Lemma~4. 
\end{proof}

\begin{lemma}
$(\fA_{m+1}, \fB_m, \fC_m, \dots, \fH_m) \implies \fB_{m+1}$.
\end{lemma}
\begin{proof}
Suppose that for every way of reducing the word $W$, where $L[W] \leq m+1$, there are the following reductions
\[
\dots \ol{T}_i \tm{MA}R_i T_i \dots \ol{T}_j R\tm{MB}_j T_j \dots \ol{T\tm{MC}}_k R_k T_k \dots
\dla{0.6}{}{MA}{MB}
\dla{1.2}{}{MA}{MC}
\]
\vspace{0cm}\\

If when reducing $[i+1, k-1]$, the word $R_j$ reduces with more than one $R$, then we easily come to a contradiction by the same methods as in Lemma~4. 
\end{proof}

\begin{lemma}
$(\fA_{m+1}, \fB_{m+1}, \fC_{m+1}, \dots, \fH_m) \implies \fC_{m+1}$. 
\end{lemma}
\begin{proof}
Suppose that when reducing the word $W$, where $L[W] \leq m+1$, there are the following reductions 
\[
\dots \ol{T}_i \tm{MA}R_i T_i \dots \ol{T}_j R\tm{MB}_j T_j \dots \ol{T}_k R\tm{MC}_k T_k \dots \ol{T}_p R\tm{MD}_p T_p \dots
\dla{0.6}{}{MA}{MB}
\dla{1.2}{}{MA}{MC}
\dla{1.8}{}{MA}{MD}
\]
\vspace{0.3cm}\\

or the following reductions

\[
\dots \ol{T}_i \tm{MA}R_i T_i \dots \ol{T}_j R\tm{MB}_j T_j \dots \ol{T}_k R\tm{MC}_k T_k \dots \ol{T\tm{MD}}_p R_p T_p \dots
\dla{0.6}{}{MA}{MB}
\dla{1.2}{}{MA}{MC}
\dla{1.8}{}{MA}{MD}
\]
\vspace{0.4cm}\\

Then by $\fA_{m+1}$ and $\fB_{m+1}$, $R_j$ in $[i+1, p-1]$ cannot reduce with any more than two $R$, and to be completely reduced in $[i, p]$ must reduce with $R_iT_i$ by more than $\frac{7}{11}$ of the letters of its length, which is impossible by Property~6. 
\end{proof}

\begin{lemma}
$(\fA_{m+1}, \fB_{m+1}, \fC_{m+1}, \fD_m, \dots, \fH_m) \implies \fD_{m+1}$. 
\end{lemma}
\begin{proof}
Suppose that when reducing the word $W$, where $L[W] \leq m+1$, we have the following reduction.
\[
\dots \ol{T}_i \tm{MA}R_i T_i \dots \ol{T}_j R_j T\tm{MB}_j \dots
\dla{0.8}{}{MA}{MB}
\]
\vspace{0cm}\\
The word $R_j$ must be completely reduced on the left by the word $R_iT_i[i+1,j-1]\ol{T_j}$. The maximal reduction which $R_j$ can have is its reduction with $R_s, T_s$, and $R_p$, where $i \leq s < p < j$ (this follows from $\fC_{m+1}, \fD_m$, and $\fE_m$). But by Property~6 these three words do not absorb more than $\frac{17}{22}$ of the letters of the word $R_i$. 
\end{proof}

\begin{lemma}
$(\fA_{m+1}, \dots, \fD_{m+1}, \fE_m, \fF_m, \fG_m, \fH_m) \implies \fE_{m+1}$. 
\end{lemma}
\begin{proof}
Suppose that when reducing the word $W$, where $L[W] \leq m+1$, there is the following reduction
\[
\dots \ol{T}_i \tm{MA}R_i T_i \dots \ol{T\tm{MB}}_j R_j T_j \dots [k\tm{MC}, p] \dots
\dla{0.8}{}{MA}{MB}
\dla{1.2}{}{MA}{MC}
\]
\vspace{0cm}\\

But, as we saw in the proof of the statement $\fD_{m+1}$ , the word $R_j$ can reduce on the right by at most $\frac{17}{22}$ of its length, and besides this, $R_j$ can only reduce (on the left) by $R_i$. In this way, $R_j$ is absorbed completely when reducing $[i, p]$, and this reduction is impossible. 
\end{proof}

\begin{lemma}
$(\fA_{m+1}, \dots, \fE_{m+1}, \fF_m, \fG_m, \fH_m) \implies \fF_{m+1}$. 
\end{lemma}
\begin{proof}
Suppose that when reducing the word $W$, where $L[W] \leq m+1$, there are the following reductions
\[
\dots \ol{T}_i \tm{MA}R_i T_i \dots \ol{T}_j R\tm{MB}_j T_j \dots \ol{T}_k \tm{MC}R_k T_k \dots \ol{T}_p R\tm{MD}_p T_p \dots
\dla{0.8}{}{MA}{MB}
\dla{1.4}{}{MA}{MD}
\dla{0.8}{}{MC}{MD}
\]
\vspace{0cm}\\

\noindent\textbf{Case I.} When reducing $[i+1, p-1]$ there remains a major trace, i.e. there is some $R_v$, which reduces with a subword of $[i+1, p-1]$ with fewer letters than $\frac{9}{22}$ of its length.

In this case $R_v$, in order to be completely reduced in $[i, p]$, must reduce either with $R_{i}$, or with $R_p$; if $v = j$, then $R_v$ also reduces with $R_p$; if $v=k$, then $R_v$ reduces with $R_i$; i.e. in every case we obtain a contradiction to $\fC_{m+1}$ or the dual of $\fC_{m+1}$. \\

\noindent\textbf{Case II.} When reducing $[i+1, p-1]$ there remains a minor trace, i.e. there are $R_{v_1}$ and $R_{v_2}$, which reduce with one another, and such that each of these words loses fewer than $\frac{4}{11}$ of its letters. 

By Property~7 the cases $v_1 < v_2 < j$ and $k<v_1<v_2$ are impossible. By $\fC_{m+1}$ and the dual of $\fC_{m+1}$, the other cases are also impossible. \\

\noindent\textbf{Case III.} When reducing $[i+1, p-1]$ there remains a full trace. 

In Lemma~4, Case~\textbf{III}, it was in particular proved that if $R_i$ reduces with $R_p$, then when reducing $[i+1, p-1]$, there cannot remain a full trace. 
\end{proof}

\begin{lemma}
$(\fA_{m+1}, \dots, \fF_{m+1}, \fG_m, \fH_m) \implies \fG_{m+1}$. 
\end{lemma}
\begin{proof}
Suppose that we have the word $W$, where $L[W] \leq m+1$, and that the following reductions take place in $W$. 
\[
\dots \ol{T}_i \tm{MA}R_i T_i \dots \tm{MC}\ol{T}_j R\tm{MB}_j T_j\tm{MD} \dots
\dla{0.8}{}{MA}{MB}
\dla{0.4}{}{MC}{MD}
\]
\vspace{0cm}

\noindent\textbf{Case I.} When reducing $[i+1, j-1]$ there remains a major trace. 

In this case, $R_v$ reduces with $T_i$ and $\ol{T}_j$. If $v = p$, then statement $\fG_{m+1}$ follows directly from Property~9. If $v \neq p$, then, as by Property~9 we have $\len{T_v} < \len{T_i}$ and $\len{T_v} < \len{T_j}$, statement $\fG_{m+1}$ follows from statement $\fG_m$ in the word $[i, v]$ or the word $[v, j]$. \\

\noindent\textbf{Case II.} When reducing $[i+1, j-1]$ there remains a minor trace. 

If $R_i$ reduces with some $R_t$, where $i < t < j$, then by the proof of $\fD_{m+1}$ we find that this case is actually impossible. 

If each of $R_{v_1}$ and $R_{v_2}$ reduces with only one of $T_i$ and $\ol{T}_j$, then by the proof of $\fD_{m+1}$ we find that this case is actually impossible. 

If $R_{v_1}$ and $R_{v_2}$ reduce with both $T_i$ and $\ol{T}_j$, then statement $\fG_{m+1}$ is proved as in Case~\textbf{I}. \\

\noindent\textbf{Case~III.} When reducing $[i+1, j-1]$ there remains a full trace. 

As a particular case of Lemma~4, Case~\textbf{III}, we have proved that if $R_i$ reduces with $\ol{T}_j$, then when reducing $[i+1, j-1]$ there cannot remain a full trace. 
\end{proof}

\begin{lemma}
$(\fA_{m+1}, \fB_{m+1}, \dots, \fG_{m+1}, \fH_m) \implies \fH_{m+1}$. 
\end{lemma}
\begin{proof}
Suppose that we have the word $W$, where $L[W] \leq m+1$, and consider some particular way of reducing the word $W$. 

Consider all reductions in which $R_z$ reduces with $R_s$, where $1 \leq z < s \leq w$, and choose the maximal of them, i.e. the reductions in which no subword of the word $[1, z-1]\ol{T}_zR_z$ does not reduce with any subword of the word $R_sT_s[s+1, w]$ (except when reducing $R_z$ and $R_s$). 

In the same way, choose all maximal reductions of $R$ with $\ol{T}$ and $T$ with $R$. 

Let $R_{i_0}, R_{i_1}, \dots, R_{i_p}$ are all the $R_i$ appearing in $W$ which for the chosen reduction is affected by the maximal reduction of an $R$ with an $R$. 

If $i_0 > 1$, then $R_1$ either does not reduce, or the maximal reduction with $\ol{T}_k$, where $1 < k \leq w$, then the word $[k, w]$ has (by statement $\fH_m$) a trace, which will also be a trace for $W$. 

We will also prove that either $i_p = w$, or else $\fG_{m+1}$ holds. 

If some $R_{i_q}$ $(q = 1, 2, \dots, p-1)$ is affected in the maximal reduction of $R_z$ with $R_s$ only on one side (e.g. the left), then either $R_{i_q}$ does not reduce with it on the right -- in which case the word $[i_q+1, w]$ will be the trace of $W$; or else $R_{i_q}$ reduces maximally with $\ol{T}_k$ -- in which case  the word $[k, w]$ will be the trace of $W$. 

It remains only to prove Lemma~11 in the case of the following reductions taking place in the word $W$. 
\[
\ol{T}_1 \tm{MA}R_1 T_1 \dots \ol{T}_{i_1} \tm{MB}R\tm{MC}_{i_1} T_{i_1} \dots \ol{T}_{i_2} \tm{MD}R\tm{ME}_{i_2} T_{i_2} \dots \quad \dots \ol{T}_{i_{p-1}} \tm{MF}R\tm{MG}_{i_{p-1}} T_{i_{p-1}} \dots \ol{T}_{i_p} R\tm{MH}_{i_p} T_{i_p}
\dla{1.2}{}{MA}{MC}
\dla{1.2}{}{MB}{ME}
\dla{2.0}{}{MD}{MG}
\dla{1.2}{}{MF}{MH}
\]
\vspace{0.3cm}\\

where the indicated arcs are all maximal reductions of an $R$ with an $R$. \\

\noindent\textbf{Case~I.} $p = 1$.

In this case, either $R_1$ reduces with $\ol{T}_w$ -- in which case $R_w$ reduces only with $R_1$, or else $R_1$, in the given reduction, loses fewer than $\frac{4}{11}$ of its letters (by statements $\fC_{m+1}, \fD_{m+1}, \fE_{m+1}$). \\

\noindent\textbf{Case~II.} $p = 2, 3$. 

Either the word $R_1$ reduces with $\ol{T}_{i_1}$, or else there remains a major trace from $R_1$. Either the word $T_{i_{p-1}}$ reduces with $R_w$, or else there remains a major trace from $R_w$. But if $R_1$ reduces with $\ol{T}_{i_1}$, and $T_{i_{p-1}}$ with $R_w$, then either there remains a major trace from $R_{i_1}$, or there remains a major trace from $R_{i_2}$, or from $R_{i_1}$ and $R_{i_2}$ there remains a minor trace, depending on what $R_{i_1}$ is reduced with on the right and what $R_{i_2}$ is reduced with on the left. \\

\noindent\textbf{Case~III.} $p > 3$.

In this case, there remains a full trace (by the definition of full traces). 
\end{proof}

\begin{theorem}
If the non-empty reduced word $Q$ is equal to the identity in the group $G$, then $Q$ contains some trace as a subword. 
\end{theorem}
\begin{proof}
By Theorem~1, if the non-empty word $Q$ is equal to the identity, then there exists a word of the form $\prod_{i=1}^w \ol{T}_i R_i T_i$ which satisfies Properties~1-10, and such that $Q \equiv \prod_{i=1}^w \ol{T}_i R_i T_i$. 

If the word $Q$ is reduced, then after completely reducing the word $\prod_{i=1}^w \ol{T}_i R_i T_i$ in any way, the word $Q$ remains. Obviously, the statements $\fA_1$, $\fB_1, \dots$, $\fH_1$ are all true. Then by Lemmas~4-11, properties $\fA_q$, $\fB_q, \dots$, $\fH_q$ all hold for all natural numbers $q$, and in particular statement $\fH_w$ is true. 
\end{proof}

\begin{lemma}
Let $Q \equiv \prod_{i=1}^w \ol{T}_i R_i T_i$. If when reducing the product $\prod_{i=1}^w \ol{T}_i R_i T_i$ (satisfying Properties~1-10), the words $R_1$ and $R_k$ (or the words $R_1$ and $\ol{T}_k$) reduce, then $\len{T_1} + 2 \geq k$. 
\end{lemma}
\begin{proof}
We prove the lemma by induction on $k$.

If $k=2$, then the lemma is trivially true, as $\len{T_1} \geq 0$. 

Suppose that the lemma is true for all $k' < k$. \\

\noindent\textbf{Case~I.} Suppose that the words $R_1$ and $R_k$ reduce. \\

\noindent\textbf{Case~IA.} Neither $R_1$ nor $R_k$ reduce with any $R_p$, where $1 < p < k$. By statement $\fG_w$ and Property~10, we have $k=2$, and the lemma is trivially true. \\

\noindent\textbf{Case~IB.} $R_1$ reduces with $R_p$, where $1 < p < k$. 
\[
\ol{T}_1 \tm{MA}R_1 \tm{MB}T_1 \dots \ol{T}_p \tm{MC}R\tm{MD}_p T_p \dots \ol{T\tm{ME}}_k R\tm{MF}_k T_k \dots
\dla{1.0}{}{MA}{MD}
\dla{1.6}{}{MA}{MF}
\dla{0.8}{}{MB}{MD}
\dla{0.8}{}{MC}{ME}
\]
\vspace{0.2cm}\\

The word $R_p$ is completely reduced when reducing $[1, k]$, and hence $R_p$ reduces with $T_1$ and $\ol{T}_k$. \\

\noindent\textbf{Case~IB1.} $R_p$ does not reduce when reducing $[2, k-1]$. Then by $\fG_w$ and Property~10, we have that $k=3$, $\len{T_1} > \len{T_p} \geq 0$, from which we have $\len{T_1} + 2 \geq k$. \\

\noindent\textbf{Case~IB2.} $R_p$ reduces with $R_q$, where $1 < q <p$. 
\[
\ol{T}_1 \tm{MA}R_1 T\tm{MB}_1 \dots \ol{T}_q \tm{MC}R_q T_q \dots \ol{T}_p \tm{MD}R\tm{ME}_p T_p \dots \ol{T\tm{MF}}_k R\tm{MG}_k T_k \dots
\dla{1.6}{}{MA}{MG}
\dla{1.0}{}{MA}{ME}
\dla{0.8}{}{MB}{ME}
\dla{0.6}{}{MC}{ME}
\dla{0.6}{}{MD}{MF}
\]\\

Then by the inductive hypothesis (of the lemma dual to Lemma~12), we have that $p \leq \len{T_p} + 2$. By $\fG_w$ and Properties~9 and 10, we have $p+1 = k$ and $\len{T_p} < \len{T_1}$, from which we have 
\[
\len{T_1} + 2 \geq \len{T_p} + 3 \geq p+1 = k. \\
\]

\noindent\textbf{Case~IB3.} $R_p$ reduces with $R_q$, where $p < q < k$. 
\[
\ol{T}_1 \tm{MA}R_1 \tm{MB}T_1 \dots \ol{T}_p \tm{MC}R\tm{MD}_p T_p \dots \ol{T}_q R\tm{ME}_q T_q \dots \ol{T\tm{MF}}_k R\tm{MG}_k T_k \dots
\dla{1.6}{}{MA}{MG}
\dla{0.9}{}{MA}{MD}
\dla{0.6}{}{MB}{MD}
\dla{0.6}{}{MC}{ME}
\dla{1.2}{}{MC}{MF}
\]
\vspace{0.2cm}\\

By the inductive hypothesis $\len{T_p} + 2 \geq k-p+1$. By $\fG_w$ and Properties~9 and 10, we have that $p=2$ and $\len{T_p} < \len{T_1}$, from which we have
\[
\len{T_1} + 2 \geq \len{T_p} + 3 \geq k-2 + 1 + 1 \geq k.\\
\]

\noindent\textbf{Case~IC.} $R_k$ reduces with $R_p$, where $1 < p < k$. The proof is identical to Case~\textbf{IB.} \\

\noindent\textbf{Case~II.} Suppose that $R_1$ reduces with $\ol{T}_k$. The proof is identical to Case~\textbf{I.} \\

This completes the proof of Lemma~12. 
\end{proof}

\begin{lemma}
Let $Q \equiv \prod_{i=1}^w \ol{T}_i R_i T_i$. If when reducing the product $\prod_{i=1}^w \ol{T}_i R_i T_i$ (satisfying Properties~1-10), the word $R_1$ reduces with $R_k$, then $\len{Q} \geq k$. 
\end{lemma}
\begin{proof}
By statements $\fC_w, \fD_w$, and $\fE_w$, the word $R_1$ can reduce only with $R_t, \ol{T}_t$, and $R_s$, where $s < t$. But these reductions remove fewer than $\frac{17}{22}$ of the number of letters of $R_1$. The same is true for $R_w$. It follows that when reducing $\prod_{i=1}^w \ol{T}_i R_i T_i$, the word $\ol{T}_1$ does not reduce at all, and the words $R_1$ and $R_w$ are not completely absorbed. Hence $\len{Q} \geq \len{T_1} + 2$, and by Lemma~12, we also have $\len{T_1}+2 \geq k$. 
\end{proof}

\begin{lemma}
Let $Q \equiv \prod_{i=1}^w \ol{T}_i R_i T_i$. If when reducing the product $\prod_{i=1}^w \ol{T}_i R_i T_i$ (satisfying Properties~1-10) there remains some full trace of the words $R_1, R_{s_1}, R_{s_2}$, $\dots$, $R_{s_{p-1}}, R_w$, then $p \leq 3\len{Q}$. 
\end{lemma}
\begin{proof}
As we already know, the words $R_1$ and $R_w$ are not completely absorbed when reducing $\prod_{i=1}^w \ol{T}_i R_i T_i$. To prove the lemma, it suffices to show that none of $R_{s_{m-1}}, R_{s_m}, R_{s_{m+1}}$, $(m = 1, 2, \dots, p-1)$ are not all simultaneously completely absorbed when reducing $\prod_{i=1}^w \ol{T}_i R_i T_i$. 
\tiny
\[
\dots \ol{T}_{s_{m-2}} R\tm{MA}_{s_{m-2}} T_{s_{m-2}} \dots \ol{T}_{s_{m-1}} \tm{MB}R\tm{MC}_{s_{m-1}} T_{s_{m-1}} \dots \ol{T}_{s_{m}} \tm{MD}R\tm{ME}_{s_{m}} T_{s_{m}} \dots \ol{T}_{s_{m+1}} \tm{MF}R\tm{MG}_{s_{m+1}} T_{s_{m+1}} \dots \ol{T}_{s_{m+2}} \tm{MH}R\tm{MI}_{s_{m+2}} T_{s_{m+2}} \dots
\dla{0.8}{}{MA}{MC}
\dla{0.8}{}{MB}{ME}
\dla{0.8}{}{MD}{MG}
\dla{0.8}{}{MF}{MI}
\]
\normalsize
\vspace{0.2cm}\\
If $R_{s_{m-1}}$ reduces on the right with the same subword as $R_{s_m}$ reduces with, and $R_{s_{m+1}}$ reduces on the left with the same subword as $R_{s_m}$, then $R_{s_m}$, when reduced, cannot lose more than $\frac{8}{11}$ of its letters, i.e. is not completely absorbed. Hence $R_{s_{m-1}}$ reduces on the right only with $R_{s}$, but on the left (by $\fC_w, \fD_w, fE_w$) it cannot lose fewer than $\frac{17}{22}$ of its letters, i.e. $R_{s_{m-1}}$ is not completely absorbed. 
\end{proof}

\begin{lemma}
Let $Q \equiv \prod_{i=1}^w \ol{T}_i R_i T_i$. If when reducing the product $\prod_{i=1}^w \ol{T}_i R_i T_i$ (satisfying Properties~1-10) the word $R_1$ and $R_k$ reduce, then $\len{T_k} \leq (d+1)\len{T_1}$, where $d$ is the length of the longest defining word. 
\end{lemma}
\begin{proof}
By statement $\fG_w$ and Property~10, each word $R_p$, where $1 < p <k$, loses with $\ol{T}_k$ fewer than $d$ letters, after which, there remain of $\ol{T}_k$ fewer $\len{T_1}$. But by Lemma~12, we have that $k \leq \len{T_1}+2$, i.e. $k - 2 \leq \len{T_1}$. It follows that 
\[
\len{T_k} \leq d \cdot (k-2) + \len{T_1} \leq d\cdot \len{T_1} + \len{T-1} = (d+1)\len{T_1}.
\]
\end{proof}

\begin{lemma}
Let $Q \equiv \prod_{i=1}^w \ol{T}_i R_i T_i$, where $\len{Q} = e$. If when reducing the product $\prod_{i=1}^w \ol{T}_i R_i T_i$ (satisfying Properties~1-10) there remains a full trace of the words $R_1, R_{s_1}, R_{s_2}, \dots, R_{s_{p-1}}, R_w$, then $\len{T_j} \leq e \cdot(d+1)^{3e}$, where $j = 1, 2, \dots, w$. 
\end{lemma}
\begin{proof}
We first prove the lemma for the words $R_1, R_{s_1}, R_{s_1}, \dots, R_{s_{p-1}}, R_w$. 

By Lemma~15, $\len{T_{s_1}} \leq (d+1)e$, as $\len{T_1} < \len{Q} = e$. We have $\len{T_{s_2}} \leq (d+1)e^2$, $\dots$, $\len{T_{s_p}} \leq (d+1)e^p$. But by Lemma~14, $p \leq 3e$. By statement $\fG_w$, for any $u$ such that $s_{t-1} < u < s_t$, we have $\len{T_u} \leq \len{T_{s_t}}$. This proves the lemma.
\end{proof}

\begin{lemma}
Let $Q \equiv \prod_{i=1}^w \ol{T}_i R_i T_i$, let $\len{Q} = e$, and let $d$ be the length of the longest defining word. If when reducing the product $\prod_{i=1}^w \ol{T}_i R_i T_i$ (satisfying Properties~1-10) there remains a full trace, then 
\[
\Big[\normalsize\prod_{i=1}^w \ol{T}_i R_i T_i {}^{{}^\partial} \leq 27e^3(d+1)^{6e}.
\]
(This upper bound is very rough, and was only chosen for ease of writing).
\end{lemma}
\begin{proof}
Indeed, we have $\len{\ol{T}_i R_i T_i} \leq 2e(d+1)^{3e}+d$ by Lemma~16, and that $w \leq 3e\{ 2 + e(d+1)^{3e} \}$ by Lemmas~14, 12, and 16. From this we have
\begin{align*}
\Big[\normalsize\prod_{i=1}^w \ol{T}_i R_i T_i {}^{{}^\partial} &\leq \{ 2e(d+1)^{3e} + d\} 3e \{ 2+ e\{ d+1 \}^{3e} \} \\
&\leq \{ 2e(d+1)^{3e}+d\}^2 \cdot 3e \\
&\leq \{ 3e(d+1)^{3e}\}^2 \cdot 3e \leq 27e^3(d+1)^{6e}.
\end{align*}
\end{proof}

\begin{theorem}
There exists an algorithm for solving the identity problem in every group from the class $K_{2/11}$. 
\end{theorem}
\begin{proof}
To prove the theorem we use a generalisation of M. Dehn's algorithm \cite{10}, which is as follows. 

For a given word $Q$, we apply as far as possible the following three operations:
\begin{enumerate}[label=\greek*)]
\item Free reduction.
\item Replacing $S$ by $T$, if $S\ol{T}$ is a defining word of the group $G$, and $\len{S} > \len{T}$. 
\item Replacing $S_1S_2$ by $T_1T_2$, if $S_1X\ol{T}_1$ and $\ol{X}S_2\ol{T}_2$ is a defining word of the group $G$ and $\len{S_1S_2} > \len{T_1T_2}$. 
\end{enumerate}
If we have $Q = 1$ in $G$, then either this process terminates in the empty word, or it ends in some word $Q_\varepsilon$, where $\len{Q_\varepsilon} = e$, such that the length of the presentation of $Q_\varepsilon$ in the form $\prod_{i=1}^w \ol{T}_i R_i T_i {}^{{}^\partial}$ is no greater than $27e^3(d+1)^{6e}$ letters, where $d$ is the length of the longest defining word of $G$. 

By Theorem~2 we conclude that there is an algorithm which solves the word problem for every group $G$ in the class $K_{2/11}$.
\end{proof}

\bibliography{makbib} 
\bibliographystyle{amsalpha}

\end{document}